\documentclass[12pt]{amsart}

\usepackage{amsmath,amssymb,amscd,amsthm,indentfirst}
\usepackage{amsfonts}
\usepackage{enumitem}
\usepackage[T1]{fontenc}
\usepackage{hyperref}
\usepackage[all]{xy}
\usepackage{stmaryrd}

\usepackage{xcolor}
\usepackage{xfrac}
\usepackage{graphicx}
\usepackage{color}

\usepackage[colorinlistoftodos,prependcaption]{todonotes}

\usepackage{wasysym}
\usepackage{tikz}

\usepackage{ifpdf,ifdraft}
\ifdraft{

\newcommand{\assafsays}[1]{\marginpar{\textbf{A:} #1}}
}
{
\newcommand{\assafsays}[1]{}
}

\newtheorem{prop}{Proposition}[section]
\newtheorem{theorem}[prop]{Theorem}
\newtheorem{cor}[prop]{Corollary}

\newtheorem{lemma}[prop]{Lemma}
\newtheorem*{claim*}{Claim}

\theoremstyle{definition}
\newtheorem{defn}[prop]{Definition}
\newtheorem{example}[prop]{Example}
\newtheorem{examplex}{Example}
\newtheorem{void}[prop]{}
\newtheorem{notation}[prop]{Notation}

\newtheorem{remark}[prop]{Remark}

\newtheorem*{assumption*}{Assumption}
\newtheorem*{example*}{Example}

\def\vp{\varphi}

\def\NN{\ensuremath{\mathbb{N}}}

\def\QQ{\ensuremath{\mathbb{Q}}}
\def\RR{\ensuremath{\mathbb{R}}}
\def\ZZ{\ensuremath{\mathbb{Z}}}

\def\End{\operatorname{End}}

\def\id{\operatorname{id}}
\def\incl{\operatorname{incl}}

\def\nsg{\trianglelefteq}

\def\res{\operatorname{res}}

\newcommand{\B}[1]{\mathbf{#1}}
\newcommand{\C}[1]{\mathcal{#1}}

\newcommand{\xto}[1]{\xrightarrow{#1}}
\newcommand{\OP}[1]{\operatorname{#1}}

\newcommand{\trc}[1]{\{ #1 \}}   
\def\omegalim{\omega\operatorname{-}\lim}

\def\can{\text{can}}
\def\Cone{\OP{Cone}}
\def\Pure{\mathfrak{Pure}}
\def\Free{F}

\def\D{\mathcal{D}}
\def\tC{\hat{\mathcal{C}}}
\def\tD{\hat{\D}}
\newcommand{\Myfloor}[1]{\lfloor #1 \rfloor}
\newcommand{\myfloor}[1]{\llfloor #1 \rrfloor}
\newcommand{\myinterval}[1]{[ #1 \rangle}

\title[Asymptotic cones]{Directional asymptotic cones of groups equipped with bi-invariant metrics}

\author{Jarek K\k{e}dra}
\author{Assaf Libman}
\address{JK: Institute of Mathematics, University of Aberdeen, Fraser Noble Building, Aberdeen AB24 3UE, U.K. and
University of Szczecin}
\email{kedra@abdn.ac.uk}
\address{AL: Institute of Mathematics, University of Aberdeen, Fraser Noble Building, Aberdeen AB24 3UE, U.K.}
\email{a.libman@abdn.ac.uk}

\begin{document}

\maketitle

\begin{abstract}
Given a bi-invariant metric on a group, we construct a version of an asymptotic
cone without using ultrafilters.  The new construction, called the {\it
directional asymptotic cone}, is a contractible topological group equipped with
a complete bi-invariant metric and admits a canonical Lipschitz homomorphism to
the standard asymptotic cone. Moreover, the directional asymptotic cone of a
countable group is separable.
\end{abstract}

\setcounter{tocdepth}{1}
\tableofcontents


\section{Introduction}

\subsection{Historical context}
In 1984 van den Dries and Wilkie \cite{MR751150} published an alternative proof
of Gromov's celebrated theorem on polynomial growth. Their proof uses methods
of non-standard analysis and provides the first construction that uses
ultrafilters of what we know now as the asymptotic cone of a group equipped
with the word metric. The construction generalises to all metric spaces and was
popularised by Gromov in \cite{MR1253544}, where he provided computations for a
substantial family of examples and posed many inspiring questions.  Asymptotic
cones became a well established tool of metric geometry and geometric group
theory \cite{MR3753580}. It is an unpleasant feature, however, that the asymptotic cone of a
space depends on the choice of a non-principal ultrafilter $\omega$ and a scaling sequence $\B s=(s_n)$, and it is often an unapproachable space.

Asymptotic cones were initially constructed to study finitely generated groups $G$ equipped with the word metric with respect to a set of generators.
In recent years, there has been growing interest in word metrics on $G$ that are both left and right invariant, see Paragraph~\ref{par:conjugation-invariant norms} below.
In this case, the asymptotic cone $\Cone_\omega(G,\B s)$ is a group equipped with a bi-invariant metric, rendering it a topological group. 
This paper grew out of the interest in these groups which, a-priori, may depend on both the choice of the ultrafilter $\omega$ and the scaling sequence $\B s$.
Our ambition was to come up with an analogous construction of the asymptotic cone which does not require a choice of an ultrafilter.

\subsection{Asymptotic cones}
Intuitively, the asymptotic cone of a metric group~$(G,d)$, or more generally of a metric space, is a construction that aims to capture ``its image from far away''.
For example, the asymptotic cone of a bounded group (or space) is a point, that of $\ZZ$ is  $\RR$ and that of $\ZZ^n$ is $\RR^n$ (both with the $L^1$-metric).

The construction of the asymptotic cone of a metric space $(X,d)$ requires to choose a {\em non-principal ultrafilter} $\omega$ on $\NN$ and a {\em scaling sequence} $\B s=(s_n)$, i.e., a sequence $s_n>0$ with $\lim_n s_n = \infty$.
Informally, $\B s$ tells us ``how fast'' we move away from $X$ by scaling the metric on $X$ to $\tfrac{1}{s_n} \cdot d$, divide distances by $s_n$ for each $n \geq 1$.
We then model the elements of the image of $X$ from infinity as sequences $(x_n)_{n \geq 1}$ of elements in $(X, \tfrac{1}{s_n}d)$ whose distance from some fixed basepoint $y \in X$ is bounded, i.e., $\limsup_n \tfrac{d(x_n,y)}{s_n} < \infty$.
We now need to define a suitable metric on the set of these sequences.
This is the role of the ultrafilter $\omega$.
Recall that an ultrafilter $\omega$ allows us to choose a unique convergent subsequence in any bounded sequence $(r_n)$ of real numbers, denoted $\omegalim r_n$.
Given two such sequences $(x_n)$ and~$(x_n')$ we define $\delta((x_n), (x_n')) = \omegalim_n \tfrac{d(x_n,x_n')}{s_n}$.
This is only a pseudo-metric.
The asymptotic cone $\Cone_\omega(X,\B s)$ is the metric space which is the set of equivalence classes of sequences at pseudo-distance $0$ with the metric $d_\omega$ induced from $\delta$.
The equivalence classes are denoted $\{ x_n \}$.

The construction  $\Cone_\omega(G,\B s)$ depends heavily on the choice of the ultrafilter.
Indeed, Drutu and Sapir \cite{MR2153979} constructed a finitely generated group $G$ which has infinitely many non-homeomorphic asymptotic cones depending on the choice of $\omega$.
See also \cite[page 189]{MR3753580} for a more detailed context. 

Even though $G$ is a group, in general $\Cone_\omega(G,\B s)$ has no natural group structure.
If the metric on $G$ is invariant to both left and right translations then $\Cone_\omega(G,\B s)$ is endowed with a natural group structure $\{g_n\} \cdot \{g_n'\} = \{ g_n g_n'\}$.
This fact was first observed by Calegari and Zhuang \cite{MR2866929}.
A Lipschitz homomorphism $\varphi \colon G \to G'$ gives rise to a Lipschitz homomorphism $\Cone_\omega(G,\B s) \xto{ \{g_n\}_n \mapsto \{ \varphi(g_n)\}_n } \Cone_\omega(G',\B s)$.

Equipping a group $G$ with a bi-invariant metric is equivalent to equipping it with a {\em conjugation-invariant norm}, see paragraph \ref{par:conjugation-invariant norms} below.
Any $A \subseteq G$ which normally generates $G$ gives rise to a conjugation-invariant {\em word} norm where for any $g \in G$ the value of $\| g\|_A$ is the minimum length of a word in the conjugates of the elements of $A$ and their inverses needed to express $g$.
See Definition \ref{defn:standard word norm}.

{\em In this paper, we will only be interested in groups equipped with conjugation-invariant norms} and their asymptotic cones.
At present we do not know if for such groups the asymptotic cone does or does not depend on the choice of the ultrafilter.
Note that the word metric on Drutu and Sapir's group $G$ we mentioned above is not bi-invariant.

\subsection{Directional asymptotic cones}
Let us dwell on the isometry $\iota \colon \Cone_\omega(\ZZ,\B s) \to
\RR$ for a non-principal ultrafilter $\omega$ and the scaling
sequence~$s_n=2^n$.  It sends the equivalence class $\{ z_n\} \in
\Cone_\omega(\ZZ,\B s)$ to $\omegalim_n \tfrac{z_n}{2^n} \in \RR$.
Let $r\in \RR$ be a positive real number expressed as
\begin{equation}
r = k_0 + \sum_{i=1}^{\infty} \frac{k_i}{2^i},
\label{Eq:cauchy-r}
\end{equation}
where $k_0\in \B N$ and $k_i \in \{-1,0,1\}$ for $i=1,2,\ldots$.
Let $w_n = \sum_{i=0}^{n} 2^{n-i}k_i$.
Then the sequence $(w_n/2^n)$ is Cauchy and $\lim_{n \to \infty} \tfrac{w_n}{2^n} = r$.
So the isomorphism $\Cone_\omega(\ZZ,\B s) \xto{\iota} \RR$ has the effect  $\iota(\{w_n\})=r$ regardless of the choice of $\omega$.
On the other hand, look at the sequence $z_n := (-1)^n w_n$.
Clearly 
$(z_n)$ defines an element in $\Cone_\omega(\ZZ,\B s)$ for any $\omega$.
However, $\iota(\{z_n\})=r$ or $\iota(\{z_n\})=-r$ depending on the choice of $\omega$, since $\pm r$ are the accumulation points of $\left\{ \tfrac{z_n}{2^n}\right\}$.
The difference between the two sequences is clear.
The first, $(w_n)$,  has a well defined ``direction'' in $\ZZ$,
due to $(w_n/2^n)$ being Cauchy,
and its ``role'' in the asymptotic cone does not depend on $\omega$.
The second, $(z_n)$, lacks ``direction'', its ``role'' in the asymptotic cone is ambiguous, a deficiency that is circumvented by the ultrafilter $\omega$.

Informally, the aspiration of this paper is to isolate from $\Cone_\omega(G,\B
s)$ those sequences that have a well defined ``direction'' independent of the
ultrafilter.  The result is the  {\em directional asymptotic cone} of a group
$G$, denoted $\tC(G)$.  In the example $\Cone_\omega(\ZZ,\B s) \cong \RR$
above, the sequences with no well defined direction did not add any new
elements to the asymptotic cone.  But this is not the case in general, as we
show in Example \ref{example:intro F2 is big} below.  In general, there are a
lot of ``non-directional elements'' in $\Cone_\omega(G,\B s)$.
We emphasise that, unlike the  asymptotic cone which is defined for any metric space, the directional asymptotic cone is specific to groups with bi-invariant metrics
and the group structure is essential.

Let us give an intuitive discussion of the construction.
We are looking to generalise the concept of a Cauchy sequence to sequences $(g_n)$ representing $\{g_n\} \in \Cone_\omega(G,\B s)$ for arbitrary groups $G$ equipped with a conjugation-invariant norm $\| \ \|$.
Our intuition is based on the following analogy. Suppose $G$ is a 
Riemannian manifold and consider the exponential map ${\rm exp}\colon T_1G\to G$.
The paths $t\mapsto {\rm exp}(tv)$ for some $v\in T_1G$ are paths in $G$
which have {\it direction}. We make the analogy precise as follows.

Let $(G,\|\ \|)$ be a group equipped with a conjugation-invariant norm.
Let $F_{\frac{1}{2^n}\B Z}(G)\cong F(G)$ be the free group generated by the set
$G\setminus\{1\}$ and equipped with the maximal conjugation-invariant norm $\|\ \|_{n}$ for
which $\|g\|_{n} = \frac{1}{2^n}\|g\|$. It is useful to think about the Cayley
graph of $F(G)$ in which an edge corresponding to a conjugate of $g$ (conjugation in $F(G)$) has length
$\frac{1}{2^n}\|g\|$.  For every $n\in \B N$ there is an isometric embedding 
$\psi_2\colon F_{\frac{1}{2^n}\B Z}(G) \to F_{\frac{1}{2^{n+1}}\B Z}(G)$ given
by 
$$
\psi_2(g) = g^2\qquad\text{(the square of $g$ in $F(G)$, not in $G$)}.
$$
It induces the sequence of isometric embeddings
$$
F(G) \stackrel{\psi_2}\longrightarrow F_{\frac{1}{2}\B Z}(G)
\stackrel{\psi_2}\longrightarrow F_{\frac{1}{4}\B Z}(G)
\stackrel{\psi_2}\longrightarrow F_{\frac{1}{8}\B Z}(G)\longrightarrow\cdots
$$
whose colimit is denoted by $F_{\B Z[1/2]}(G)$. It is a free product of the
abelian groups $\B Z[1/2]$ generated by the set $G\setminus\{1\}$ and is equipped with the
natural bi-invariant word metric.  A sequence $\gamma_n \in
\Free_{\tfrac{1}{2^{n-1}}\ZZ}(G) \cong \Free(G)$ is called an {\em asymptotic
direction} if its image in $\Free_{\ZZ[1/2]}(G)$ is a Cauchy sequence.
The image of such sequence $(\gamma_n)$ under the canonical homomorphism 
$\pi\colon F(G) \to G$ which sends every generator $g \in F(G)$ to $g \in G$
(this is where the group structure on $G$ is essential!) gives a sequence of
elements $g_n=\pi(\gamma_n) \in G$,
which represents an element in $\Cone_\omega(G,\B s)$, where $s_n=2^{n-1}$.  We
call such elements {\em Cauchy sequences in $G$} relative to the scaling
sequence $s_n=2^{n-1}$.  
Such Cauchy sequences $(g_n)$ comprise the {\em directional
asymptotic cone} of $G$, which we denote by $\tC(G)$.

To illustrate the above idea consider the example $\Cone_{\omega}(\ZZ,\B s)$
and \eqref{Eq:cauchy-r}. Define $\gamma_n\in F_{\frac{1}{2^n}\ZZ}(\ZZ)$ recursively
by
$$
\gamma_0=k_0\quad \text{ and } \quad \gamma_n=\psi_2(\gamma_{n-1})*k_n,
$$
where each $k_i\in \ZZ$ is a generator of $F_{\frac{1}{2^i}\ZZ}(\ZZ)$.
The sequence $(\gamma_n)$ is Cauchy in $F_{\ZZ[1/2]}(\ZZ)$, because
$\|\gamma_{n+1}\gamma_n^{-1}\|=\|k_{n+1}\|\leq \frac{1}{2^{n+1}}$, where
the norm is in $F_{\ZZ[1/2]}(\ZZ)$. Moreover, $\pi(\gamma_n) = w_n$, which
is a sequence with a ``direction''.

One clear advantage of the construction is that it does not require the choice of an ultrafilter.
However, what about scaling sequences other than $s_n=2^{n-1}$?
In practice, the construction of the directional asymptotic cone depends on the choice of an unbounded subset $T$ of the interval $(0,\infty)$.
We call $T$ a {\em scaling set} and think of it as the home of all scaling sequences $s_n \in T$.
The directional asymptotic cone with respect to a scaling set $T$, denoted $\tC(G,T)$, is defined as a certain quotient of the group $\prod_{t \in T} G$ with a suitable metric.
The details of the construction of the directional asymptotic cone are in Sections \ref{sec:FSA}-- \ref{Sec:diecrional asymptotic cone}.

Among all scaling sets, the maximal one $(0,\infty)$ is a natural choice and we write $\tC(G)$ for $\tC(G,(0,\infty))$.
We say that $G$ is {\em independent of scaling} if all the groups $\tC(G,T)$ are isometric,  see Definition \ref{def:scaling independence}.
In this case $\tC(G)$ has particularly nice properties in relation to the ultrafilter asymptotic cones. 
Examples of groups that are independent of scaling include all nilpotent groups (Proposition \ref{prop:cone of general nilpotent groups}) and all free groups equipped with the conjugation-invariant word metric (Theorem \ref{theorem:F(A) independent of scaling}).
We do not know of any group that is not independent of scaling.

\subsection*{Relation to the ultrafilter cone}
The following theorem combines Proposition \ref{P:directional into asymptotic}, Theorem \ref{T:equivalence scaling independence and injective rho} and Theorem \ref{T:abelian independent of scaling}.

\begin{theorem}
Let $G$ be  equipped with a conjugation-invariant norm.
\begin{enumerate}[label=(\alph*)]
\item
For any non-principal ultrafilter $\omega$ and scaling sequence $\B s$ there is a natural homomorphism 
\[
\rho_{\omega,\B s} \colon \tC(G) \to \Cone_\omega(G,\B s)
\]
which is Lipschitz with constant $1$. 

\item
 $\rho_{\omega,\B s}$ is injective if and only if $G$ is independent of scaling.
\end{enumerate}
\end{theorem}

\begin{example}
In Proposition \ref{prop:cone of general nilpotent groups} and Theorem \ref{T:abelian independent of scaling} we show that if $G$ is nilpotent, or more generally if the commutator subgroup of any finitely generated subgroup of $G$ is bounded, then $G$ is independent of scaling and that, in fact, $\rho_{\omega,\B s} \colon \tC(G) \to \Cone_\omega(G,\B s)$ is an isometric embedding.
In Proposition \ref{prop:tCZn} we show that if $G=\ZZ^n$ with the $L^1$-norm then it is an isometry.
\end{example}

\begin{example}
Let ${\Free}_n$ be the free group on $n$ generators equipped with the bi-invariant word metric associated with the standard generating set.
Let $\Theta \subseteq {\Free}_n$ be a subset consisting of representatives of conjugacy classes of elements that are not proper powers. 
Let $\Free_\RR(\Theta)$ denote the free product $\underset{\theta \in \Theta}{*} \RR$.
In Section \ref{S:coneFn} we show that $\Free_n$ is independent of scaling and that there is a homomorphism $\Free_\RR(\Theta) \to \tC(\Free_n)$ which is injective and has dense image. In particular, the directional cone $\tC(\Free_n)$ embeds
into every ultrafilter cone $\Cone_{\omega}(\Free_n,\B s)$.
\end{example}

\subsection{Algebraic and metric properties}

\subsubsection*{Functoriality}
In Propositions \ref{P:naturality C and tC} and \ref{P:tC is complete} we prove that $\tC(-)$ is a functor from the category of groups with bi-invariant metrics and Lipschitz homomorphisms to the category of complete groups with bi-invariant metrics and Lipschitz homomorphisms.

\subsubsection*{Metric properties.}
The directional asymptotic cone shares metric properties with the ultrafilter cone.
We prove the following result in Proposition \ref{P:tC is complete}) and Corollary \ref{C:tC length space}.

\begin{theorem}
Let $G$ be equipped with a conjugation-invariant norm and $T \subseteq (0,\infty)$ be a scaling set.
\begin{enumerate}[label=(\alph*)]
\item
$\tC(G,T)$ is a complete metric space.

\item
$\tC(G,T)$ is a length space.
\end{enumerate}
\end{theorem}

Unlike the ultrafilter cones, which tend to be wild metric spaces with very complicated fundamental groups, the directional asymptotic cone is quite tame.

\begin{theorem}
Let $G$ be equipped with a conjugation-invariant norm.
\begin{enumerate}[label=(\alph*)]
\item
If $G$ is countable then $\tC(G,T)$ is separable i.e., it is a Polish group for any scaling set $T$. See Proposition \ref{P:polish cone}.

\item
$\tC(G)$ is contractible.
The contraction is through a continuous family of Lipschitz homomorphisms $\varphi_t \colon \tC(G) \to \tC(G)$ indexed by $t \in [0,1]$.
See Theorem \ref{T:tC contractible} and Example \ref{example:ample sets}.
\end{enumerate}
\end{theorem}

The next example shows that the directional asymptotic cone is much ``smaller'' than the ultrafilter asymptotic cone.

\begin{example}\label{example:intro F2 is big}
Let $\B F_2$ denote the free group on two generators with the conjugation-invariant word norm.
By the theorem above $\tC(F_2)$ is separable.
In contrast, we claim that any ultrafilter asymptotic cone of $\B F_2$ 
is not a separable metric space.
To see this, recall that by \cite[Theorem 1.D]{MR3426433} any locally compact regular tree $T$ embeds isometrically into the Cayley graph of $\B F_2$ with respect to the bi-invariant word metric.
We will choose such $T$ of even degree $2n \geq 4$.
Clearly, $T$ is the Cayley graph of $\B F_n$.
Thus, we get an isometric embedding $i \colon \B F_n \to \B F_2$, where we note that $\B F_n$ has the ordinary word metric and $\B F_2$ has the bi-invariant word metric.
There results an isometric embedding $\Cone_\omega(\B F_n,\B s) \to \Cone_\omega(\B F_2,\B s)$.
Since $\B F_n$ is a non-elementary hyperbolic group, by \cite[Theorem 3.A.7]{MR1902363} $\Cone_\omega(\B F_n,\B s)$ is isometric to a non-empty real tree in which every point is a branching point of degree $\aleph$.
This space is clearly  not separable, and in particular $\Cone_\omega(\B F_2, \B s)$ cannot be separable 

In contrast to the contractibility of $\tC(G)$, Karlhofer shows in \cite[Proposition 7.2]{2203.10889} that the asymptotic cone of $G = \oplus_n \B Z/n\B Z$ equipped with the word metric associated with generators of the form $(0,\ldots,0,1,0,\ldots)$ is not simply connected.
We don't know of an example of a {\em finitely generated} group equipped with the bi-invariant word metric whose ultrafilter asymptotic cone is not contractible. 
\end{example}

\subsubsection*{Algebraic properties}

Similar to ultrafilter asymptotic cones the directional ones respects products:

\begin{theorem}[See Theorem \ref{T:tC products}]
For any groups $H,G$ with conjugation-invariant norm
$$
\tC(G \times H) \cong \tC(G) \times \tC(H).
$$  
\end{theorem}

Directional asymptotic cones exhibit similar phenomena as ultrafilter ones in connection to extensions.

\begin{theorem}\label{thm:intro extensions}
Let $1 \to N \to G \xto{\pi} H \to 1$ be a short exact sequence of groups with conjugation-invariant norms where $\pi$ is Lipschitz. 
Consider the induced map
\[
\tC(G) \xrightarrow{\tC(\pi)} \tC(H).
\]
\begin{enumerate}[label=(\alph*)]
\item
If $N$ is bounded in $G$ and if there exists a Lipschitz set theoretic section $H \to G$ then  $\tC(\pi)$ is bi-Lipschitz equivalence.
In particular it is an isomorphism of groups.
See Theorem \ref{T:tC quotient with section}.

\item
If the metric on $H$ is the quotient metric induced from $G$ via $\pi$, see Definition \ref{VD:metric quotient}, and if $N$ is bounded in $G$ then $\tC(\pi)$ is an isometry.
See Corollary \ref{cor:metric quotient bounded kernel}.
\end{enumerate}
\end{theorem}

\subsection{Relation to the construction of Calegari and Zhuang} 

Let $F$ denote a free group and let $W \subseteq F$, which we refer to as a set of words.
Let $G$ be a group.
The set of $W$-verbal elements (or simply verbal elements) in $G$ is 
\[
X_W = \bigcup_{\varphi \in \OP{Hom}(F,G)} \varphi(W).
\]
The $W$-verbal subgroup of $G$ is $G_W = \langle X_W \rangle$.
We equip $G_W$ with the standard word norm $\| \ \|_{X_W}$ with respect to the set of generators $X_W$.

In \cite{MR2866929} Calegari and Zhuang study verbal subgroups making use of the asymptotic cones of $G_W$ (with the verbal word norm).
They fix a non-principal ultrafilter $\omega$ and the scaling sequence $\B d=(d_n)$ where $d_n=n$, and denote $\hat{A}_W(G) = \Cone_\omega(G_W,\B d)$, or simply $\hat{A}_W$.
Clearly, $\hat{A}_W$ depends on the choice of the ultrafilter (and a scaling sequence~$\B d$).
Then in \cite[Section 3.4, Definition 3.8]{MR2866929} they define a subgroup $A_W \leq \hat{A}_W$ which they call the {\em real cone} of $G$.
This group consists of all elements of $\Cone_\omega(G_W,\B d)$ of the form 
$\left\{ g_1^{\lfloor t_1n \rfloor} \cdots g_k^{\lfloor t_kn \rfloor} \right\}_n$, 
where $g_i~\in~G_W$ and $t_i \in \RR$ and $\lfloor \ \rfloor$ is the floor function.
It is not clear (to us) to what extent the group $A_W$ depends on the choice of the ultrafilter.

The construction of the real cone is tightly related to our construction of the directional asymptotic cone of $G_W$ as follows.
Let $\Free_\RR(G_W)=\underset{g \in G_W}{*} \RR$ denote the free product, see Definition \ref{def:FSA}.
A key feature of the directional asymptotic cone described in Section \ref{V:hat pi construction} is that it is equipped with a {\em canonical} homomorphism $\eta \colon \Free_\RR(G_W) \to \tC(G_W)$ whose image is dense.
By composing it with the homomorphism 
\[
\rho \colon \tC(G_W) \to \Cone_\omega(G_W,\B d)
\]
(see Proposition \ref{P:directional into asymptotic}) we obtain a homomorphism $\rho' \colon \Free_\RR(G_W) \to \hat{A}_W$.
By easy inspection of $\eta$ and $\rho$,  the real cone $A_W$ is precisely the image of $\rho'$.
Since $\rho$ is Lipschitz and $\tC(G_W)$ is a complete metric space, the closure of $A_W$ in $\hat{A}_W$ is equal to the image of the directional cone $\tC(G_W)$ in $\Cone_\omega(G_W,\B d)$.

Some of the fundamental properties of the real cone proven in \cite{MR2866929} have their analogues for directional cones, and in fact, they follow from them.
For example, in \cite[Lemma 3.9 and 3.10]{MR2866929} it is shown that $A_W$ has an action of $\RR$ and consequently that it is a contractible space.

\subsection*{Acknowledgements}
This work was funded by Leverhulme Trust Research Project Grant RPG-2017-159.
For the purpose of open access, the authors have applied a Creative Commons
Attribution (CC BY) licence to any Author Accepted Manuscript version
arising from this submission.

%
%
%

\section{Conjugation invariant norms}\label{sec:conjugation-invariant norms}

In this section we review and collect results about bi-invariant metric and conjugation-invariant norms.
The proofs of the results we introduce are straightforward and are therefore deferred to Appendix \ref{app:cin proofs}.

\begin{void}\label{par:conjugation-invariant norms}
Let $G$ be a group.
A {\bf norm} on $G$ is a function $\nu \colon G \to [0,\infty)$ that satisfies the following
conditions:
\begin{enumerate}[label=(\arabic*)]
\item\label{par:conjugation-invariant norms:sep}
$\nu(g)=0 \iff g=1$ \hfill(non-degeneracy or point separation);
\item\label{par:conjugation-invariant norms:sym}
$\nu(g)=\nu(g^{-1})$ \hfill (symmetry);
\item\label{par:conjugation-invariant norms:triangle}
$\nu(gh) \leq \nu(g)+\nu(h)$ \hfill (triangle inequality).
\end{enumerate}
By relaxing the condition \ref{par:conjugation-invariant norms:sep} we obtain a {\bf pseudo\-norm} or a {\bf degenerate norm} on $G$.
It gives rise to a (pseudo) metric on $G$ via $d_{\nu}(g,h)=\nu(g^{-1}h)$.
It is left-invariant, i.e., $d_\nu(kg,kh)=d_\nu(g,h)$ for any $g,h,k \in G$.

A (pseudo) norm is called {\bf conjugation-invariant} if in addition, for every $g,h \in G$ it satisfies
\[
\nu(hgh^{-1})=\nu(g).
\]
Conjugation invariant (pseudo) norms are in one-to-one correspondence with bi-invariant (pseudo) metrics on $G$ via $d_{\nu}(g,h)=\nu(g^{-1}h)$.
By {\bf bi-invariant} we mean that for any $g,h,k \in G$ we have $d_\nu(kg,kh)=d_\nu(g,h)=d_\nu(gk,hk)$.
In this case $(G,d_\nu)$ becomes a topological group (but this is not necessarily the case if the norm is not conjugation-invariant).
\end{void}

\begin{void}[Metric spaces associated to pseudometrics]\label{V:pseudometric to metric}
A pseudometric $\tilde{d}$ on  $\tilde{X}$ induces an equivalence relation $\tilde{x} \sim \tilde{y} \iff \tilde{d}(\tilde{x},\tilde{y})=0$ and a metric space $(X,d)$ on the set of equivalence classes where $d([\tilde{x}],[\tilde{y}])=\tilde{d}(\tilde{x},\tilde{y})$.

Let $\nu$ be a conjugation-invariant pseudonorm on $G$ and let
\[
G_0 = \{ g \in G : \nu(g)=0\}.
\]
Then $G_0 \nsg G$  and $\|gG_0\|:=\nu(g)$ defines a conjugation-invariant norm on $G/G_0$. 
The metric associated to this norm is the metric associated to the pseudometric $d_{\nu}$.
\end{void}

The next result is a simple application of the definitions and will be used repeatedly.

\begin{lemma}
\label{L:conj invariant consequences}
Let $\| \ \|$ be a conjugation-invariant pseudonorm on $G$.
Then for any $g_1,\dots,g_n$ and $h_1,\dots,h_n$ in $G$
\[
\| (g_1\dots g_n)^{-1} (h_1\dots h_n)\| \leq \sum_{i=1}^n \|g_i^{-1} h_i\|.
\]
\end{lemma}

\begin{proof}
See Appendix \ref{app:cin proofs}.
\end{proof}

\begin{void}
Let $f \colon (G,\mu) \to (H,\nu)$ be a homomorphism of groups equipped with pseudonorms, whence pseudometrics.
We leave it to the reader to verify that $f$ is a Lipschitz function if and only if there exists $C \geq 0$ such that $\nu(f(g)) \leq C \mu(g)$ for all $g \in G$.
\end{void}

\begin{defn}\label{VD:metric quotient}
Let  $\pi \colon G \to H$ be a surjective homomorphism of groups equipped with conjugation-invariant pseudonorms $\| \ \|_G$ and $\| \ \|_H$. 
We call $\pi$ a {\bf metric quotient} if for every $h \in H$
\begin{equation}\label{eqn:metric quotient condition}
\| h\|_H = \inf \left\{ \|g\|_G : g \in \pi^{-1}(h)\right\}.
\end{equation}
\end{defn}

Clearly, a metric quotient $\pi \colon G \to H$ is Lipschitz with constant $1$.


\begin{lemma}\label{L:quotient pseudonorm}
Let $\pi \colon G \to H$ be a surjective homomorphism and let $\| \ \|_G$ be a pseudonorm on $G$.
Define a function $\| \ \|_H \colon H \to [0,\infty)$ 
by means of \eqref{eqn:metric quotient condition}.
Then $\| \ \|_H$ is a pseudonorm on $H$.
It is conjugation-invariant if $\| \ \|_G$ is, in which case it is a metric quotient, hence Lipschitz with constant $1$.
\end{lemma}

\begin{proof}
See Appendix \ref{app:cin proofs}.
\end{proof}

\medskip
\noindent
{\bf Remark:}
Lemma \ref{L:quotient pseudonorm} is valid for groups with pseudonorms but fails for general pseudometric spaces.

\begin{prop}\label{P:quotient complete}
Let $\pi \colon G \to H$ be a metric quotient of groups equipped with conjugation-invariant norms $\| \ \|_G$ and $\| \ \|_H$ respectively.
If $G$ is a complete metric group then so is $H$.
\end{prop}

\begin{proof}
See Appendix \ref{app:cin proofs}.
\end{proof}

\begin{prop}\label{P:quotient Lipschitz maps}
Consider the following commutative diagram of homomorphisms between groups equipped with conjugation-invariant norms.
\[
\xymatrix{
G \ar[r]^\varphi \ar@{->>}[d]_{\pi_G} &
H \ar@{->>}[d]^{\pi_H}
\\
G' \ar[r]^{\varphi'} & 
H'
}
\]
Suppose that the vertical arrows are metric quotients. 
If $\varphi$ is Lipschitz then $\varphi'$ is Lipschitz with the same constant.
\end{prop}

\begin{proof}
See Appendix \ref{app:cin proofs}.
\end{proof}

The next result is left to the reader.

\begin{prop}\label{P:metric quotient 2 out of 3}
Let $f \colon G \to G'$ and $g \colon G' \to G''$ be homomorphisms between groups equipped with conjugation-invariant pseudonorms.
If $f$ and $g$ are metric quotients then so is $g \circ f$.
If $f$ and $g \circ f$ are metric quotients then so is $g$.
\end{prop}

Next we discuss an important collection of conjugation-invariant norms, a
slight generalisation of the standard conjugation-invariant {\em word norms}.
In the case of a free group, such a norm was investigated by Nie
\cite{1412.0101} under the name {\it weighted cancellation length}. They are
essential for the construction of the directional cone.

\begin{void}\label{V:word norms}
By a {\bf length function} on a set $A$ we mean a function 
\[
\mu \colon A \to [0,\infty).
\]
Let $G$ be a group and $A \subseteq G$ a normal generating set. 
That is, $G$ is the smallest normal subgroup containing $A$.
Let $\mu \colon A \to [0,\infty)$ be a length function.
The {\bf conjugation-invariant  word norm} associated to $\mu$ is defined by 
\begin{equation}
\label{eqn:def norm mu}
\| g \|_{G;\mu} =
\inf \left\{ \sum_{i=1}^n \mu(a_i) \ : \ g=(g_1 a_1^{\pm 1} g_1^{-1})\cdots (g_n a_n^{\pm 1} g_n^{-1}), \ a_i \in A, g_i \in G\right\}.
\end{equation}

The terminology is a somewhat misleading since in general $\| \ \|_{G;\mu}$ is only a {\em pseudonorm}.
It is straightforward to check that $g \mapsto \|g\|_{G;\mu}$ is  a conjugation-invariant.
The details are left to the reader.

Notice that  $\|a\|_{G;\mu} \leq \mu(a)$ for all $a \in A$ and that in general strict inequality may hold.
Also, even if $\mu(a) \neq 0$ for all $a \in A$, the resulting $\| \ \|_{G;\mu}$ may be degenerate.
For example, $A=\{ \tfrac{1}{n} \}$ (normally) generates $\QQ$ and $\mu(\tfrac{1}{n})=\tfrac{1}{n^2}$ yields a vanishing pseudonorm on $\QQ$ since $\| \tfrac{1}{n} \|_{\QQ;\mu} = \lim_k k \mu(\tfrac{1}{nk})=0$.
\end{void}

An important special case is the ``standard'' conjugation-invariant norm:

\begin{defn}\label{defn:standard word norm}
Suppose that $G$ is normally generated by $A$.
The conjugation-invariant norm defined in  \ref{V:word norms} with respect to the constant length function $1 \colon A \to [0,\infty)$ is the {\em standard conjugation-invariant word norm}.
It will be denoted $\| \ \|_{G;A}$.
\end{defn}

Thus, $\|g\|_{G;A}$ is the {\em minimum} integer $n \geq 0$ such that $g$ can be written as a product of conjugates of elements of $A$ and their inverses.
It is easy to check that this is a norm rather than a pseudo norm.

\begin{example}\label{example:standard Zn L1}
The standard basis $E=\{e_1,\dots,e_n\}$ normally generates $\ZZ^n$ and the standard word norm $\| \ \|_{\ZZ^n;E}$ coincides with the $L^1$-norm $\| (x_1,\dots,x_n)\|_1=\sum |x_i|$.
\end{example}

\medskip
\noindent
{\bf Important remark about notation:}
In order to improve readability, whenever $G$ is understood from the context we will take the liberty to write $\| \ \|_\mu$ instead of $\| \ \|_{G;\mu}$ and $\| \ \|_A$ instead of $\| \ \|_{G;A}$ whenever this does not cause confusion.

\begin{prop}\label{P:norm-mu universally Lipschitz}
The norm $\| \ \|_{G;\mu}$ on $G$ defined in \eqref{eqn:def norm mu} has the property that for any homomorphism $f \colon G \to H$ where $H$ is equipped with a conjugation-invariant pseudonorm $\| \ \|_H$, if $\|f(a)\|_H \leq C \cdot \mu(a)$ for all $a \in A$ then $f$ is Lipschitz with constant $C$.
\end{prop}

\begin{proof}
See Appendix \ref{app:cin proofs}.
\end{proof}

\begin{prop}\label{prop:finite generating sets bi lipschitz equivalence}
Suppose $G$ is normally generated by finite subsets $A,A'$.
Then the standard conjugation-invariant word norms $\| \ \|_A$ and $\| \ \|_{A'}$ are bi-Lipschitz equivalent.
\end{prop}

\begin{proof}
See Appendix \ref{app:cin proofs}.
\end{proof}

\begin{prop}\label{prop:quotient of canonical word norms}
Let $\pi \colon G \to H$ be a group epimorphism and $A \subseteq G$ a normally generating set.
Set $B=\pi(A)$ and consider the standard conjugation-invariant norm $\| \ \|_A$ and $\| \ \|_B$ on $G$ and $H$ respectively.
Then $\pi \colon (G,\| \ \|_A) \to (H,\| \ \|_B)$ is a metric quotient map.
\end{prop}

\begin{proof}
See Appendix \ref{app:cin proofs}.
\end{proof}

\begin{void}\label{V:completions}

Let $(X,d)$ be a pseudometric space.
Let $\B C(X)$ denote the set of all Cauchy sequences $\B x =(x_n)$ in $X$.
There is a pseudometric on $\B C(X)$ defined by $\delta(\B x,\B y) = \lim_{n \to \infty} d(x_n,y_n)$.
The metric space that $(\B C(X), \delta)$ gives rise to, see \ref {V:pseudometric to metric}, is denoted $(\widehat{X},\widehat{d})$ and called the {\bf completion} of $X$.
There is an embedding $X \to \B C(X)$ which sends any $x \in X$ to the constant sequence.
Its image is dense in $\B C(X)$.
If $X$ is a metric space (rather than pseudometric) then the resulting $X \to \widehat{X}$ is an isometric embedding with dense image.
\end{void}

\begin{remark}
By construction $\widehat{X}$ is always a bona fide metric space even if we started with a {\em pseudo} metric space $X$.
\end{remark}

\section{The groups $\Free_S(A)$ for subrings $S \subseteq \RR$}\label{sec:FSA}

\begin{defn}\label{def:FSA}
Let $A$ be a set and $S \leq \RR$ a subring.
Consider the additive group of $S$ and set
\[
\Free_S(A) = \underset{{a \in A}}{*} S,
\]
the free product of copies of $S$ indexed by $a \in A$.
The ring structure of $S$ will be essential in Definition \ref{defn:psit in S}.
\end{defn}

Given a set $X$ we will write
\[
W(X)
\]
for the set of all words in the alphabet $X$ with concatenation as monoidal operation.

\begin{void}[Presentation of $\Free_A(S)$]\label{V:presentation FSA}
Let $A(S)$ be the following set of ``symbols''
\[
A(S) = \{a(s) \ | \ a \in A,  s \in S\}.
\]
It is clear that $\Free_S(A)$ has the following presentation
\begin{equation}\label{eqn:presentation of FSA}
\Free_S(A) = \left \langle A(S) \  \big| \ a(0) = 1 \ , \ a(r+s)=a(r)\cdot a(s) \right \rangle
\end{equation}
We will use this presentation throughout.

Thus, $\Free_S(A)$ is the quotient of the monoid $W(A(S))$ of all words on the alphabet $A(S)$ subject to the relations $a(0) \sim \{\}$ (the empty word) and $a(r+s)=a(r) \cdot a(s)$.

Clearly $A(S)$ generates $\Free_S(A)$, and hence normally generates it.
Any word $\B w \in W(A(S))$ has the form
\[
\B w = a_1(s_1) \cdots a_n(s_n)
\]
for some $a_i \in A$ and $s_i \in S$.
It is called reduced if $s_i \neq 0$ for all $i$ and $a_{i} \neq a_{i+1}$.
It is well known that any $g \in \Free_S(A)$ is represented by a unique reduced word $\B w$.
\end{void}

{\bf Throughout the remainder of the section} we fix a set $A$, a length function $\mu \colon A \to [0,\infty)$ and a subring $S \leq \RR$.

\begin{defn}\label{defn:total length function}
The {\bf total length function} obtained from $\mu$ is the function 
\[
\mu_S \colon W(A(S)) \to [0,\infty)
\]
defined for any $\B w \in W(A(S))$ of the form $\B w = a_1(s_1) \cdots a_n(s_n)$ by
\[
\mu_S (\B w) = \sum_{i=1}^n |s_i| \cdot \mu(a_i).
\]
We will frequently abuse notation and write $\mu$ instead of $\mu_S$.
\end{defn}

\begin{defn}\label{def:mu word norm}
Recall that $A(S)$ generates $\Free_S(A)$.
Let $\| \ \|_{\Free_S(A);\mu_S}$ denote the norm on $\Free_S(A)$ defined in \ref{V:word norms} with respect to the length function $\mu_S$ on $A(S)$.
\end{defn}

\begin{void}\label{V:when S=Z its F(A)}
When $S=\ZZ$ it is clear that $\Free_\ZZ(A)=\underset{a \in A}{*} \ZZ$ is the free group generated by $A$, denoted $F(A)$.
Indeed, $a(k)$ in $A(\ZZ)$ corresponds to $a^k$ in $\Free(A)$. 

If $\mu$ is the constant length function $1 \colon A \to [0,\infty)$ then the norm $\| \ \|_{\Free_\ZZ(A);1}$ is the standard conjugation-invariant word norm on $\Free(A)$, see Definition \ref{defn:standard word norm}.
\end{void}

The results in \cite{MR3426433} give an alternative description of the conjugation-invariant norm on $\Free(A)$ as a {\bf cancellation norm}.

In this section we will generalise this result and obtain a similar description of the norm $\| \ \|_{\Free_S(A);\mu_S}$ as a cancellation norm.
Once the language is set up, the results are expected.
In the interest of not blurring the essence of the construction by the details, the proofs of the results in this section are deferred to Appendix \ref{app:proof of FSA}.

\begin{defn}\label{defn:syllables}
The {\bf $A$-type} of a symbol (element) $a(s) \in A(S)$ is $a$.
A word $\B w \in W(A(S))$ is called a {\bf syllable} if all the letters it involves have the same $A$-type, i.e., if it has the form
\[
\B w = a(s_1) \cdots a(s_n)
\]
for some $a \in A$ and $s_1,\dots,s_n \in S$.
\end{defn}

The following is clear and is left to the reader to verify.

\begin{lemma}\label{lem:syllable basic 1}
\begin{enumerate}[label=(\alph*)]
\item\label{lem:syllable basic 1:syllabic rep}
Any word $\B w \in W(A(S))$ can be written in a unique way as the concatenation of syllables
\[
\B w = \B w_1 \cdots \B w_n
\]
where the $A$-types of $\B w_i$ and  $\B w_{i+1}$ are distinct for all $1 \leq i <n$.

\item
Any syllable $\B w = a(s_1) \cdots a(s_n)$ is equivalent in $\Free_S(A)$ to $a(s)$ where $s=\sum_i s_i$.
Hence, $\B w$ is trivial in $\Free_S(A)$ if and only if $s=0$.
\end{enumerate}
\end{lemma}

\begin{defn}\label{def:syllabic representation}
The presentation of a word $\B w=\B w_1 \cdots \B w_n$ in the alphabet $A(S)$ as in Lemma \ref{lem:syllable basic 1}\ref{lem:syllable basic 1:syllabic rep} is called the {\bf syllabic form} of $\B w$.
\end{defn}

\begin{defn}
For any $r \in \RR$ let $\myinterval{0,r}$ denote the interval $[0,r]$ if $r \geq 0$ or the  interval $[r,0]$ if $r\leq 0$.
If $S \leq \RR$ is a subring, let $\myinterval{0,r}_S$ denote $\myinterval{0,r} \cap S$.
\end{defn}

We note that if $t \in \myinterval{0,r}$ then it is always the case that $|t| \leq |r|$.

\begin{defn}\label{defn:S-cancellation sequence}
Let $\B w =a_1(s_1) \dots a_n(s_n)$ be a word in $W(A(S))$.
A {\bf sequence} in $\B w$ is a word $\B u \in W(A(S))$ of the form
\[
\B u = a_1(r_1) \cdots a_n(r_n)
\]
where $r_i \in \myinterval{0,s_i}_S$ for all $1 \leq i \leq n$.
It is called a {\bf cancellation sequence} in $\B w$ if
\[
\B w - \B u := a_1(s_1-r_1) \cdots a_n(s_n-r_n)
\]
represents the trivial element in $\Free_S(A)$.
\end{defn}


\begin{defn}\label{defn:can norm on WAS}
Define a {\em function}
\[
\| \ \|_{\can;S;\mu} \colon W(A(S)) \to [0,\infty)
\]
as follows.
For any $\B w \in W(A(S))$
\[
\| \B w \|_{\can;S;\mu} = \inf \{ \mu_S(\B u) : \text{$\B u$ is a cancellation sequence in $\B w$} \}
\]
When $S$ and $\mu$ are understood from the context we will write $\| \ \|_{\can}$ instead of $\| \ \|_{\can;S;\mu}$.
\end{defn}

The value of $\| \ \|_{\can}$ is in general mysterious, but is easy to calculate on syllables:

\begin{lemma}\label{lem:minimal cancellation seq in syllable}
Let $\B w = a(s_1) \cdots a(s_n)$ be a syllable in $W(A(S))$.
Set $s= \sum_i s_i$.
Then 
\[
\| \B w \|_{\can} = |s | \cdot \mu(a).
\]
Moreover, there exists a cancellation sequence $\B u$ in $\B w$ such that $\mu(\B u) =\| \B w\|_{\can}$.
\end{lemma}

\begin{proof}
See Appendix \ref{app:proof of FSA}
\end{proof}

The infimum in Definition \ref{defn:can norm on WAS} of $\| \ \|_{\can}$ is, in fact, a minimum:

\begin{prop}\label{prop:can-WAS minimum cancellation sequence}
Consider some $\B w \in W(A(S))$.
Then there exists a cancellation sequence $\B v \in W(A(S))$ in $\B w$ such that $\mu_S(\B v)=\| \B w\|_{\can}$.
We call it a {\bf minimal cancellation sequence}.
\end{prop}

\begin{proof}
Appendix \ref{app:proof of FSA}
\end{proof}

Next we show that  $\| \B w \|_{\can}$ only depends on the image of $\B w$ in $\Free_S(A)$.

\begin{prop}\label{prop:can-norm invariant representatives}
Suppose that $\B w, \B w' \in W(A(S))$ represent the same element in $\Free_S(A)$.
Then $\| \B w\|_{\can}=\| \B w'\|_{\can}$.
\end{prop}

\begin{proof}
Appendix \ref{app:proof of FSA}
\end{proof}

Proposition \ref{prop:can-norm invariant representatives} justifies the following definition.

\begin{defn}\label{defn:can norm of FSA}
The {\bf cancellation  norm} on $\Free_S(A)$ is the function $\| \ \|_{\can;S;\mu} \colon \Free_S(A) \to [0,\infty)$ defined for any $g \in \Free_S(A)$ by
\[
\|g\|_{\can;S;\mu} = \| \B w\|_{\can;S;\mu}
\]
where $\B w \in W(A(S))$ is a word in the alphabet $A(S)$  representing $g$.
When $S$ and $\mu$ are understood from the context we will write $\| \ \|_{\can}$.
\end{defn}

\begin{remark}
One checks that this is a pseudonorm thanks to Lemma \ref{lem:can WAS triangle}.
Alternatively, this is a special case of Theorem \ref{thm:FSA cancellation=CIWN} below.
\end{remark}

\begin{remark}
In general $\| \ \|_{\can;S;\mu}$ is only a pseudo-norm; it is a norm if and only if $\mu>0$.
This will be shown below in Proposition \ref{prop:mu>0 => norm} in conjunction with Theorem \ref{thm:FSA cancellation=CIWN}.
\end{remark}

\begin{remark}
By Proposition \ref{prop:can-WAS minimum cancellation sequence}, $\|g\|_{\can}=\mu_S(\B u)$ where $\B u$ is a minimal cancellation sequence in a word $\B w$ in the alphabet $A(S)$ which represents $g \in \Free_S(A)$.
\end{remark}

The proofs of Theorem \ref{thm:FSA cancellation=CIWN} -- Proposition \ref{prop:FSA dense FRA} below are deferred to Appendix \ref{app:proof of FSA}.

\begin{theorem}\label{thm:FSA cancellation=CIWN}
The function $\| \ \|_{\can} \colon \Free_S(A) \to [0,\infty)$ is equal to the conjugation-invariant word norm $\| \ \|_{\Free_S(A);\mu_S}$ defined in \ref{V:word norms}.
In particular, $\| \ \|_{\can}$ is a conjugation-invariant pseudo-norm.
\end{theorem}

\begin{prop}\label{prop:mu>0 => norm} 
The pseudo-norm $\| \ \|_{\Free_S(A);\mu}$ is a norm if and only if $\mu>0$.
\end{prop}

\begin{prop}\label{prop:restriction norms FSA<FRA}
Let $S \leq \RR$ be a subring.
Then  the norm on $\Free_S(A)$ as a subgroup of $(\Free_\RR(A); \| \ \|_{\Free_\RR(A);\mu_\RR})$ is the same as the norm $\| \ \|_{\Free_S(A);\mu_S}$.
\end{prop}

\begin{prop}\label{P:geodesics in F_R(A)}
$\Free_\RR(A)$ is a geodesic pseudometric space.
\end{prop}

\begin{prop}\label{prop:FSA dense FRA}
If $S$ is dense in $\RR$ then $\Free_S(A)$ is dense in $\Free_\RR(A)$ equipped with the norm $\| \ \|_{\Free_\RR(A);\mu_\RR}$.
\end{prop}

We next investigate basic naturality properties of the construction $\Free_S(A)$.
Any function $f \colon A \to B$ of sets gives rise to a natural homomorphism 
\[
f_* \colon \Free_S(A) \to \Free_S(B)
\]
which on generators $a(s) \in A(S) \subseteq \Free_S(A)$ has the effect $f_*(a(s))=f(a)(s) \in B(S)$.
This is the homomorphism $\underset{a \in A}{*} S \to \underset{b \in B}{*} S$ which sends the copy of $S$ corresponding to $a \in A$ via the identity $S \to S$ to the copy indexed by $f(a) \in B$.

\begin{defn}\label{D:proto-Lipschitz}
Let $\mu_A$ and $\mu_B$ be length functions on sets $A$ and $B$.
A function $f \colon A \to B$ is called {\bf proto-Lipschitz} if there exists $C \geq 0$ such that $\mu_B (f(a)) \leq C \mu_A(a)$ for all $a \in A$.
\end{defn}

\begin{prop}\label{P:naturality FRA} 
Suppose that $\mu_A$ and $\mu_B$ are length functions on sets $A,B$.
Let $f \colon A \to B$ be a proto-Lipschitz function with constant $C$. 
Then $f_*\colon \Free_S(A) \to \Free_S(B)$ is Lipschitz with constant $C$.
\end{prop}

\begin{proof}
Appendix \ref{app:proof of FSA}
\end{proof}

We next define a collection of endomorphisms $\psi_t \colon \Free_S(A) \to \Free_S(A)$ that will be central to this paper.

\begin{defn}\label{defn:psit in S}
Let $S \leq \RR$ be a subring and $t \in S$.
Use the universal property of $\Free_S(A)=\underset{a \in A}{*} S$ to define
\[
\psi_t \colon \Free_S(A) \to \Free_S(A)
\]
as the unique homomorphism 
which extends the homomorphisms $S \xto {s \mapsto t \cdot s} S \leq \Free_S(A)$, one for each $a \in A$.
\end{defn}

In terms of the presentation of $\Free_S(A)$ in \ref{V:presentation FSA}, for any $a_1(s_1) \cdots a_n(s_n)$ in $\Free_S(A)$
\begin{equation}\label{E:psi-t-explicit}
\psi_t \left( a_1(s_1) \cdots a_n(s_n) \right) \, = \,  a_1(ts_1) \cdots a_n(ts_n).
\end{equation}

\begin{prop}\label{P:psit monoidal properties}
The following holds for $\psi_t \colon \Free_\RR(A) \to \Free_\RR(A)$.
\begin{enumerate}[label=(\roman*)]
\item $\psi_0$ is the trivial homomorphism.
\item $\psi_1$ is the identity homomorphism.
\item $\psi_t \circ \psi_s = \psi_{st}$.
\end{enumerate}
In particular, if $t$ is invertible in $S$ then $\psi_t$ is an isomorphism with inverse $\psi_{t^{-1}}$.
\end{prop}

\begin{proof}
Immediate from \eqref{E:psi-t-explicit}.
\end{proof}

\begin{prop}\label{P:psit commutes with fstar}
For any $f \colon A \to B$  the following square commutes for any $t \in S$.
\[
\xymatrix{
\Free_S(A) \ar[r]^{f_*} \ar[d]_{\psi_t} &
\Free_S(B) \ar[d]^{\psi_t} 
\\
\Free_S(A) \ar[r]^{f_*} &
\Free_S(B)  
}
\]
\end{prop}

\begin{proof}
It suffices to check on generators $a(s) \in A(S)$.
Indeed, 
\[
f_*(\psi_t(a(s))) = f_*(a(ts))=f(a)(ts) =\psi_t(f(a)(s)) = \psi_t(f_*(a(s))).
\]
\end{proof}

It is immediate from the definition of $\psi_t$ in \eqref{E:psi-t-explicit} that $\| \psi_t(a(s))\| = |ts| \cdot \mu(a) = |t| \cdot \| a(s)\|$ for any generator $a(s) \in A(S)$.
The next lemma is fundamental to this paper.

\begin{lemma}\label{L:psit FRA properties}
\begin{enumerate}[label=(\alph*)]
\item
\label{L:psit FRA properties:multiplicative}
For any $w \in F_\RR(A)$ and any $t \in \RR$
\[
\|\psi_t(w)\|_{F_\RR(A);\mu_\RR} = |t| \cdot \|w\|_{F_\RR(A);\mu_\RR}.
\]

\item
\label{L:psit FRA properties:difference}
For any $w \in F_\RR(A)$ there exists $C_w \geq 0$ such that for any $s,t \in \RR$ 
\[
\| \psi_s(w)^{-1}\psi_t(w)\|_{F_\RR(A);\mu_\RR} \leq C_w |s-t|.
\]
\end{enumerate}
\end{lemma}

\begin{proof}
Appendix \ref{app:proof of FSA}
\end{proof}

\begin{defn}\label{def:w bracket}
Let $\B w=a_1(s_1) \cdots a_n(s_n)$ be a word in $W(A(S))$.
For any $t \in \RR$ set
\[
\B w\{t\} = a_1(\lfloor ts_1 \rfloor) \cdots a_n(\lfloor ts_n \rfloor)
\]
where $\lfloor \ \rfloor$ is the floor function.
Note that $\B w\{t\}$ is a word in the alphabet $A(\ZZ)$.
\end{defn}

\begin{prop}\label{prop:w truncated}
Consider some $g \in \Free_\RR(A)$ and a word $\B w$ in the alphabet $A(\RR)$ which represents it.
Then viewing the words $\B w\{t\}$ as elements in $\Free(A) \leq \Free_\RR(A)$, the latter equipped with the norm $\| \ \|_{\Free_\RR(A);\mu_\RR}$,
\[
\lim_{t \to \infty} \psi_{1/t}(\B w\{t\}) = g.
\]
\end{prop}

\begin{proof}
Write $\B w=a_1(s_1) \cdots a_n(s)$.
The result follows from Lemma \ref{L:conj invariant consequences}:
\[
\| \B w\cdot \psi_{1/t}(\B w\{t\})^{-1}\|_{\Free_\RR(A);\mu}
\leq
\sum_i \| a(s_i) a(-\tfrac{\lfloor ts_i \rfloor}{t}) \|_{\Free_\RR(A);\mu}
\leq
\sum_i |s_i - \tfrac{\lfloor ts_i \rfloor}{t}| \cdot \mu(a_i) \xto{t \to \infty} 0.
\]
\end{proof}

\section{The group of asymptotic directions}
\label{section:direction}

\noindent
{\bf Running assumption:}
As in the previous section all sets $A,B,\dots$ are assumed to be equipped with length functions $\mu_A,\mu_B,\dots$ (see Section \ref{V:word norms}).
By default, for any subring $S \leq \RR$ the groups $\Free_S(A)$ are equipped with the (pseudo) norm $\| \ \||_{\Free_S(A);\mu_S}$ from Definition \ref{def:mu word norm}.
If $S=\ZZ$ we simply write $\Free(A)$ instead of $\Free_\ZZ(A)$.

\medskip

\begin{defn}\label{def:scaling set}
A {\bf scaling set} is an unbounded subset $T$ of the interval $(0,\infty)$.
\end{defn}

\medskip
\noindent
{\bf Notation:}
Throughout $T,T'$ etc. will denote scaling sets.

\begin{defn}\label{D:Cauchy function}
A function $f \colon T \to X$ to a pseudo-metric space $(X,d)$ is called {\em Cauchy} if for any $\epsilon>0$, if $t,t' \in T$ are sufficiently large then $d(f(t),f(t'))<\epsilon$.
\end{defn}

\begin{defn}\label{D:DAd}
Let $T \subseteq (0,\infty)$ be unbounded.
A function 
\[
w \colon T \to F(A)
\] 
is called an {\bf asymptotic direction} if the function $f_w \colon T \to \Free_\RR(A)$ defined by
\[
f_w(t) = \psi_{1/t}(w(t))
\]
is Cauchy, where $\Free_\RR(A)$ is equipped with the norm $\| \ \|_{\Free_\RR(A);\mu}$.
Let 
\[
\C D(A,T;\mu) \subseteq \prod_{t \in T} \Free(A)
\]
denote the {group of all asymptotic directions} in $F(A)$ with respect to $T$ (and $\mu$).
When $\mu$ is understood from the context we write $\C D(A,T)$.
When $T=(0,\infty)$ we write $\C D(A;\mu)$ or $\C D(A)$.
\end{defn}

To justify the terminology we need to show that $\C D(A)$ is a subgroup of $\prod_{t \in T} \Free(A)$.
This follows easily from the triangle inequality, the symmetry and conjugation-invariance of the norm and the fact that $\psi_t$ are homomorphisms.
The details are left to the reader.

If $x_1,x_2,\dots$ is a Cauchy sequence in a pseudo metric space $(X,d)$ and $y \in X$ is a basepoint, then $d(x_n,y)$ is a Cauchy sequence in $\RR$.
Since $\RR$ is complete, this justifies the following definition.

\begin{defn}\label{D:norm on D(A) 2}
For any $w \in \C D(A,T;\mu)$ define a conjugation-invariant pseudonorm on $\C D(A,T;\mu)$ by
\[
\| w \|_{\C D(A,T;\mu)} = \lim_{t \in T} \| \psi_{1/t}(w(t))\|_{\Free_\RR(A);\mu}.
\]
We write $\| \ \|_{\C D(A)}$ when $T$ (and $\mu$) is understood from the context.
\end{defn}

That $\| \ \|_{\C D(A)}$ is a conjugation-invariant pseudonorm follows easily from the fact that $\| \ \|_{\Free_\RR(A)}$ has this property and that $\psi_{t}$ is a homomorphism.
It follows from Lemma \ref{L:psit FRA properties} and Proposition \ref{prop:restriction norms FSA<FRA} that

\begin{equation}\label{E:borm on D(A) via F(A)}
\| w \|_{\C D(A)} = \lim_{t \in T} \tfrac{1}{t} \| w(t)\|_{\Free(A);\mu}.
\end{equation}

With the notation of Section \ref{V:pseudometric to metric} we make the following definition.

\begin{defn}
The {\bf group of asymptotic directions} on $A$ is
\[
\tD(A,T;\mu) = \C D(A,T;\mu) / \C D(A,T;\mu)_0.
\]
We will omit $T$ and $\mu$ from the notation whenever they are understood from the context.
Let $\| \ \|_{\tD(A)}$ be the induced conjugation-invariant norm.
We will denote by $[w]$ the image of $w \in \C D(A)$ in $\tD(A)$.
\end{defn}

\begin{prop}\label{P:tD(A) is complete}
The group $\C D(A,T)$ is a complete pseudo metric space and hence, $\tD(A,T)$ is a complete metric spaces.
\end{prop}

\begin{proof}
Let $w_1,w_2,\dots$ be a Cauchy sequence in $\C D(A)$.
We will find a pseudo limit $v$ in $\C D(A)$.
Since a Cauchy sequence converges if and only if it contains a convergent subsequence, we may assume by passage to a subsequence that 
\[
\|w_n^{-1}w_{n+1}\|_{\C D(A)} < \tfrac{1}{2^{n+1}}
\]
for all $n \geq 1$.
It follows from the definition of $\| \ \|_{\C D(A)}$ that there exist $r_1,r_2,\dots$ such that for any $n \geq 1$ the following hold.
\begin{itemize}
\item[(a)] 
For any $s,t \in T$ such that $s,t > r_n$ 
\[
\| \, \psi_{1/s}(w_n(s))^{-1} \cdot \psi_{1/t}(w_n(t)) \, \|_{\Free_\RR(A);\mu} < \tfrac{1}{2^{n}}.
\]
(This is because $w_n \in \C D(A)$).

\item[(b)]
For any $t \in T$ such that $t>r_n$ 
\[
\| \, \psi_{1/t}(w_{n+1}(t) \cdot w_n(t)^{-1})\, \|_{\Free_\RR(A);\mu} < \tfrac{1}{2^{n+1}}.
\]
(This is because $\|w_{n+1}w_n^{-1}\|_{\C D(A)}<\tfrac{1}{2^{n+1}}$).
\end{itemize}
By increasing the $r_n$'s we may arrange that $0<r_1<r_2< \dots$ and that $r_n \to \infty$.
For convenience we set $r_0=0$ and $w_0=1$, the constant function $w_0 \colon T \to \Free(A)$.
Observe that for any $j \geq i \geq 1$  and any  $t \in T \cap (r_{j-1},\infty)$
\begin{multline*}
\|\psi_{1/t}(w_i(t))^{-1} w_j(t))\|_{\Free_\RR(A);\mu} \leq
\sum_{k=i}^{j-1} \|\psi_{1/t}(w_k(t))^{-1} w_{k+1}(t))\|_{\Free_\RR(A);\mu} 
\\
<
\sum_{k=i}^{j-1} \tfrac{1}{2^{k+1}} < \tfrac{1}{2^{i}}.
\end{multline*}
Define $v \colon T \to \Free(A)$ as follows.
For any $t \in T$ there exists a unique $n \geq 0$ such that $r_n< t \leq r_{n+1}$.
Set
\[
v(t)=w_n(t).
\]
We claim that $v \in \C D(A)$.
To see this, given $\epsilon>0$ choose $n \geq 1$ such that $\tfrac{1}{2^{n-1}}<\epsilon$.
Consider some  $s, t > r_n$ and suppose that $s \geq t$.
Then $r_{i}<t \leq r_{i+1}$ and $r_{j}<s\leq r_{j+1}$ for some $j \geq i \geq n$.
Since $n \geq 1$,
\begin{align*}
\| \psi_{1/t}  & (v(t))^{-1}  \psi_{1/s}(v(s))\|_{\Free_\RR(A);\mu}  =  \| \psi_{1/t}(w_i(t))^{-1} \psi_{1/s}(w_j(s))\|_{\Free_\RR(A);\mu}
\\ 
& \leq 
\| \psi_{1/t}(w_i(t))^{-1} \psi_{1/s}(w_i(s))\|_{\Free_\RR(A);\mu} + \| \psi_{1/s}(w_i(s)^{-1} w_j(s))\|_{\Free_\RR(A);\mu}
\\
& <
\tfrac{1}{2^i} + \tfrac{1}{2^i}  < \tfrac{1}{2^{i-1}} \leq \tfrac{1}{2^{n-1}} < \epsilon.
\end{align*}
It remains to show that $w_n \to v$ in $\C D(A)$.
Given $\epsilon>0$ choose $m \geq 1$ such that $\tfrac{1}{2^{m}}<\epsilon$.
Suppose that $n \geq m$.
If $t > r_n$ then $r_{i}<t \leq r_{i+1}$ for some $i \geq n$ and since $n \geq m \geq 1$
\[
\|\psi_{1/t}(w_n(t)^{-1} v(t))\|_{\Free_\RR(A);\mu} =
\|\psi_{1/t}(w_n(t)^{-1} w_i(t))\|_{\Free_\RR(A);\mu} < \tfrac{1}{2^n} \leq \tfrac{1}{2^{m}}.
\]
Taking the limit $t \to \infty$ we get for all $n \geq m$
\[
\| w_n^{-1} v\|_{\C D(A)} \leq \tfrac{1}{2^m}<\epsilon.
\]
We deduce that $w_n \to v$.
This completes the proof.
\end{proof}

We remind the reader that $\Free_\RR(A)$ is equipped with the norm $\| \ \|_{\Free_\RR(A);\mu}$.
We will denote its completion by $\widehat{\Free_\RR(A)}$.
By definition, the elements of  $\C D(A)$ are functions $w \colon T \to \Free(A)$ such that $f_w(t)=\psi_{1/t}(w(t))$ is a Cauchy function into $\Free_\RR(A)$, hence it has a limit in $\widehat{\Free_\RR(A)}$.

\begin{defn}\label{D:Psi D(A)-->FR(A)}
Define a homomorphism $\Psi \colon \C D(A,T) \to \widehat{\Free_\RR(A)}$ by 
\[
\Psi(w) = \lim_{t \to T}\, \psi_{1/t}(w(t)). 
\]
Define a homomorphism $\hat{\Psi} \colon \tD(A) \to \widehat{\Free_\RR(A)}$ by 
\[
\hat{\Psi}([w]) = [\Psi(w)]. 
\]
\end{defn}

\begin{defn}\label{D:iota FR(A)-->tD(A)}
Define a homomorphism $\iota \colon \Free_\RR(A) \to \tD(A,T)$ as follows.
Given $w \in \Free_\RR(A)$ choose a word $\B w$ in the alphabet $A(\RR)$ which represents $w$.
There results a function $w\trc{\cdot} \colon T \to \Free(A)$ defined by $t \mapsto \B w\trc{t}$, see Definition \ref{def:w bracket}. 
Set
\[
\iota(w) = [\B w\trc{\cdot}].
\]
\end{defn}

We say that a homomorphism $\varphi \colon G \to H$ between groups equipped with pseudonorms is {\bf norm preserving} if $\|\varphi(g)\|=\| g\|$ for all $g \in G$.

Next we show that $\hat{\Psi}$ and $\iota$ in Definitions \ref{D:Psi D(A)-->FR(A)} and \ref{D:iota FR(A)-->tD(A)} do not depend on the choice of representatives.

\begin{prop}\label{P:properties Psi iota} 
The function $\hat{\Psi}$ and $\iota$ are well defined homomorphisms.
Moreover, $\hat{\Psi}$ is an isometry $\tD(A,T) \cong \widehat{\Free_\RR(A)}$ and $\hat{\Psi} \circ \iota$ is the canonical homomorphism $\Free_\RR(A) \to \widehat{\Free_\RR(A)}$.
In particular, $\iota$ is norm preserving and its image is dense in $\tD(A,T)$.
\end{prop}

\begin{proof}
We will omit $T$ from the notation.
It is clear that $\Psi$ is well defined and it is a homomorphism because the norm on $\C D(A)$ is conjugation-invariant.
It is norm preserving because for any  $w \in \C D(A)$ 
 \begin{multline*}
\| \Psi(w)\|_{\widehat{\Free_\RR(A)}} = \| \lim_{t\to \infty} \psi_{1/t}(w(t))\|_{\widehat{\Free_\RR(A)}} =
\lim_{t\to \infty} \|  \psi_{1/t}(w(t))\|_{\widehat{\Free_\RR(A)}} 
\\
=
\lim_{t\to \infty} \|  \psi_{1/t}(w(t))\|_{\Free_\RR(A)} =
\|w\|_{\C D(A)}.
\end{multline*}
In particular $\OP{Ker}(\Psi)=\C D(A)_0$.
So $\Psi$ factors through $\hat{\Psi}$ which is therefore a norm preserving monomorphism, i.e., an isometric embedding.
Since $\tD(A)$ is complete by Proposition \ref{P:tD(A) is complete}, the image of $\hat{\Psi}$ is closed in $\widehat{\Free_\RR(A)}$.

We now prove that $\iota$ is well defined.
Suppose that $\B w$ represents $w \in \Free_\RR(A)$.
It follows from Proposition \ref{prop:w truncated} that $t \mapsto \psi_{1/t}(\B w\trc{t})$ is a Cauchy function in $\Free_\RR(A)$, hence by definition $\B w\trc{\cdot} \in \C D(A)$.
By Propositions \ref{prop:w truncated} and \ref{L:psit FRA properties}\ref{L:psit FRA properties:multiplicative}, if $\B w'$ is another word representing $w$ then 
\[
\| \B w\trc{\cdot}^{-1} \B w'\trc{\cdot}\|_{\C D(A)} =
\lim_t \| \psi_{1/t}(\B w\trc{t}^{-1} \cdot \B w'\trc{t})\|_{\Free_\RR(A);\mu} =
\| w^{-1} w \|_{\Free_\RR(A);\mu} =0.
\]
So $\iota$ is independent of choices, hence it is well defined.
It is a homomorphism because by construction $(\B w \cdot \B w')\trc{t}=\B w\trc{t} \cdot \B w'\trc{t}$ (concatenation of words, see Definition \ref{def:w bracket}).

By Proposition \ref{prop:w truncated}, for any $w \in \Free_\RR(A)$ and a representing word $\B w$,
\[
\hat{\Psi}(\iota(w)) = \hat{\Psi}([\B w\trc{\cdot}]) = \lim_t \psi_{1/t}(\B w\trc{t}) = w.
\]
(The limit is in $\Free_\RR(A)$).
Therefore $\hat{\Psi} \circ \iota$ is the canonical homomorphism $\Free_\RR(A) \to \widehat{\Free_\RR(A)}$.
In particular, the image of $\hat{\Psi}$ contains a dense subset of $\widehat{\Free_\RR(A)}$ and since it is closed, $\hat{\Psi}$ is surjective, hence an isometry.
Since $\Free_\RR(A) \to \widehat{\Free_\RR(A)}$ is norm preserving, so is $\iota$.
\end{proof}

Next, we show that $\C D(-,T)$ and $\tD(-,T)$ are functorial with respect to proto-Lipschitz functions (Definition \ref {D:proto-Lipschitz}) and with respect to inclusions $T' \subseteq T$ of scaling sets.
Recall that any function $f \colon A \to B$ induces a homomorphism $f_* \colon \Free(A) \to \Free(B)$.

\begin{prop}\label{P:naturality D and tD}
The assignments $A \mapsto \C D(A,T)$ and $A \mapsto \tD(A,T)$ are functors from the category of sets equipped with length functions and with proto-Lipschitz functions as morphisms to the category of groups with conjugation-invariant (pseudo) norms and Lipschitz homomorphisms.

In more detail, a proto-Lipschitz $f \colon A \to B$ with constant $C$ induces homomorphisms $\C D(f) \colon \C D(A,T) \to \C D(B,T)$ and $\tD(f) \colon \tD(A,T) \to \tD(B,T)$ defined by $\C D(f)(w) = f_* \circ w$ and $\tD(f)([w]) = [f_* \circ w]$, both are Lipschitz with constant $C$.
Also, $\C D(g \circ f)=\C D(g) \circ \C D(f)$ and $\C D(\OP{id})=\OP{id}$, and similarly $\tD(g \circ f)=\tD(g) \circ \tD(f)$ and $\tD(\OP{id})=\OP{id}$.
\end{prop}

\begin{proof}
Let $f \colon A \to B$ be proto Lipschitz with constant $C$.
By Proposition \ref{P:norm-mu universally Lipschitz} both $f_* \colon \Free(A) \to \Free(B)$ and $\Free_\RR(f) \colon \Free_\RR(A) \to \Free_\RR(B)$ are Lipschitz homomorphisms with constant $C$.
Consider some $w \in \C D(A)$ and set $v=f_* \circ w$.
Then $v \colon T \to \Free(B)$ and by Proposition \ref{P:psit commutes with fstar}, for any $t \in T$
\[
\psi_{1/t}(v(t))= \psi_{1/t}(f_* \circ w(t)) = \Free_\RR(f) (\psi_{1/t}(w(t))).
\]
Since the right hand side is Cauchy, $v \in \C D(B)$.

Thus, the assignment $w \mapsto f_* \circ w$ gives rise to a well defined function $\C D(f) \colon \C D(A) \to \C D(B)$.
It is a homomorphism because $f_*$ is.
It is Lipschitz with constant $C$ because 
by \eqref{E:borm on D(A) via F(A)}
\begin{align*}
\| \C D(f)(w)\|_{\C D(B)} &=
\lim_t \tfrac{1}{t} \| f_*(w(t))\|_{\Free(B);\mu_B} \\
& \leq
C \lim_t \tfrac{1}{t} \| w(t)\|_{\Free(A);\mu_A} 
\\
&=
C \| w\|_{\C D(A)}.
\end{align*}
In particular $\C D(f)$ carries $\C D(A)_0$ into $\C D(B)_0$ so $\C D(f)$ induces a homomorphism $\tD(f) \colon \tD(A) \to \tD(B)$ which is Lipschitz with constant $C$ and $\tD(f)([w])=[f_* \circ w]$.
The functorial properties of $A \mapsto \C D(A,T)$ follow from those of $A \mapsto \Free(A)$, i.e., $(g \circ f)_*=g_* \circ f_*$ and $\OP{id}_*=\OP{id}$.
Hence the functorial properties of $A \mapsto \tD(A,T)$.
\end{proof}

\begin{prop}\label{P:naturality tDA at T}
An inclusion $T \subseteq T'$ of scaling sets gives rise to natural transformations of functors 
\begin{align*}
\res^{T'}_T \colon \C D(-,T') \to \C D(-,T)    & ,&  &\res^{T'}_T(w) = w|_T \\
\widehat{\res}^{T'}_T \colon \tD(-,T') \to \tD(-,T)      & ,& & \hat{\res}^{T'}_T([w]) = [w|_T]
\end{align*}
The second of which is a an isometric embedding.
\end{prop}

\begin{proof}
Clearly, if $u \colon T' \to \Free_\RR(A)$ is Cauchy then so is $u|_T$.
Therefore, $\res^{T'}_{T} \colon \C D(A,T') \to \C D(A,T)$ defined by $w \mapsto w|_T$ is a well defined homomorphism.
By definition and Proposition \ref{L:psit FRA properties}\ref{L:psit FRA properties:multiplicative} it is norm preserving:
\begin{multline*}
\| w|_T\|_{\C D(A,T)} 
= \lim_{t \in T} \|\psi_{1/t}(w(t))\|_{\Free_\RR(A);\mu_A} 
 \\
=
\lim_{t \in T'} \|\psi_{1/t}(w(t))\|_{\Free_\RR(A);\mu_A} =
\| w\|_{\C D(A,T')}.
\end{multline*}
It is a natural transformation of functors because by inspection the following square commutes for any proto-Lipschitz $f \colon A \to B$.
\[
\xymatrix{
\C D(A,T') \ar[r]^{\C D(f,T')} \ar[d]_{\res^{T'}_T} &
\C D(B,T') \ar[d]^{\res^{T'}_T} 
\\
\C D(A,T) \ar[r]^{\C D(f,T)} &
\C D(B,T).
}
\]
Since it is norm preserving there results a natural transformation of functors $\widehat{\res}^{T'}_T \colon \tD(-,T') \to \tD(-,T)$ defined by $\widehat{\res}^{T'}_T([w])=[w|_T]$.
By inspection the following triangle is commutative
\[
\xymatrix{
\tD(A,T') \ar[rr]^{\widehat{\res}^{T'}_T} \ar[dr]_{\hat{\Psi}} & &
\tD(A,T) \ar[dl]^{\hat{\Psi}} 
\\
& \widehat{\Free_\RR(A)}
}
\]
Since the arrows $\hat{\Psi}$ are isometries by Proposition \ref{P:properties Psi iota}, $\widehat{\res}^{T'}_T$ is an isometric embedding.
\end{proof}

Our next goal is to prove that $\tD(A,T)$ are contractible spaces.

\begin{notation}\label{nota:f_t}
Let $f \colon A \times B \to Y$ be a function (of sets).
Given $a \in A$ we write $f_a \colon B \to Y$ for the restriction of $f$ to $\{a\} \times B \cong B$.
Similarly, $f_b$ is the restriction to $A \times \{b\} \cong A$. The proof of the following
lemma is left to the reader.
\end{notation}

\begin{lemma}\label{L:extend homotpy from dense subspace}
Suppose that $A \subseteq X$ is a dense subspace of a metric space and $Y$ is a complete metric space.
Let $T$ be a metric space and endow $T \times X$ with the metric $\max\{d_T,d_X\}$.
Suppose that $f \colon T \times A \to Y$ is a function with the following properties.
With the notation \ref{nota:f_t},
\begin{enumerate}[label=(\roman*)]
\item \label{L:extend homotopy:Lipschitz fibre}
For any $t \in T$ the function $f_t \colon A \to Y$ is Lipschitz with constant $C_t$.

\item \label{L:extend homotopy:continuous bounded}
If $J \subseteq T$ is bounded then $\{C_t : t \in J\}$ is a bounded subset of $\RR$.

\item \label{L:extend homotopy:continuous fa}
For any $a \in A$ the resulting function $f_a \colon T \to Y$ is continuous.
\end{enumerate}
Then $f$ extends uniquely to a continuous function $\tilde{f} \colon T \times X \to Y$.
\qed
\end{lemma}

\begin{prop}\label{P:extend psi_t to D}
There is a continuous function $h \colon \RR \times \tD(A,T) \to \tD(A,T)$ such that the following diagram commutes for all $t \in \RR$.
\[
\xymatrix{
{F_\RR(A)} \ar[r]^{\psi_t} \ar[d]_{\iota} &
{F_\RR(A)} \ar[d]^{\iota} \\
\tD(A,T) \ar[r]_{h_t} & \tD(A,T),
}
\]
Moreover, each $h_t$ is a homomorphism and
\begin{enumerate}[label=(\roman*)]
\item\label{P:extend psi_t to D:Lip}
$\|h_t(x)\|_{\tD(A)}= |t| \cdot \|x\|_{\tD(A)}$ 

\item\label{P:extend psi_t to D:nat} 
$h_s \circ h_t=h_{st}$ and $h_1=\id$ and $h_0$ is the trivial homomorphism.
\end{enumerate}
\end{prop}

\begin{proof}
By composing with the isometry $\hat{\Psi}$  we may replace $\tD(A)$ with $\widehat{\Free_\RR(A)}$ and the map $\iota$ with the canonical homomorphism $i \colon \Free_\RR(A) \to \widehat{\Free_\RR(A)}$.
Since the latter is norm preserving, and since $\psi_t$ is Lipschitz with constant $|t|$ by Lemma \ref{L:psit FRA properties}\ref{L:psit FRA properties:multiplicative}, the universal property of the completion gives rise to a unique $h_t$  which extends $i \circ \psi_t \colon \Free_\RR(A) \to \widehat{\Free_\RR(A)}$ to a function $\tilde{h}_t \colon \widehat{\Free_\RR(A)} \to \widehat{\Free_\RR(A)}$.
Thus, there exists a unique function $h_t$ rendering the diagram commutative and is Lipschitz.
Item \ref{P:extend psi_t to D:Lip} follows from Lemma \ref{L:psit FRA properties}\ref{L:psit FRA properties:multiplicative} since the image of $\iota$ is dense in $\tD(A,T)$.
Similarly, \ref{P:extend psi_t to D:nat} follows from Proposition \ref{P:psit monoidal properties}.

The continuity of $h \colon \RR \times \tD(A,T) \to \tD(A,T)$ follows by applying  Lemma \ref{L:extend homotpy from dense subspace} with $X=\tD(A,T)$ and $A$ the image of $\Free_\RR(A)$ and $T=\RR$.
Indeed, conditions \ref{L:extend homotopy:Lipschitz fibre}, \ref{L:extend homotopy:continuous bounded} of that Lemma hold by Lemma \ref{L:psit FRA properties}\ref{L:psit FRA properties:multiplicative} which shows that $C_t=|t|$, and condition \ref{L:extend homotopy:continuous fa} follows from Lemma \ref{L:psit FRA properties}\ref{L:psit FRA properties:difference}.
\end{proof}

\begin{defn}\label{def:End(G)}
Let $G$ be a topological group.
Denote by $\End(G)$ the set of endomorphisms of $G$, which are continuous
with respect to the compact-open topology.
\end{defn}

Recall that the compact-open topology is designed so that there is an ``adjunction'' bijection
\[
\{ \text{continuous $f \colon \RR \to \End(G)$}\} \leftrightarrow
\{ \text{continuous $h \colon \RR \times G \to G$,  $h_t \in \End(G)$}\}
\]
Observe that $\tD(A,T)$ is equipped with a conjugation-invariant norm, and therefore it becomes a topological group.

\begin{cor}\label{C:tDA is contractible in category of groups}
$\tD(A,T)$ is contractible via a homotopy $h_t$ defined on the unit interval $[0,1]$ such that each $h_t$ is a homomorphism.
Hence, $\OP{End}(\tD(A,T))$ equipped with the compact-open topology is contractible.
\end{cor}

Recall that a (pseudo) metric space $X$ is called a {\em length space} if for any $x,y \in X$ and any $\epsilon>0$ there exists a path $\gamma \colon I \to X$ from $x$ to $y$ such that $\ell(\gamma)<d(x,y)+\epsilon$, where $\ell(\gamma)$ is the length of $\gamma$.

\begin{lemma}\label{L:closure of length spaces}
Let $G$ be a group equipped with a conjugation-invariant norm and $H$ a dense subgroup.
Suppose that $H$ is a length space and that $G$ is complete.
Then $G$ is a length space.
\end{lemma}

\begin{proof}
Since the metric on $G$ is invariant it suffices to show that for any $g \in G$ and any $\epsilon >0$ there exists a path $\gamma$ from $1$ to $g$ such that $\ell(\gamma)<\|g\|+\epsilon$.
Choose a sequence $h_n \in H$ such that $h_n \to g$.
We may assume that $\|g^{-1}h_n\|<\tfrac{1}{2^{n+2}}$ for all $n \geq 1$, hence $\|h_n^{-1}h_{n+1}\|<\tfrac{1}{2^{n+1}}$.
For every $n \geq 2$ set $\Delta_n=h_{n-1}^{-1} h_n \in H$.
Then $\| \Delta_n\| < \tfrac{1}{2^n}$ for all $n \geq 2$.
Since $H$ is a length space, choose paths $\delta_n \colon I \to H$ from $1$ to $\delta_n$ such that $\ell(\delta_n)<\tfrac{1}{2^{n-1}}$.
Observe that for any $t \in I$
\[
\|\delta_n(t)\| \leq \ell(\delta_n|_{[0,t]}) \leq \ell(\delta_n)<\tfrac{1}{2^{n-1}}.
\]
Given $\epsilon>0$ choose $n \geq 2$ such that $\tfrac{\epsilon}{2} < \tfrac{1}{2^{n-2}}$.
Choose a path $\beta \colon I \to H$ from $1$ to $h_n$ such that $\ell(\beta)<\|h_n\|+\tfrac{\epsilon}{2}$.
Observe that for every $k \geq n$
\[
h_k = h_n \Delta_{n+1} \cdots \Delta_k.
\]
Define paths $\alpha_k \colon I \to H$ by $\alpha_k = \beta \cdot \delta_{n+1} \cdots \delta_k$, i.e., 
\[
\alpha_k(t) = \beta(t) \cdot \delta_{n+1}(t) \cdots \delta_k(t), \qquad (t \in I).
\]
It is clear that $\alpha_k$ are continuous because the norm on $G$ makes it a topological group.
Since $G$ is complete, $\alpha_k$ converges uniformly to some $\alpha \colon I \to G$ which is therefore also continuous.
Clearly $\alpha(0)=1$ and 
\[
\alpha(1)=\lim_{k \geq n} \alpha_k(1) = \lim_{k \geq n} h_n \cdot \Delta_{n+1}\cdots \Delta_k = \lim_{k \geq n} h_k = g.
\]
So $\alpha$ is a path from $1$ to $g$.
By Lemma \ref{L:conj invariant consequences}, for any $t,t' \in I$
\[
\| \alpha_k(t)^{-1}\alpha(t')\| \leq \|\beta(t)^{-1} \beta(t')\| + \sum_{i={n+1}}^k \|\delta_i(t)^{-1} \delta_i(t')\|.
\]
It follows that $\ell(\alpha_k) \leq \ell(\beta)+\sum_{i=n+1}^k \ell(\delta_i)$.
Since the convergence is uniform, 
\[
\ell(\alpha)=\lim_{k \geq n} \ell(\alpha_k) \leq \ell(\beta) + \tfrac{1}{2^{n-1}} < \|h_n\|+\tfrac{\epsilon}{2} + \tfrac{1}{2^{n-1}} < \|g\| + \epsilon.
\]
This completes the proof.
\end{proof}

\begin{cor}\label{C:tD is length space} 
$\tD(A,T)$ is a length space.
\end{cor}

\begin{proof}
By Proposition \ref{P:properties Psi iota} we may replace $\tD(A,T)$ with $\widehat{\Free_\RR(A)}$.
The natural completion homomorphism $i \colon \Free_\RR(A) \to \widehat{\Free_\RR(A)}$ is norm preserving, so its image is a geodesic space by Proposition \ref{P:geodesics in F_R(A)}.
Apply Lemma \ref{L:closure of length spaces}.
\end{proof}

%
%
\section{The directional asymptotic cone}
\label{Sec:diecrional asymptotic cone}

\noindent
{\bf Running assumption and notation:}
\begin{itemize}[leftmargin=*]
\item
Throughout, $T,T',\dots$ denote scaling sets (Definition \ref{def:scaling set}).

\item
If $\lambda \colon A \to [0,\infty)$ is a length function and $S \leq \RR$ a subring, we write
\[
\Free_S(A;\lambda)
\]
for the group $\Free_S(A)$ equipped with the norm $\| \ \|_{\Free_S(A);\lambda_S}$, see Definition \ref{def:mu word norm}.

\item
Throughout, {\em all groups} $G$ are assumed to be equipped with a conjugation-invariant norm $\| \ \|$.

\item
If a length function on (the underlying set of) $G$ is not specified, it is always assumed to be the norm $\| \ \| \colon G \to [0,\infty)$, and we will denote it by $\mu$.
In this case we write $\Free_S(G)$ instead of $\Free_S(G;\mu)$ or (the confusing notation) $\Free_S(G;\| \ \|)$.
\end{itemize}

\begin{void}\label{V:notation generators FRG}
We denote the generators of $\Free(G)$ by $\overline{g}$ where $g \in G$.
Similarly, the generators of $\Free_\RR(G)$ are denoted $\overline{g}(r)$ for $g \in G$ and $r \in \RR$.
Notice that $\overline{g^{-1}} \neq \overline{g}^{-1}$ in $\Free(G)$. 
The first element is a generator, the other is the inverse of a different generator. 
Similarly, $\overline{g}(-r) = \overline{g}(r)^{-1} \neq \overline{g^{-1}}(r)$ in $\Free_\RR(G)$.

The assignment $\overline{g} \mapsto g$ gives rise to a canonical homomorphism
\[
\pi \colon F(G) \to G.
\]
It is Lipschitz with constant $1$, by Proposition \ref{P:norm-mu universally Lipschitz} (recall that $\Free(G)$ is equipped with the norm $\| \ \|_{\Free(G);\mu}$, where $\mu(g)=\|g\|$).
Given a scaling set $T$, $\pi$ induces a homomorphism
\[
\pi_* \colon \prod_{t \in T} \Free(G) \to \prod_{t \in T} G, \qquad \pi_*(w) = \pi \circ w.
\]
That is, given a function $w \colon T \to \Free(G)$, we get $\pi_*(w)(t)=\pi(w(t))$ for all $t \in T$.
\end{void}

Recall from Definition \ref{D:DAd} that $\C D(G,T)$ is a subgroup of $\prod_{t \in T} \Free(G)$  equipped with the norm in Definition \ref{D:norm on D(A) 2}.
Set 
\[
\C C(G,T) \overset{\text{def}}{=} \pi_*\big(\C D(G,T)\big).
\]
Thus, $\C C(G,T)$ is the image of $\C D(G,T)$ under $\pi_* \colon \prod_T \Free(G) \to \prod_T G$.
As in the previous sections, we will suppress $T$  from the notation whenever it doesn't cause confusion. 
By slight abuse of notation 
\[
\pi_* \colon \C D(G,T) \to \C C(G,T)
\]
will also denote the restriction of $\pi_*$ to $\C D(G,T)$.
Equip $\C C(G,T)$ with the quotient conjugation-invariant pseudonorm in Lemma~\ref{L:quotient pseudonorm}, i.e., for any $g \colon T \to G$ in $\C C(G,T)$
\[
\| g \|_{\C C(G,T))} = \inf \big\{ \|w\|_{\C D(G,T)} :  w \in \pi_*^{-1}(g) \big\}.
\]

\begin{defn}\label{D:tC}
The {\bf directional asymptotic cone} of $G$ is the metrification of $\C C(G,T)$, see Section \ref{V:pseudometric to metric}.
That is,
\[
\tC(G,T) = \C C(G,T)/\C C(G,T)_0
\]
equipped with the induced conjugation-invariant norm. 
We will write $[g]$ for the image of $g \in \C C(G,T)$ in $\tC(G,T)$.
\end{defn}

\begin{void}\label{V:hat pi construction}
It is clear from the definitions that $\pi_* \colon \C D(G,T) \to \C C(G,T)$ carries $\C D(G,T)_0$ into $\C C(G,T)_0$.
There results an epimorphism of normed groups
\[
\hat{\pi} \colon \tD(G,T) \xto{\ [w] \mapsto [\pi_*(w)] \ } \tC(G,T). 
\]
By the definition of $\| \ \|_{\C C(G)}$ and since the quotients $\C D(G,T) \to \tD(G,T)$ and $\C C(G,T) \to \tC(G,T)$ are norm preserving, one easily checks that
\[
\| [g]\|_{\tC(G)} = \inf \, \left\{ \|[w]\|_{\tD(G)} : [w] \in \hat{\pi}^{-1}([g]) \right\}.
\]
Thus, $\| \ \|_{\tC(G)}$ is the quotient norm in Lemma \ref{L:quotient pseudonorm} induced from $\| \ \|_{\tD(G)}$ and $\hat{\pi}$ is a metric quotient.
It follows from Proposition \ref{P:properties Psi iota} that the image of the composition
\[
\Free_\RR(G) \xto{\iota} \tD(G,T) \xto{\hat{\pi}} \tC(G,T)
\]
is dense.
\end{void}

\begin{prop}\label{P:tC is complete} 
$\tC(G,T)$ is a complete normed group.
\end{prop}

\begin{proof}
Immediate from Section \ref{V:hat pi construction} and Propositions  \ref{P:quotient complete} and  \ref{P:properties Psi iota}.
\end{proof}

The following two fundamental lemmas will be used repeatedly.

\begin{lemma}\label{L:limsup<norm}
For any $g \in \C C(G,T)$
\[
\limsup_{t \in T} \, \tfrac{1}{t} \| g(t)\| \leq \| g \|_{\C C(G)}.
\]
\end{lemma}

\begin{proof}
Choose an arbitrary $\epsilon>0$.
By definition of $\| \ \|_{\C C(G)}$ there exists $w \in \C D(G)$ such that $\pi_*(w)=g$ and $\|w\|_{\C D(G)} < \|g\|_{\C C(G)} + \epsilon$.
Since $\pi \colon \Free(G) \to G$ is Lipschitz with constant $1$ we get  $\|g(t)\| \leq \|w(t)\|_{\Free(G);\mu}$ for all $t\in T$ (where $\mu = \| \ \|$, the norm in $G$).
It follows from \eqref{E:borm on D(A) via F(A)} that
\[
\limsup_{t \in T} \tfrac{1}{t} \|g(t)\| \leq \limsup_{t \in T} \tfrac{1}{t} \|w(t)\|_{\Free(G);\mu}
= \|w\|_{\C D(G)} 
< \|g\|_{\C C(G)}+ \epsilon.
\]
Since $\epsilon>0$ was arbitrary, the result follows.
\end{proof}

\begin{lemma}\label{L:characterize norm 0 in C}
For any function $g \colon T \to G$,
\[
\limsup_{t \in T} \, \tfrac{1}{t} \|g(t)\| =0
\] 
if and only if $g \in \C C(G,T)$ and $\|g\|_{\C C(G)}=0$.
\end{lemma}

\begin{proof}
Suppose that $\limsup_t \tfrac{1}{t} \|g(t)\| =0$.
Define $w \colon T \to \Free(G)$ by 
\[
w(t) \, \overset{\text{def}}{=} \, \overline{g(t)}.
\]
Set $\mu = \| \ \|$, the norm on $G$.
Lemma \ref{L:psit FRA properties} and Proposition \ref{prop:restriction norms FSA<FRA} imply that for any $s,t>0$
\begin{multline*}
\| \psi_{1/t}(w(t))^{-1} \psi_{1/s}(w(s))\|_{\Free_\RR(G);\mu} 
\\
\leq
\|\psi_{1/t}(\overline{g(t)})\|_{\Free_\RR(G);\mu} + \|\psi_{1/s}(\overline{g(s)})\|_{\Free_\RR(G);\mu} 
\\
=
\tfrac{1}{t}\|\overline{g(t)}\|_{\Free(G);\mu} + \tfrac{1}{s}\|\overline{g(s)}\|_{\Free(G);\mu} =
\tfrac{1}{t}\|g(t)\| + \tfrac{1}{s}\|g(s)\|.
\end{multline*}
By assumption, this is arbitrarily small if $s,t \gg 0$, so $w(t) \in \C D(G)$.
Clearly $g=\pi_* \circ w$, so $g \in \C C(G)$.
Moreover, by \eqref{E:borm on D(A) via F(A)}
\[
\| g \|_{\C C(G)} \leq \|w\|_{\C D(G)} = \lim_t \tfrac{1}{t} \|w(t)\|_{\Free(G);\mu} = \lim_t \tfrac{1}{t} \|g(t)\| =0.
\]
This proves the ``only if'' part of the lemma.
The opposite implication follows  from Lemma \ref{L:limsup<norm}.
\end{proof}

Recall the homomorphisms $\hat{\pi} \colon \tD(G,T) \to \tC(G,T)$ and $\iota \colon \Free_\RR(G) \to \tD(G,T)$ from Section \ref{V:hat pi construction} and Definition~\ref{D:iota FR(A)-->tD(A)}.

\begin{lemma}\label{L:pihat powers}
Consider some $g \in G$ and $r \in \RR$.
\begin{enumerate}[label=(\roman*)]
\item\label{L:pihat powers:1}
For any integer $p \neq 0$.
\[
\hat{\pi} \circ \iota\Big(\overline{g}(r)\Big) = \hat{\pi} \circ \iota \left(\overline{g^p}(\tfrac{r}{p})\right).
\]

\item\label{L:pihat powers:2}
For any $h \in G$ 
\[
\hat{\pi} \circ \iota\Big(\overline{hgh^{-1}}(r)\Big) = \hat{\pi} \circ \iota\Big(\overline{g}(r)\Big).
\]
\end{enumerate}
\end{lemma}

\begin{proof}
\ref{L:pihat powers:1}
By definition $\hat{\pi}(\iota(\overline{g}(r)))$ is represented by the function $\gamma(t)=g^{\Myfloor{rt}}$ and $\hat{\pi}(\iota (\overline{g^p}(\tfrac{r}{p})))$ is represented by $\gamma'(t) = g^{p \Myfloor{rt/p}}$.
They represent the same element in $\tC(G)$ by Lemma \ref{L:characterize norm 0 in C} since $| \Myfloor{tr} - p\Myfloor{\tfrac{tr}{p}} | < |p|$ so 
\[
\lim_{t \in T} \tfrac{1}{t}\| \gamma(t)^{-1} \gamma'(t)\| =
\lim_{t \in T} \tfrac{1}{t}\| g^{p \Myfloor{rt/p} - \Myfloor{rt}} \| \leq
\lim_{t \in T}  \|g\| \cdot \tfrac{|\Myfloor{tr} - p \Myfloor{\tfrac{tr}{p}}|}{t} = 0.
\]
\ref{L:pihat powers:2}
By definition $\hat{\pi} \circ \iota(\overline{g}(r))$ is represented by $\gamma(t)=g^{\Myfloor{rt}}$ and $\hat{\pi} \circ \iota(\overline{hgh^{-1}}(r))$ is represented by $\gamma'(t)=(hgh^{-1})^{\Myfloor{rt}}=hg^{\Myfloor{rt}}h^{-1}$.
Since
\begin{align*}
\lim_{t \in T} \tfrac{1}{t}\|\gamma(t)^{-1}\gamma'(t)\| &=
\lim_{t \in T} \tfrac{1}{t}\| g^{-\Myfloor{rt}} hg^{\Myfloor{rt}}h^{-1}\| \\
& \leq
\lim_{t \in T} \tfrac{1}{t}\| g^{-\Myfloor{rt}} hg^{\Myfloor{rt}}\| + \lim_{t \in T} \tfrac{1}{t}\|h^{-1}\| \\
& =
\lim_{t \in T} \tfrac{1}{t}(\|h\|+\|h^{-1}\|)=0,
\end{align*}
they represent the same element in $\tC(G)$ by Lemma \ref{L:characterize norm 0 in C}.
\end{proof}

\begin{void}\label{V:stable length}
Recall that a sequence $\{ a_n\}_{n \geq 1}$ of real numbers is called subadditive if
\[
a_{n+m} \leq a_n + a_m
\]
for every $n,m \geq 1$.
Fekete's Lemma \cite{MR1544613} asserts that the sequence $\tfrac{a_n}{n}$ is convergent and
\[
\lim_{n \to \infty} \tfrac{a_n}{n} = \inf \{ \tfrac{a_n}{n} : n \geq 1\}.
\]
Let $G$ be a group equipped with conjugation-invariant norm.
By Fekete's Lemma 
\[
\tau(g) \overset{\text{def}}{=} \inf_n \frac{\|g^n\|}{n} = \lim_n \frac{\|g^n\|}{n}
\]
which is called the {\bf stable length} or {\bf translation length} of $g \in G$.
Observe that $\tau(g^p)=|p| \cdot \tau(g)$ for any integer $p$ and that $\tau(hgh^{-1})=\tau(g)$ for any $h \in G$.
\end{void}

\begin{cor}\label{cor:Cnorm image of generators in FRG}
Consider a generator $\alpha=\overline{g}(r) \in \Free_\RR(G)$.
Then
\[
\| \hat{\pi}(\iota(\alpha))\|_{\tC(G,T)} = |r| \cdot \tau(g).
\]
\end{cor}

\begin{proof}
Assume $r \neq 0$
First, $[h]=\hat{\pi}(\iota(\alpha))$ is represented by the function $h \colon T \to G$ given by
$h(t)=g^{\Myfloor{rt}}$.
It follows from Lemma \ref{L:limsup<norm} that $\limsup_{t \in T} \tfrac{1}{t} \|h(t)\| \leq \|[h]\|_{\tC(G,T)}$.
Since $T$ is unbounded and since $\|g^{-n}\|=\|g^{n}\|$, we get from \ref{V:stable length} above 
\[
\limsup_{t \in T} \tfrac{1}{t}\| g^{\Myfloor{rt}}\| 
=
\limsup_{t \in T} \tfrac{|\Myfloor{rt}|}{t} \cdot \tfrac{\| g^{|\Myfloor{rt}|} \|}{|\Myfloor{rt}|}
=
|r| \cdot \lim_{n \to \infty} \|g^n\| =|r| \cdot \tau(g).
\]
Thus $\|[h]\|_{\tC(G,T)} \geq |r| \cdot \tau(g)$.
By Lemma \ref{L:pihat powers}\ref{L:pihat powers:1}, for any $k >0$ we get $[h]=[\hat{\pi}(\iota(\overline{g^k}(\tfrac{r}{k})))]$ in $\tC(G,T)$.
Since $\hat{\pi}$ is a metric quotient, see paragraph \ref{V:hat pi construction}, and by Proposition \ref{P:properties Psi iota} and Theorem \ref{thm:FSA cancellation=CIWN} (recall that $\mu \colon G \to [0,\infty)$ is the norm on $G$)
\[
\| [h]\|_{\tC(G,T)} 
\leq 
\| \iota(\overline{g^k}(\tfrac{r}{k})) \|_{\tD(G,T)}
=
\| \overline{g^k}(\tfrac{r}{k})\|_{\Free_\RR(G)} 
= \tfrac{|r|}{k} \|g^k\| 
\xto{ \  k \to \infty \ } |r| \cdot \tau(g).
\]
We deduce that $\|[h]\|_{\tC(G,T)}=\tau(g)$.
\end{proof}


We will use Proposition \ref{P:properties Psi iota} to identify $\tD(G,T)$ with $\widehat{\Free_\RR(G;\mu)}$  where $\mu$ is the norm on $G$.
We abusively write $\hat{\pi} \colon \widehat{\Free_\RR(G;\mu)} \to \tC(G,T)$ for the metric quotient homomorphism in paragraph \ref{V:hat pi construction}.

\begin{prop}\label{prop:quotient of tau}
Let $\mu,\tau \colon G \to [0,\infty)$ denote the norm and the stable length function, respectively.
Consider $\Theta \subseteq G$ with the property that for any $g \in G$ there exists $\theta \in \Theta$ and $p \in \ZZ$ such that $g$ is conjugate to $\theta^p$.
Let $\OP{id}^\tau_\mu \colon \Free_\RR(G;\mu) \to \Free_\RR(G;\tau)$ denote the identity homomorphism of the underlying groups.

Then there exists a homomorphism $\hat{\xi} \colon \widehat{\Free_\RR(G;\tau)} \to \tC(G;T,\mu)$ which renders the following diagram commutative.
\[
\xymatrix{
{\Free}_\RR(G;\mu) \ar[r]^\iota \ar[d]_{\OP{id}^\tau_\mu} &
\widehat{\Free_\RR(G;\mu)} \ar[d]_{\widehat{\OP{id}^\tau_\mu}} \ar[r]^{\hat{\pi}} 
&
\tC(G;T,\mu) 
\\
{\Free}_\RR(G;\tau) \ar[r]^\iota  &
\widehat{\Free_\RR(G;\tau)} \ar[ur]^{\hat{\xi}} 
\\
{\Free}_\RR(\Theta;\tau) \ar[r]^\iota  \ar@{^(->}[u]^{\OP{incl}} &
\widehat{\Free_\RR(\Theta;\tau)} \ar@{^(->}[u]^{\widehat{\OP{incl}}} \ar[uur]_{\hat{\xi}_\Theta}
}
\]
Moreover, $\xi_\Theta$ is a metric quotient homomorphism.
\end{prop}

\begin{proof}
Set $\xi =  \hat{\pi} \circ \iota \circ {({\OP{id}^\tau_\mu})^{-1}}$.
By Corollary \ref{cor:Cnorm image of generators in FRG} and Proposition \ref{P:norm-mu universally Lipschitz} $\xi$ is Lipschitz with constant $1$.
By the universal property of the completion and Proposition \ref{P:tC is complete} there exists a unique homomorphism $\hat{\xi}$ as in the diagram, Lipschitz with constant $1$, and rendering the  upper triangle commutative.
Set $\hat{\xi}_\Theta = \hat{\xi} \circ \widehat{\OP{incl}}$.
The diagram is commutative.
It remains to show that $\hat{\xi}_\Theta$ is a metric quotient homomorphism.

Define a homomorphism $\sigma \colon \Free_\RR(G;\tau) \to \Free_\RR(\Theta;\tau)$ as follows.
For a generator $\overline{g}(r)$, choose $\theta \in \Theta$ and $p \in \ZZ$ such that $g$ is conjugate (via some element $h \in G$) to $\theta^p$ and set $\sigma(\overline{g}(r)) = \overline{\theta}(pr)$.
If $g \in \Theta$ then we choose $\theta = g$ and $p=1$.
The relations \eqref{eqn:presentation of FSA} in paragraph \ref{V:presentation FSA} are satisfies by this assignment and it therefore extents to a unique homomorphism $\sigma \colon \Free_\RR(G;\tau) \to \Free_\RR(\Theta;\tau)$.
It is Lipschitz with constant $1$ by Proposition \ref{P:norm-mu universally Lipschitz} because for a generator $\overline{g}(r)$ 
\begin{multline*}
\|\sigma(\overline{g}(r))\|_{\Free_\RR(\Theta;\tau)} = \|\overline{\theta}(pr)\| = |pr| \cdot \tau(\theta) = |r| \cdot \tau(\theta^{p}) 
\\
= |r| \cdot \tau(h\theta^p h^{-1}) 
= |r| \cdot \tau(g) = \|\overline{g}(r)\|_{\Free(G;\tau)}.
\end{multline*}
Then $\sigma$ extends to a homomorphism $\widehat{\sigma} \colon \widehat{\Free_\RR(G;\tau)} \to \widehat{\Free_\RR(\Theta;\tau)}$, Lipschitz with constant $1$.
It is clear that $\sigma \circ \OP{incl}$ is the identity on $\Free_\RR(\Theta;\tau)$, hence $\widehat{\sigma}$ is a left inverse for $\widehat{\OP{incl}}$.
Furthermore, we claim that $\hat{\xi}_{\Theta} \circ \widehat{\sigma}= \hat{\xi}$.
Indeed using the density of $\Free_\RR(G;\tau)$ in its completion, it suffices to check, with the aid of Lemma \ref{L:pihat powers}, that on generators $\overline{g}(r) \in \Free_\RR(G;\tau)$
\begin{multline*}
(\hat{\xi}_\Theta \circ \widehat{\sigma})(\iota(\overline{g}(r))) =
(\hat{\xi}_\Theta \circ \iota)(\sigma(\overline{g}(r))) =
(\hat{\xi}_\Theta \circ \iota)(\overline{\theta}(pr)) 
\\
=
(\hat{\xi} \circ \iota)(\overline{\theta}(pr)) =
(\hat{\pi} \circ \iota)(\overline{\theta}(pr)) 
=
(\hat{\pi} \circ \iota)(\overline{h\theta^p h^{-1}}(r))
\\ 
=
(\hat{\pi} \circ \iota)(\overline{g}(r)) 
=
\hat{\xi}(\iota(\overline{g}(r))).
\end{multline*}
Consider some $[g] \in \tC(G;\mu)$.
Since $\hat{\pi}$ is a metric quotient, given $\epsilon >0$ there exists $\hat{\alpha} \in \widehat{\Free_\RR(G;\mu)}$ such that $\hat{\pi}(\hat{\alpha})=[g]$ and $\|\hat{\alpha}\|_{\widehat{\Free_\RR(G;\mu)}} < \|[g]\| + \epsilon$.
Set $\hat{\alpha}_\Theta = \hat{\sigma}(\widehat{\OP{id}^\tau_\mu}(\hat{\alpha}))$.
Then $\hat{\xi}_\Theta(\hat{\alpha}_\Theta)=\hat{\xi}(\widehat{\OP{id}^\tau_\mu}(\hat{\alpha}))=[g]$.
Since $\hat{\sigma}$ and $\widehat{\OP{id}^\tau_\mu}$ are Lipschitz with constant $1$ 
\[
\| \hat{\alpha}_\Theta \|_{\widehat{\Free_\RR(\Theta;\tau)}} \leq
 \|\hat{\alpha}\|_{\widehat{\Free_\RR(G;\mu)}} < \|[g]\|_{\tC(G)} + \epsilon.
\]
Since $\widehat{\xi}_\Theta$ is Lipschitz with constant $1$ it now follows that it is a metric quotient homomorphism.
\end{proof}

We will now address the naturality of the construction $\tC(G,T)$.
By inspection, any homomorphism $\varphi \colon G \to H$ fits into the commutative diagram
\begin{equation}\label{E:pi is natural transformation}
\xymatrix{
\Free(G) \ar[r]^{\varphi_*} \ar[d]_{\pi} &
\Free(H) \ar[d]^\pi 
\\
G \ar[r]^\varphi &
H.
}
\end{equation}

\begin{prop}[Naturality] \label{P:naturality C and tC}
The assignments $G \mapsto \C C(G,T)$ and $G \mapsto \tC(G,T)$ give rise to functors on the category of groups with conjugation-invariant (pseudo) norm and Lipschitz homomorphisms between them.
Moreover, in relation to the naturality of $\C D(-,T)$ and $\tD(-,T)$ in Proposition \ref{P:naturality D and tD}, the homomorphisms $\pi \colon \C D(-,T) \to \C C(-,T)$ and $\hat{\pi} \colon \tD(-,T) \to \tC(-,T)$ are natural transformations of functors.

In more detail, any Lipschitz homomorphism $\varphi \colon G \to H$ between groups equipped with conjugation-invariant pseudonorm gives rise to Lipschitz homomorphisms 
\begin{align*}
\C C(\varphi) \colon \C C(G,T) \to \C C(H,T) & & \text{defined by } \C C(\varphi)(g) = \varphi \circ g 
\\
\tC(\varphi) \colon \tC(G,T) \to \tC(H,T) & & \text{defined by } \tC(\varphi)([g]) = [\varphi \circ g]
\end{align*}
%
They render the following squares commutative
\[
\xymatrix{
\C D(G,T) \ar[r]^{\C D(\varphi)} \ar@{->>}[d]_{\pi_*}
&
\C D(H,T) \ar@{->>}[d]^{\pi_*}
\\
\C C(G,T) \ar[r]^{\C C(\varphi)} &
\C C(H,T)
}
\qquad 
\xymatrix{
\tD(G,T) \ar[r]^{\tD(\varphi)} \ar@{->>}[d]_{\hat{\pi }}
&
\tD(H,T) \ar@{->>}[d]^{\hat{\pi}}
\\
\tC(G,T) \ar[r]^{\tC(\varphi)} &
\tC(H,T).
}
\]
We abusively denote both $\C C(\varphi)$ and $\tC(\varphi)$ by $\varphi_*$.
Then $\OP{id}_*=\OP{id}$ and $(\psi \circ \vp)_*=\psi_* \circ \vp_*$.
\end{prop}

\begin{proof}
We will omit $T$ from the notation.
Let $K(G)$ denote the kernel of $\pi \colon \C D(G) \to \C C(G)$.
Similarly, $\hat{K}(G)$ denotes the kernel of $\hat{\pi} \colon \tD(G) \to \tC(G)$.
By definition, if $w \in K(G)$ then $\pi \circ w$ is the constant function $1 \colon T \to G$.

Suppose that $\varphi \colon G \to H$ is Lipschitz with constant $C$.
By Proposition \ref{P:naturality D and tD} there are Lipschitz homomorphisms $\C D(\varphi) \colon \C D(G) \xto{w \mapsto \varphi \circ w} \C D(H)$ and $\tD(\varphi) \colon \tD(G) \xto{[w] \mapsto [\varphi \circ w]} \tD(H)$.
By the commutativity of \eqref{E:pi is natural transformation}, for any $w \in K(G)$
\[
\pi_* (\C D(\varphi)(w)) = \pi \circ \C D(\varphi)(w) = \pi \circ \varphi_* \circ w   = \varphi \circ \pi \circ w = 1.
\]
So $\C D(\varphi)(w) \in K(H)$ and therefore $\C D(\varphi)(K(G)) \subseteq K(H)$.

If $[w] \in \hat{K}(G)$ then $\pi \circ w$ represents the trivial element in $\tC(G)$ so by Lemma \ref{L:characterize norm 0 in C}, $\lim_{t \in T} \tfrac{1}{t} \| \pi(w(t))\|=0$.
By the commutativity of \eqref{E:pi is natural transformation}
\[
\hat{\pi}(\tD(\varphi)([w])) = 
\hat{\pi}([\varphi_* \circ w]) =
[\pi \circ \varphi_* \circ w] =
[\varphi \circ \pi \circ w].
\]
Since $\varphi$ is Lipschitz, $\lim_{t \in T} \tfrac{1}{t} \| \varphi \circ \pi \circ w\| =0$. 
By Lemma \ref{L:characterize norm 0 in C} $\hat{\pi}(\tD(\varphi)([w]))$ is the trivial element in $\tC(H)$.
It follows that $\tD(\varphi)(\hat{K}(G)) \subseteq \hat{K}(H)$.

As a result there exist unique homomorphisms $\C C(\varphi) \colon\C C(G) \to \C C(H)$ and $\tC(\varphi) \colon \tC(G) \to \tC(H)$ rendering the  squares in the statement of the proposition commutative.
By Proposition \ref{P:quotient Lipschitz maps} these are Lipschitz homomorphisms with constant $C$.
The functorial properties follows from the surjectivity of $\pi_*$ and $\hat{\pi}$ which imply that $\C C(\varphi)$ and $\tC(\varphi)$ filling in the above diagrams commutative are unique.
\end{proof}

We now address the naturality of the construction $\tC(G,T)$ with respect to $T$.

\begin{prop}\label{P:naturality tC scaling sets}
An inclusion  $T \subseteq T'$ of scaling sets gives rise to natural transformations of functors
\begin{align*}
\res^{T'}_T \colon \C C(-,T') \to \C C(-,T)    & ,&  &\res^{T'}_T(g) = g|_T \\
\widehat{\res}^{T'}_T \colon \tC(-,T') \to \tC(-,T)      & ,& & \hat{\res}^{T'}_T([g]) = [g|_T]
\end{align*}
Moreover, $\widehat{res}^{T'}_T \colon \tC(G,T') \to \tC(G,T)$ is a metric quotient.
\end{prop}

\begin{proof}
Fix $G$.
For any scaling set $S$ and let $K(S)$ denote the kernel of $\pi_* \colon \C D(G,S) \to \C C(G,S)$.

First, we show that there are homomorphisms $\res^{T'}_T \colon \C C(G,T') \to \C C(G,T)$ and $\widehat{\res}^{T'}_T \colon \tC(G,T') \to \tC(G,T)$ rendering the following squares commutative where the arrows in the first rows are defined in Proposition \ref{P:naturality tDA at T}.
\[
\xymatrix{
\C D(G,T') \ar[r]^{\res^{T'}_T} \ar@{->>}[d]_{\pi_*} &
\C D(G,T) \ar@{->>}[d]^{\pi_*} 
\\
\C C(G,T') \ar@{.>}[r]^{\res^{T'}_T} &
\C C(G,T)
}
\qquad
\xymatrix{
\tD(G,T') \ar[r]^{\widehat{\res}^{T'}_T} \ar@{->>}[d]_{\hat{\pi}} &
\tD(G,T) \ar@{->>}[d]^{\hat{\pi}} 
\\
\tC(G,T') \ar@{.>}[r]^{\widehat{\res}^{T'}_T} &
\tC(G,T)
}
\]
We consider the first diagram.
Choose some $w \in K(T')$.
Then $\pi_*(w) = \pi \circ w$ is the constant function $1 \colon T' \to \Free(G)$.
Therefore, $\pi_*(\res^{T'}_T(w)) = \pi \circ w|_T$ is the constant function $1 \colon T \to \Free(G)$, so $\res^{T'}_T$ in the first row carries $K(T')$ into $K(T)$.
Since $\pi_*$ is surjective, the broken arrow $\res^{T'}_T$ rendering the square commutative exists and is unique.
Its uniqueness and the fact that $\res^{T'}_T$ in the first row is a natural transformation, shows that the same holds for $\res^{T'}_T$ at the bottom.
By inspection, $\res^{T'}_T(g) = g|_T$.

We now consider the second square.
Choose some $[w] \in \hat{K}(T')$.
Notice that
\[
\hat{\pi}(\widehat{\res}^{T'}_T([w])) = \hat{\pi}([w|_T) = [\pi \circ w|_T]. 
\]
Since $\hat{\pi}([w])=[\pi \circ w]$ is the trivial element in $\tC(G,T')$, Lemma \ref{L:characterize norm 0 in C} implies 
\[
0 = \lim_{t \in T'} \tfrac{1}{t}\| \pi \circ w(t)\|= \lim_{t \in T} \tfrac{1}{t}\| \pi \circ w|_T(t)\|.
\]
It follows from the same lemma that $\hat{\pi}(\widehat{\res}^{T'}_T([w]))$ is the trivial element in $\tC(G,T)$.
Therefore $\res^{T'}_T$ at the first arrow carries $\hat{K}(T')$ into $\hat{K}(T)$.
Since $\hat{\pi}$ is surjective the broken arrow $\widehat{\res}^{T'}_T$ rendering the square commutative exists and is unique.
The same argument above shows that It is a natural transformation of functors and that it is given by the formula $\widehat{\res}^{T'}_T([g])=[g|_T]$.
Proposition \ref{P:naturality tDA at T} shows that $\widehat{\res^{T'}_T}$ in the first row is an isometry.
Since $\hat{\pi}$ is a metric quotient (see Section \ref{V:hat pi construction}), Proposition \ref{P:metric quotient 2 out of 3} shows that the broken arrow $\widehat{\res^{T'}_T}$ is a metric quotient too.
\end{proof}

\begin{prop}\label{P:restriction injective => isometry} 
The homomorphism $\widehat{\res}^{T'}_T \colon \tC(G,T') \to \tC(G,T)$ is an isometry if and only if it is injective.
\end{prop}

\begin{proof}
$\tC(G,T)$ and $\tC(G,T')$ are metric groups (not just pseudometric groups).
\end{proof}


\section{Metric and topological properties of the directional asymptotic cone}

\noindent
{\bf Running assumptions:}
All groups $G,H,\dots$ are equipped with a conjugation-invariant norm.
Throughout $T,T',\dots$ denote scaling sets.

\begin{void}
By Proposition \ref{P:tC is complete} $\tC(G,T)$ is a complete metric space.
\end{void}

\begin{prop}\label{P:polish cone}
If $G$ is countable then $\tC(G,T)$ is a separable.
\end{prop}

\begin{proof}
Clearly $\Free_\QQ(G)$ is countable, so $\tD(G,T)$ is separable by Proposition \ref{prop:FSA dense FRA} and Proposition \ref{P:properties Psi iota}.
Since $\hat{\pi}$ is a metric quotient, see \ref{V:hat pi construction}, it follows that $\tC(G,T)$ is separable.
\end{proof}

\begin{lemma}\label{L:quotient length spaces}
Suppose that $G$ is a group equipped with a conjugation-invariant pseudonorm and that $\pi \colon G \to H$ is a metric quotient.
If $G$ is a length space then so is $H$.
\end{lemma}

\begin{proof}
The norms are invariant to translation so it suffices to prove that for any $h \in H$ and $\epsilon >0$ there exists \ path $\gamma \colon I \to H$ from $1$ to $h$ such that $\ell(\gamma)<\|h\|+\epsilon$.
Choose $g \in G$ such that $\|g\| < \|h\|+\tfrac{\epsilon}{2}$ and a path $\beta \colon I \to G$ from $1$ to $g$ with $\ell(\beta) < \|g\|+\tfrac{\epsilon}{2}$.
Set $\gamma = \pi \circ \beta$.
Since $\pi$ is Lipschitz with constant $1$ we get $\ell(\gamma) \leq \ell(\beta)<\|h\|+\epsilon$.
\end{proof}

\begin{cor}\label{C:tC length space} 
The directional asymptotic cone $\tC(G,T)$ is a length space. 
\end{cor}

\begin{proof}
By Section \ref{V:hat pi construction} $\hat{\pi} \colon \tD(G,T) \to \tC(G,T)$ is a metric quotient.
Apply Corollary \ref{C:tD is length space} and Lemma \ref{L:quotient length spaces}.
\end{proof}

Our next goal is to prove that under some conditions on the scaling set, the directional asymptotic cone is a contractible group.

\begin{defn}\label{D:ample scaling set} 
A scaling set $T$ is called {\bf ample} if for every $s >0$ there exists a function $\sigma \colon T \to T$ such that $\lim_{t \in T} \tfrac{\sigma(t)}{t}=s$.
\end{defn}

\begin{example}\label{example:ample sets}
\begin{enumerate}[label=(\roman*),leftmargin=*]
\item The maximal scaling set $T=(0,\infty)$ is ample.
In this case $\sigma(t)=st$.
\item $T=\{1,2,3,\cdots\}$ is ample.
In this case $\sigma(t)=\lfloor st \rfloor$.
More generally, if $T$ is a non-constant arithmetic sequence of positive integers, then it is ample.
\item $T=\{ 2^n : n=0,1,2,\dots\}$ is not ample.
\end{enumerate}
\end{example}

The following lemma is an exercise in Calculus and is left to the reader.

\begin{lemma}\label{L:limsup domination}
Let $T$ be a scaling set, $f \colon T \to \RR$ a (bounded) function and $\sigma \colon T \to T$ a function such that $\lim_{t \in T} \sigma(t)=\infty$.
Then
\[
\limsup_{t \in T} f(\sigma(t)) \leq \limsup_{t \in T} f(t).
\]
\end{lemma}

Recall the homomorphism $\iota \colon \Free_\RR(G) \to \tD(G,T)$ from Definition \ref{D:iota FR(A)-->tD(A)} and Proposition \ref{P:properties Psi iota}.

\begin{lemma}\label{L:h-t preserves kerel limsup lemma}
Suppose that $T$ is an ample scaling set.
For $\alpha \in \Free_\RR(G)$ and $s >0$ set $[g]=\hat{\pi}(\iota(\alpha))$ and $[h]=\hat{\pi}(\iota(\psi_s(\alpha)))$.
Then
\[
\limsup_{t \in T} \tfrac{1}{t} \|h(t)\|_G \leq s \cdot \|[g]\|_{\tC(G,T)}.
\]
\end{lemma}

\begin{proof}
Suppose that $\alpha = \overline{g_1}(r_1) \cdots \overline{g_k}(r_k)$.
By the definitions of $\psi_s$, of $\iota$ and of $\hat{\pi}$, for any $t \in T$
\[
g(t) = g_1^{\Myfloor{r_1t}} \cdots g_n^{\Myfloor{r_nt}} \qquad \text{and} \qquad h(t) = g_1^{\Myfloor{r_1st}} \cdots g_n^{\Myfloor{r_nst}}.
\]
Let $\sigma \colon T \to T$ be a function such that $\lim_{t \in T} \tfrac{\sigma(t)}{t}=s$.
Since $s>0$ it is clear that $\lim_{t \in T} \sigma(t)=\infty$.
Let $k \colon T \to G$ be the function $k(t)=g_1^{\Myfloor{r_1\sigma(t)}} \cdots g_n^{\Myfloor{r_n\sigma(t)}}$.
By Lemma \ref{L:conj invariant consequences}, for any $t \in T$
\[
\tfrac{1}{t}\| h(t)^{-1} k(t)\| \leq
\sum_{i=1}^n \|g_i\| \cdot \left| \tfrac{\Myfloor{r_ist}}{t} - \tfrac{\Myfloor{r_i\sigma(t)}}{t} \right| \xto{t \to \infty} 0.
\]
Since $\sigma(t) \in T$ for all $t \in T$, we use Lemmas \ref{L:limsup domination} and \ref{L:limsup<norm} to estimate:
\begin{align*}
\limsup_{t \in T} \tfrac{1}{t} \|h(t)\| &=
\limsup_{t \in T} \tfrac{1}{t} \|k(t)\| 
\\
&=
\limsup_{t \in T} \tfrac{\sigma(t)}{t} \cdot \tfrac{1}{\sigma(t)} \|g_1^{\Myfloor{r_1\sigma(t)}} \cdots g_n^{\Myfloor{r_n\sigma(t)}}\|
\\
& \leq
s \cdot \limsup_{t \in T} \tfrac{1}{t} \|g_1^{\Myfloor{r_1t}} \cdots g_n^{\Myfloor{r_nt}}\|
\\
& \leq
s \cdot \| [g] \|_{\tC(G,T)}.
\end{align*}
\end{proof}

Recall the continuous map $h \colon [0,\infty) \times \tD(G,T) \to \tD(G,T)$ in Proposition \ref{P:extend psi_t to D}.


\begin{lemma}\label{L:h-t preserves kerel}
Suppose that $T$ is ample.
Let $K$ denote the kernel of $\hat{\pi} \colon \tD(G,T) \to \tC(G,T)$.
Then $h_s(K) \leq K$ for any $s \geq 0$.
\end{lemma}

\begin{proof}
We suppress $T$ from the notation.
When $s=0$ there is nothing to prove since $h_s$ is the trivial homomorphism. 
So we assume $s>0$.

Consider some $[u] \in K$ represented by $u \in \C D(G)$.
By Proposition \ref{P:properties Psi iota} the image of $\iota$ is dense in $\tD(G)$.
Choose a sequence $\alpha_k \in \Free_\RR(G)$ such that $\lim_k \iota(\alpha_k)=[u]$.
Since $\hat{\pi}$ is continuous,
\[
\lim_k \hat{\pi}(\iota(\alpha_k))=\hat{\pi}([u])=1.
\] 
(The limit is in $\tC(G)$).
Denote $[w]=h_s([u])$ for some $w \in \C D(G)$.
Our goal is to show that $\hat{\pi}([w])$ is the trivial element in $\tC(G)$.

Denote $[v_k]=\iota(\psi_s(\alpha_k))$ for some $v_k \in \C D(G)$.
Then by Proposition \ref{P:extend psi_t to D}
\[
[w]=h_s([u])=h_s(\lim_k \iota(\alpha_k)) = \lim_k \iota(\psi_s(\alpha_k))= \lim_k [v_k].
\]
(the limit is in $\tD(G)$).
Since $\hat{\pi}$ is continuous, 
\[
\lim_k \hat{\pi}([v_k]) = \hat{\pi}([w]).
\]
(The limit is in $\tC(G)$).
By definition $\hat{\pi}([w])=[\pi \circ w]$.
By the triangle inequality and the subadditivity of $\limsup$, for any $k \geq 1$
\[
\limsup_{t \in T} \tfrac{1}{t} \| \pi(w(t))\| \leq
\limsup_{t \in T} \tfrac{1}{t}\| \pi(v_k(t)^{-1}w(t))\| + \limsup_{t \in T} \tfrac{1}{t} \|\pi(v_k(t))\|.
\]
By definition $\hat{\pi}([v_k]^{-1}[w])$ is represented by the function $\pi(v_k(t)^{-1}w(t))$ so by Lemma \ref{L:limsup<norm} the first term on the right hand side is estimated by
\[
\limsup_{t \in T} \tfrac{1}{t}\| \pi(v_k(t)^{-1}w(t))\| \leq \|\hat{\pi}([v_k]^{-1}[w])\|_{\tC(G)} \xto{k \to \infty} 0.
\]
For the second term, notice that $\pi \circ v_k$ represents $\hat{\pi}([v_k]) = \hat{\pi}(\iota(\psi_s(\alpha_k)))$ so Lemma \ref{L:h-t preserves kerel limsup lemma} yields
\[
\limsup_{t \in T} \tfrac{1}{t} \|\pi(v_k(t))\| \leq s \cdot \| \hat{\pi}(\iota(\alpha_k))\|_{\tC(G)} \xto{k \to \infty} 0.
\]
Hence, $\limsup_{t \in T} \tfrac{1}{t} \| \pi(w(t))\| =0$ and  Lemma \ref{L:characterize norm 0 in C} shows that $\hat{\pi}([w])=1$.
\end{proof}

\begin{theorem}\label{T:tC contractible}
Suppose that $T$ is ample.
There exists a continuous function $\hat{h} \colon [0,\infty) \times \tC(G,T) \to \tC(G,T)$ rendering the following diagram commutative
\[
\xymatrix{
[0,\infty) \times \tD(G,T) \ar[r]^(0.6){h} \ar[d]_{[0,\infty) \times \hat{\pi}} &
\tD(G,T) \ar[d]^{\hat{\pi}}
\\
[0,\infty) \times \tC(G,T) \ar[r]_(0.6){\hat{h}} & 
\tC(G,T).
}
\]
It has the property that $\hat{h}_t \colon \tC(G,T) \to \tC(G,T)$ are homomorphisms such that $\|\hat{h}_t(x)\|_{\tC(G)} = t \cdot \|x\|_{\tC(G)}$ and $\hat{h}_1=\OP{id}$ and $\hat{h}_0$ is the trivial homomorphism and $\hat{h}_t \circ \hat{h}_s = \hat{h}_{ts}$.
In particular, $\tC(G,T)$ is contractible and $\OP{End}(\tC(G,T))$, see Definition \ref{def:End(G)}, is contractible.
\end{theorem}

\begin{proof}
We suppress $T$ from the notation.
By Lemma \ref{L:h-t preserves kerel}, for every $s \geq 0$ there is a homomorphism $\hat{h}_s \colon \tC(G) \to \tC(G)$ such that the following square commutes
\[
\xymatrix{
\tD(G) \ar[r]^{h_s} \ar[d]_{\hat{\pi}} &
\tD(G) \ar[d]^{\hat{\pi}} 
\\
\tC(G) \ar[r]_{\hat{h}_s} &
\tC(G).
}
\]
We obtain a function $\hat{h}$ as in the statement of the theorem which renders the square in its statement commutative.
By Proposition \ref{P:extend psi_t to D} $\|h_t(x)\|_{\tD(G)}=t \cdot \|x\|_{\tD(G)}$ for all $x \in \tD(G)$.
Since $\hat{\pi}$ is a metric quotient, it follows that $\|\hat{h}_t(x)\|_{\tC(G)} \leq t \cdot \|x\|_{\tC(G)}$ for all $x \in \tC(G)$, namely $\hat{h}_t$ is Lipschitz with constant $t$.
This also shows that $\hat{h}_x \colon [0,\infty) \to [0,\infty)$ is Lipschitz for any $x \in \tC(G)$, hence continuous.
Apply Lemma \ref{L:extend homotpy from dense subspace} with $A=X=Y=\tC(G)$ and $f=\hat{h}$ to deduce that $\hat{h}$ is continuous.

Since $\hat{\pi}$ is surjective, $\hat{h}_t$ rendering the diagram commutative is unique and therefore $\hat{h}_1=\OP{id}$ and $\hat{h}_0$ is trivial and $\hat{h}_t \circ \hat{h}_s=\hat{h}_{ts}$, because $h_t$ has these properties.
In particular $\hat{h}$ gives a homotopy from $\hat{h}_1=\OP{id}_{\tC(G)}$ and the trivial homomorphism $\hat{h}_0$, so $\tC(G)$ is contractible.
Since $\hat{h}_t$ are homomorphism, this contraction is a homotopy in the category of groups, so $\OP{End}(\tC(G))$ with the compact-open topology is contractible as well.
\end{proof}

\section{Algebraic properties of the directional asymptotic cone}

\noindent
{\bf Running assumptions:}
All group $G,H,\dots$ are equipped with a conjugation-invariant norm.
Throughout $T,T',\dots$ denote scaling sets.

\begin{prop}\label{P:cone of abelian is abelian}
If $G$ is abelian then $\tC(G,T)$ is abelian.
More generally, if $G$ is nilpotent of class $n$ then so is $\tC(G,T)$.
\end{prop}

\begin{proof}
The product in $\C C(G,T)$ is induced by the point-wise product of functions $f \colon T \to G$ and $\tC(G,T)$ is a quotient.
\end{proof}

Recall that $g \in G$ is called {\bf distorted} if its stable length vanishes, i.e., $\tau(g)=0$, see paragraph \ref{V:stable length}.

\begin{prop}\label{P:distorted => trivial cone}
If $G$ is bounded then $\tC(G,T)=1$.
More generally, if every element in $G$ is distorted then $\tC(G,T)=1$
(cf.  Example \ref{E:symmetric8}).
\end{prop}

\begin{proof}
Recall that $\hat{\pi} \circ \iota \colon \Free_\RR(G) \to \tC(G,T)$ is Lipschitz with dense image, see Proposition \ref{P:properties Psi iota} and paragraph \ref{V:hat pi construction}.
It remains to show its image is trivial, which is the case if the image of every generator $\alpha=\overline{g}(r) \in \Free_\RR(G)$ is trivial.
Now $[h]:=\hat{\pi}(\iota(\alpha))$ is represented by $h \colon T \to G$ where $h(t)=g^{\Myfloor{rt}}$.
Then
\[
\limsup_{t \in T} \tfrac{1}{t}\|h(t)\| = 
\limsup_{t \in T} \tfrac{|\Myfloor{rt}|}{t} \cdot \tfrac{1}{|\Myfloor{rt}|} \| g^{\Myfloor{rt}}\| =
|r| \cdot \tau(g) = 0
\]
because $T$ is unbounded.
Lemma \ref{L:characterize norm 0 in C} implies that $[h]=1$. 
\end{proof}

Let $G$ and $H$ be (pseudo) normed groups.
The associated $L^1$ (pseudo) norm on $G \times H$ is defined by $\|g \times h\|_{G \times H} := \|g\|_G+\|h\|_H$.
Clearly if $\| \ \|_G$ and $\| \ \|_H$ are conjugation-invariant then so is $\| \ \|_{G \times H}$.

\begin{theorem}\label{T:tC products}
Equip $G \times H$ with the associated $L^1$ norm.
Then for any scaling set $T$ there is a natural isometry of groups
\[
\tC(G \times H, T) \xrightarrow[\cong]{ \ [g \times h] \mapsto [g] \times [h] \ } \tC(G,T) \times \tC(H,T)
\]
where the codomain is equipped with the associated $L^1$ norm.
\end{theorem}

\begin{proof}
Let $\lambda^G \colon G \times H \to G$ and $\lambda^H \colon G \times H \to H$ denote the canonical projections.
Let $i^G \colon G \to G \times H$ and $i^H \colon H \to G \times H$ be the canonical inclusions.
They are clearly Lipschitz with constant $1$.
By Proposition \ref{P:naturality C and tC} they induce homomorphisms  $\lambda^G_*:=\tC(\lambda^G)$ and $\lambda^H_*:=\tC(\lambda^H)$ and $i^G_*:=\tC(i^G)$ and $i^H_*:=\tC(i^H)$, all are Lipschitz with constant $1$.
By the description of $i^G_*$ and $i^H_*$ in Proposition \ref{P:naturality C and tC} and since $G$ and $H$ commute in $G \times H$, it follows that their images in $\tC(G \times H)$ commute.
There results a homomorphism
\[
i \colon \tC(G) \times \tC(H) \to \tC(G \times H), \qquad
i([g],[h]) = i_G^*([g]) \cdot i^H_*([h]).
\]
Then $i$ is Lipschitz with constant $1$ because the metric on the domain is the $L^1$ metric so
\begin{multline*}
\|i([g],[h])\|_{\tC(G \times H)} = \|i^G_*([g]) \cdot i^H_*([h])\| \leq \|i^G_*([g])\|+\|i^H_*([h])\| 
 \\
\leq \|[g]\|+\|[h]\|= \|([g],[h])\|_{\tC(G) \times \tC(H)}.
\end{multline*}
Use $\lambda^G_*$ and $\lambda^H_*$ to define a homomorphism
\[
\Lambda= \colon \tC(G \times H) \xto{(\lambda^G_*,\lambda^H_*)} \tC(G) \times \tC(H).
\]
By the description of the homomorphisms in Proposition \ref{P:naturality C and tC} it is clear that $\Lambda$ and $i$ are inverses of each other because for any $[g] \in \tC(G)$ and $[h] \in \tC(H)$ represented by $g \colon T \to G$ and $h \colon T \to H$,
\[
\Lambda(i([g(t)],[h(t)])) = \Lambda([g(t)\times 1] \cdot [1 \times h(t)]) = ([g],[h]).
\]
This shows that $\Lambda \circ i$ is the identity.
Similarly, since any function $f \colon T \to G \times H$ is a pair of functions $g(t)$ and $h(t)$ it follows that $i \circ \Lambda$ is the identity.

The homomorphisms $\lambda^G, \lambda^H$ give rise to a homomorphism
\[
\lambda \colon \Free_\RR(G \times H) \to \Free_\RR(G) \times \Free_\RR(H)
\]
which on generators has the effect $\lambda(\overline{g \times h}(r)) = (\overline{g}(r),\overline{h}(r))$.
One easily checks that this assignment respects the relations \eqref{eqn:presentation of FSA} in Paragraph \ref{V:presentation FSA}, hence $\lambda$ is well defined.
It is Lipschitz with constant $1$ by Proposition \ref{P:norm-mu universally Lipschitz} because for a generators we get
\begin{multline*}
\| \lambda(\overline{g \times h}(r))\| = \|(\overline{g}(r),\overline{h}(r))\| = \| \overline{g}(r)\|+\|\overline{h}(r)\| 
\\
=
|r| \cdot (\|g\|+\|h\|) = |r| \cdot \|(g,h)\| = \| \overline{g \times h}(r)\|.
\end{multline*}
By checking generators of $\Free_\RR(G \times H)$ and using that $G,H$ commute in $G \times H$, it follows that the following square commutes
\[
\xymatrix{
\Free_\RR(G \times H) \ar[r]^{\lambda} \ar[d]_{\hat{\pi} \circ \iota} &
\Free_\RR(G) \times \Free_\RR(H) \ar[d]^{(\hat{\pi} \circ \iota, \hat{\pi} \circ \iota)}
\\
\tC(G \times H) \ar[r]_\Lambda & 
\tC(G) \times \tC(H).
}
\]
Upon completion we get a homomorphism $\widehat{\lambda} \colon \widehat{\Free_\RR(G \times H)} \to \widehat{\Free_\RR(G)} \times \widehat{\Free_\RR(H)}$, Lipschitz with constant $1$.
By Proposition \ref{P:properties Psi iota} and Paragraph \ref{V:hat pi construction} we get the commutative diagram
\[
\xymatrix{
\widehat{\Free_\RR(G \times H)} \ar[r]^{\widehat{\lambda}} \ar[d]_{\hat{\pi}} &
\widehat{\Free_\RR(G)} \times \widehat{\Free_\RR(H)} \ar[d]^{(\hat{\pi}, \hat{\pi})}
\\
\tC(G \times H) \ar[r]_\Lambda & 
\tC(G) \times \tC(H).
}
\]
where the vertical arrows are metric quotient homomorphisms.
Then $\Lambda$ is Lipschitz with constant $1$ by Proposition \ref{P:quotient Lipschitz maps}.
Since its inverse $i$ is also Lipschitz with constant $1$ both of them are isometries.
\end{proof}

\begin{theorem}\label{T:tC quotient with section}
Let $f \colon G \to H$ be a surjective homomorphism, Lipschitz with constant $C$.
Assume further that
\begin{enumerate}[label=(\alph*)]
\item
$\OP{Ker}(f)$ is a bounded subgroup of $G$.
\item
$f$ admits a set-theoretic section $s \colon H \to G$ (not necessarily a homomorphism) such that there exists $C'$ such that $\|s(h)\|_G \leq C' \|h\|$ for all $h \in H$.
\end{enumerate}
Then $\tC(f,T) \colon \tC(G,T) \to \tC(H,T)$ is an isomorphism, Lipschitz with constant $C$, whose inverse is the homomorphism
\[
s_! \colon \tC(H,T) \xto{ [h] \mapsto [s \circ h] } \tC(G,T)
\]
which is Lipschitz with constant $C'$.
\end{theorem}

\begin{proof}
We will suppress $T$ from the notation and write $K$ for $\OP{Ker}(f)$.

First, if $k \colon T \to G$ is a function with image in $K$ then by Lemma \ref{L:characterize norm 0 in C}  $k \in \C C(G)$ and $[k]=1 \in \tC(G)$ because
\[
\limsup_{t \in T} \tfrac{1}{t} \|k(t)\| \leq \limsup_{t \in T} \, \tfrac{\OP{diam}_G(K)}{t} = 0.
\]
Therefore, if $g \colon T \to G$ is in $\C C(G)$ then $g \cdot k \in \C C(G)$ and $[g]=[g \cdot k]$ in $\tC(G)$.

By Proposition \ref{P:naturality FRA} $f$ and $s$ induce homomorphisms $f_* \colon \Free(G) \to \Free(H)$ and $s_* \colon \Free(H) \to \Free(G)$, Lipschitz with constants $C$ and $C'$.
Recall that $\pi^G \colon \Free(G) \to G$ and $\pi^H \colon \Free(H) \to H$ are the canonical homomorphisms.
Since $f \circ s=\OP{Id}$ it follows that $f_* \circ s_*=\OP{Id}$, so \eqref{E:pi is natural transformation} implies that for any $w \in \Free(H)$
\[
f(\pi^G(s_*(w))) = \pi^H(f_*(s_*(w)) = \pi^H(w) = f(s(\pi^H(w))).
\]
We deduce that for every $w \in \Free(H)$
\[
\text{$\pi^G(s_*(w))$ and $s(\pi^H(w))$ differ by an element of $K$.}
\]
By Proposition \ref{P:naturality C and tC} $f$ induces $\tC(f) \colon \tC(G) \to \tC(H)$, Lipschitz with constant $C$ given by $\tC(f)([g])=[f \circ g]$.
A-priori a similar $\tC(s)$ does not exit since $s$ is not a homomorphism.
By Proposition \ref{P:naturality D and tD} $s$ induces  $\tD(s) \colon \tD(H) \to \tD(G)$, Lipschitz with constant $C'$, given by $\tD(s)([w])=[s_* \circ w]$.
We obtain the following solid diagram
\[
\xymatrix{
\tD(H) \ar[r]^{\tD(s)} \ar@{->>}[d]_{\hat{\pi}^H} &
\tD(G) \ar@{->>}[d]^{\hat{\pi}^G} 
\\
\tC(H) \ar@{..>}[r]^{s_!} &
\tC(G).
}
\]
Our next goal is to prove that there is a homomorphism $s_!$ of the form given in the statement of the theorem that renders the square commutative.
Choose some $[h] \in \tC(H)$ represented by $h \colon T \to H$ in $\C C(H)$.
Choose some $w \in \C D(H)$ which lifts $h$, namely $w \colon T \to \Free(H)$ is in $\C D(H)$ and $h= \pi^H_*(w) = \pi^H \circ w$.
Set
\[
s_!([h]) \, \overset{\text{def}}{=} \, \hat{\pi}^G(\tD(s)([w]))=\hat{\pi}^G([s_* \circ w])=[\pi^G \circ s_* \circ w].
\]
We have observed above that $\pi^G \circ s_* \circ w$ and $s \circ \pi^H \circ w$ differ by some $k \colon T \to G$ with values in $K$ and therefore they represent the same element in $\tC(G)$, i.e.,
\[
s_!([h]) = [s \circ \pi^H \circ w] = [s \circ h].
\]
In particular $s_!([h])$ is independent of the choice of $w$.
In order to prove that $s_!$ is independent of the choice of representative $h \colon T \to H$, observe that 
\[
\text{$s(h_1)s(h_2)$ and $s(h_1h_2)$ differ by an element of $K$ for any $h_1,h_2 \in H$.}
\]
If $h' \colon T \to H$ and $h \colon T \to H$ represent the same element in $\tC(H)$ then $h'=h \cdot \delta$ for some $\delta \colon T \to H$ which represents the trivial element in $\tC(H)$.
By Lemma \ref{L:characterize norm 0 in C} this is equivalent to $\limsup_{t \in T} \tfrac{\|\delta(t)\|}{t}=0$.
Then for any $t \in T$
\[
s(h'(t))=s(h(t)\cdot \delta(t)) = s(h(t)) \cdot s(\delta(t)) \cdot k(t),
\]
for some $k \colon T \to K \subseteq G$.
We have seen that $[k]=1$.
Since $s$ is proto-Lipschitz, $\limsup_{t \in T} \tfrac{\|s(\delta(t))\|}{t} \leq \limsup_{t \in T} C'\tfrac{\|\delta(t)\|}{t} =0$, hence $[s \circ \delta]=1$.
We deduce that $[s \circ h']=[s \circ h]$, so $s_!$ is independent of representatives, and is therefore well defined.
It is a homomorphism because given $[h_1],[h_2] \in \tC(H)$ observe that for any $t \in T$
\[
s(h_1(t) \cdot h_2(t)) = s(h_1(t)) \cdot s(h_2(t)) \cdot k(t),
\]
for some $k \colon T \to K \subseteq G$, so 
\[
s_!([h_1]\cdot [h_2]) = [s \circ (h_1 \cdot h_2)] = [(s \circ h_1) \cdot (s \circ h_2) \cdot k] = s_!([h_1]) \cdot s_!([h_2]).
\]
Clearly, by its definition $s_!$ renders the square above commutative.
Since $\hat{\pi}^H$ and $\hat{\pi}^G$ are metric quotient maps and $\tD(s)$ is Lipschitz with constant $C'$, Proposition \ref{P:quotient Lipschitz maps} shows that $s_!$ is Lipschitz with constant $C'$.

Since $f \circ s=\OP{Id}$, for any $[h] \in \tC(H)$
\[
\tC(f) (s_!([h])) = [f \circ s \circ h] = [h].
\]
Hence, $\tC(f) \circ s_!=\OP{Id}$ and in particular $s_!$ injective.
Given $[g] \in \tC(G)$ set $[h]=\tC(f)([g])=[f \circ g]$.
Since $s$ is a section, $s \circ f \circ g$ differs from $g$ by some $k \colon T \to K \subseteq G$ and in particular $[s \circ f \circ g]=[g]$.
Then
\[
s_!([h]) = [s \circ f \circ g] = [g].
\]
This shows that $s_!$ is surjective.
Hence it is an isomorphism and the same is true for $\tC(f)$.
\end{proof}

\begin{cor}\label{cor:metric quotient bounded kernel}
Let $\pi \colon G \to H$ be a metric quotient with bounded kernel.
Then $\tC(\pi) \colon \tC(G,T) \to \tC(H,T)$ is an isometry. 
\end{cor}

\begin{proof}
Since $\pi$ is a metric quotient (Definition \ref{VD:metric quotient}) it is Lipschitz with constant $1$ and, given $\epsilon>0$ one can choose a section $s \colon H \to G$ such that $\| h\|_H \leq \|s(h)\|_G \leq \|h\|_H (1+\epsilon)$.
 By Theorem \ref{T:tC quotient with section} $\tC(\pi) \colon \tC(G,T) \to \tC(H,T)$ is bijective, Lipschitz with constant $1$ and its inverse $s_!$ is Lipschitz with constant $(1+\epsilon)$.
Since $\epsilon$ was arbitrary, $s_!$ is Lipschitz with constant $1$, hence $\tC(\pi)$ is an isometry.
\end{proof}

The next result deals with the directional asymptotic cones of subgroups of $G$ which normally generated by words in $G$ of a fixed pattern.
Examples include the commutator subgroup of $G$, or the subgroups $\Gamma^n G$ of the upper central series.

Let $w \in \Free(X)$ be thought of as a word in the alphabet $X$.
Given an assignment of elements of $G$ to the letters in $X$, i.e., a function $\alpha \colon X \to G$, let $w(\alpha) \in G$ denote the evaluation of the word $w$ at the assignment $\alpha$.
Let $w(G)$ denote the subset of $G$ of all such evaluations.
Notice that if $N \nsg G$ then $w(N) \leq N$ so $w(G/N) = w(G)N/N$.

\begin{prop}\label{P:words preserved by cone}
Consider some $w \in \Free(X)$.
Let $H \leq G$ be the normal subgroup generated by $w(G)$.
Equip $w(G)$ with a {\em bounded} length function $\mu$ and equip $H$ with the associated conjugation-invariant pseudonorm in Section \ref{V:word norms}. 
Then $w(\tC(H,T))=1$ for any scaling set $T$.
\end{prop}

\begin{proof}
We suppress $T$ from the notation.
Say $\mu$ is bounded by $C$.

Consider a function $\alpha \colon X \to \tC(H)$.
Then $\alpha(x)=[h(x)]$ for some $h(x) \colon T \to H$ in $\C C(H)$.
Thus $\alpha \in (H^T)^X$ and it corresponds to $\alpha' \in (H^X)^T$ where $\alpha'(t)(x) = \alpha(x)(t)$.
By inspection $w(\alpha) \in \tC(H)$ is represented by the function $k \colon T \to H$ defined by $k(t)=w(\alpha'(t))$.
Since $w(\alpha'(t))$ is a generator of $H$ and since $\mu$ is bounded, it follows that $\|k(t)\| \leq \mu(w(\alpha'(t))) \leq C$.
Therefore $\limsup_{t \in T} \tfrac{\|k(t)\|}{t}=0$ and by Lemma \ref{L:characterize norm 0 in C} $w(\alpha)=[k]=1$.
Since $\alpha$ was arbitrary, $w(\tC(H))=1$.
\end{proof}

\section{Independence of scaling}
\label{section:indepence of scaling}

\noindent
{\bf Running assumptions:}
All groups $G,H,\dots$ are equipped with a conjugation-invariant norm.
Throughout $T,T',\dots$ denote scaling sets.

\begin{defn}\label{def:scaling independence}
A group $G$  is called {\bf independent of scaling} if for any inclusion $T \subseteq T'$ of scaling sets, $\widehat{\res}^{T'}_T \colon \tC(G,T') \to \tC(G,T)$ is an isometry.
\end{defn}

\medskip
\noindent
{\bf Notation:}
For the remainder of the section, when we write $\tC(G)$ we mean $\tC(G,(0,\infty))$, i.e., the full scaling set.

It is clear that if $T \subseteq T' \subseteq T''$ are scaling sets then $\widehat{\res}^{T'}_T\circ\widehat{\res}^{T''}_{T'} = \widehat{\res}^{T''}_T$.
Applying this to the maximal scaling set $T''=(0,\infty)$ we get: 

\begin{prop}\label{prop:scaling independence equivalent}
A group $G$ is independent of scaling if and only if the homomorphisms $\widehat{\res}_{T} \colon \tC(G) \to \tC(G,T)$ are isometries for all scaling sets $T$. 
\end{prop}

\begin{void}\label{paragraph:lim g(t)/t independent of representative}
Let $[g] \in \tC(G,T)$.
If $\lim_{t \in T} \tfrac{\|g(t)\|}{t}$ exists then it is independent of the representative because by Lemma \ref{L:characterize norm 0 in C} $g,h \colon T \to G$ represent the same element in $\tC(G,T)$ if and only if $\limsup_{t \in T} \tfrac{\|g(t)^{-1}h(t)\|}{t}=0$.
\end{void}

\begin{lemma}\label{L:lim exists => independence scaling}
If $\lim_t \tfrac{\|g(t)\|}{t}$ exists for any $[g] \in \tC(G)$ then $G$ is independent of scaling.
\end{lemma}

\begin{proof}
By Propositions \ref{prop:scaling independence equivalent} and \ref{P:restriction injective => isometry} we need to show that $\widehat{\res}_T \colon \tC(G) \to \tC(G,T)$ are injective.
If $[g] \in \tC(G)$ is in the kernel, then by Lemma \ref{L:characterize norm 0 in C} $\lim_{t \in T} \tfrac{\|g(t)\|}{t}=0$, and since $T \subseteq (0,\infty)$ is unbounded $\lim_{t >0} \tfrac{\|g(t)\|}{t}=0$, so $[g]=1$.
\end{proof}

Recall the homomorphisms $\hat{\pi} \colon \tD(G,T) \to \tC(G,T)$ and $\iota \colon \Free_\RR(G) \to \tD(G,T)$ from Section \ref{V:hat pi construction} and Definition~\ref{D:iota FR(A)-->tD(A)}.

\begin{lemma}\label{L:limit g(t) exists if on FQG}
Suppose that $\lim_{t >0} \tfrac{\|h(t)\|}{t}$ exists for any $[h] \in \tC(G)$ 
in the image of $\Free_\QQ(G)$ under $\hat{\pi} \circ \iota$.
Then $\lim_{t>0} \tfrac{\|g(t)\|}{t}$ exists for all $[g] \in \tC(G)$.
If, in addition, for every $[h]$ of the form above
\[
\lim_{t >0} \tfrac{\|h(t)\|}{t}=\|[h]\|_{\tC(G)},
\]
then for every $[g] \in \tC(G)$ 
\[
\| [g] \|_{\tC(G)} = \lim_{t>0} \tfrac{\|g(t)\|}{t}
\]
\end{lemma}

\begin{proof}
We need to show that $\tfrac{\|g(t)\|}{t}$ is Cauchy.
Choose $\epsilon>0$.
Since $\iota(\Free_\RR(G))$ is dense in $\tC(G)$ by Proposition \ref{P:properties Psi iota} and $\Free_\QQ(G)$ is dense in $\Free_\RR(G)$ by Proposition \ref{prop:FSA dense FRA}  and $\hat{\pi}$ is Lipschitz, we can choose $\alpha \in \Free_\QQ(G)$ such that $[h]=\hat{\pi}(\iota(\alpha))$ is at distance $<\tfrac{\epsilon}{2}$ from $[g]$.
That is, $\| [g]^{-1}\cdot[h]\|_{\tC(G)}<\tfrac{\epsilon}{2}$.
By Lemma \ref{L:limsup<norm} 
\[
\limsup_{t>0} \tfrac{| \, \|g(t)\|-\|h(t)\|\, |}{t} \leq 
\limsup_{t>0} \tfrac{|  \|g(t)^{-1}h(t)\| |}{t} \leq
\| [g]^{-1} [h]\|_{\tC(G)} < \tfrac{\epsilon}{2}.
\]
By hypothesis $L=\lim_{t>0} \tfrac{\| h(t)\|}{t}$ exists, so $|\tfrac{\|g(t)\|}{t}-L|<\tfrac{\epsilon}{2}$ for all $t \gg 0$.
In particular $\tfrac{\|g(t)\|}{t}$ is Cauchy, as needed.

Assume that $\lim_{t>0} \tfrac{\|h(t)\|}{t}=\|[h]\|_{\tC(G)}$ for all $[h] \in \tC(G)$ in the image of $\Free_\QQ(G)$.
Given $[g] \in \tC(G)$ choose a sequence $[h_n] \in \tC(G)$ in the image of $\Free_\QQ(G)$ such that $\| [g]^{-1} [h_n]\|_{\tC(G)}< \tfrac{1}{n}$.
It follows from the display above (and since we already know that $\tfrac{\|g(t)\|}{t}$ is Cauchy) that for any $n \geq 1$
\[
\left| \lim_{t>0} \tfrac{\|g(t)\|}{t}-\|[h_n]\|_{\tC(G)} \right| < \tfrac{1}{n}.
\]
Therefore $\lim_{t>0} \tfrac{\|g(t)\|}{t} = \lim_n \|[h_n]\|_{\tC(G)} = \| [g]\|_{\tC(G)}$. 
\end{proof}

Lemma \ref{lem:multivariate powers of elements} below gives a criterion for the first condition in Lemma \ref{L:limit g(t) exists if on FQG} and Theorem \ref{T:abelian independent of scaling} gives a condition for the second.

\begin{lemma}\label{lemma:images FQG in CG}
Any element of $\tC(G)$ in the image of $\Free_\QQ(G)$ under $\hat{\pi} \circ \iota$ is the image of an element $\alpha \in \Free_\QQ(G)$ of the form $\overline{g_1}(\tfrac{1}{m}) \cdots \overline{g_k}(\tfrac{1}{m})$ for some $g_1,\dots,g_k \in G$ and some integer $m>0$.
\end{lemma}

\begin{proof}
Suppose $[h] \in \tC(G)$ has the form $\hat{\pi}(\iota(\alpha'))$ where $\alpha' \in \Free_\QQ(G)$.
Then $\alpha'=\overline{g_1}(\tfrac{p_1}{m}) \cdots \overline{g_k}(\tfrac{p_k}{m})$ for some integers $p_i$ and $m>0$.
Set $\alpha=\overline{g_1^{p_1}}(\tfrac{1}{m}) \cdots \overline{g_k^{p_k}}(\tfrac{1}{m})$.
It follows from Lemma \ref{L:pihat powers}\ref{L:pihat powers:1} that
\[
[h]=\hat{\pi}(\iota(\alpha')) = 
\prod_{i=1}^n\hat{\pi}(\iota(\overline{g_i}(\tfrac{k_i}{m}))) = 
\prod_{i=1}^n\hat{\pi}(\iota(\overline{g_1^{k_i}}(\tfrac{1}{m}))) = 
\hat{\pi}(\iota(\alpha)).
\]
\end{proof}

\begin{lemma}\label{lem:multivariate powers of elements}
Suppose that for any $g_1,\dots,g_k \in G$ the limit
\[
\lim_{n \to \infty} \tfrac{1}{n} \|g_1^n \dots g_k^n\|
\] 
exists.
Then $\lim_{t >0} \tfrac{\|h(t)\|}{t}$ exists for any $[h] \in \tC(G)$ in the image of $\hat{\pi} \circ \iota \colon \Free_\QQ(G) \to \tC(G)$.
\end{lemma}

\begin{proof}
By Lemma \ref{lemma:images FQG in CG} $[h]=\hat{\pi}(\iota(\alpha))$ where $\alpha=\overline{g_1}(\tfrac{1}{m}) \cdots \overline{g_k}(\tfrac{1}{m})$, thus $h(t)=g_1^{\Myfloor{\tfrac{t}{m}}} \cdots g_k^{\Myfloor{\tfrac{t}{m}}}$.
Since $0 \leq t-m\Myfloor{\tfrac{t}{m}} < m$, 
\[
\lim_{t\to \infty} \tfrac{\| h(t)\| }{t} =
\lim_{t\to \infty} \tfrac{\Myfloor{\tfrac{t}{m}}}{t} \cdot  \tfrac{1}{\Myfloor{\tfrac{t}{m}}} \| g_1^{\Myfloor{\tfrac{t}{m}}} \cdots g_k^{\Myfloor{\tfrac{t}{m}}} \| =
\tfrac{1}{m} \lim_{n \to \infty} \tfrac{1}{n} \| g_1^n \cdots g_k^n\|
\]
and the limit exists by the hypothesis.
\end{proof}

\begin{defn}\label{def:locally bounded commutators}
A group $G$ has {\em locally bounded commutators} if the commutator subgroup of any finitely generated subgroup $H \leq G$ is bounded, i.e., $[H,H]$ is a bounded subgroup of $G$.
\end{defn}

\begin{example}\label{example:nilpotent loc bdd comm}
Let $G$ be a nilpotent group equipped with some conjugation-invariant norm. 
Then $G$ has locally bounded commutators.
Indeed, if $H \leq G$ is finitely generated, then $H$ is also nilpotent and it is shown in \cite[Theorem 5.H]{MR3426433} that $[H,H]$ is bounded. 
\end{example}

\begin{theorem}\label{T:abelian independent of scaling}
Suppose that $G$ has locally bounded commutators.
Then 
\begin{enumerate}[label=(\alph*)]
\item\label{item:T loc com:indep scaling}  
$G$ is independent of scaling.
\item\label{item:T loc com:limit}
 $\| [g]\|_{\tC(G)}=\lim_{t>0} \tfrac{\|g(t)\|}{t}$ for any $[g] \in \tC(G)$. 
\item\label{item:T loc com:abelian}
$\tC(G)$ is abelian.
\end{enumerate}
\end{theorem}

\begin{proof}
Consider some $[h]$ in the image of $\hat{\pi} \circ \iota \colon \Free_\QQ(G) \to \tC(G)$.
By Lemma \ref{lemma:images FQG in CG} $[h]=\hat{\pi}(\iota(\alpha))$ for some $\alpha \in \Free_\QQ(G)$ of the form $\alpha=\overline{g_1}(\tfrac{1}{m}) \cdots \overline{g_k}(\tfrac{1}{m})$.
Set $\alpha'=\overline{g_1 \cdots g_k}(\tfrac{1}{m}) \in \Free_\QQ(G)$ and $[h']=\hat{\pi}(\iota(\alpha'))$.
Let $H$ be the subgroup of $G$ generated by $g_1,\dots,g_k$.
Clearly, for any $t>0$ there exists $\gamma(t) \in [H,H]$ such that
\[
h(t) = g_1^{\Myfloor{t/m}} \cdots g_k^{\Myfloor{t/m}} = (g_1 \cdots g_k)^{\Myfloor{t/m}} \cdot \gamma(t) = h'(t) \cdot \gamma(t).
\]
By Lemma \ref{L:characterize norm 0 in C}  $\gamma$ represents the trivial element in $\tC(G)$ since by assumption $[H,H]$ is bounded in $G$.
Therefore $[h]=[h']$ in $\tC(G)$ and by Corollary \ref{cor:Cnorm image of generators in FRG}, 
\begin{multline*}
\lim_{t>0} \tfrac{1}{t} \|h(t)\| 
=
\lim_{t>0} \tfrac{1}{t} \|h'(t)\| 
=
\lim_{t>0} \tfrac{\Myfloor{t}}{t} \cdot \tfrac{1}{\Myfloor{t}} \cdot \|(g_1 \cdots g_k)^{\Myfloor{t}}\| 
\\
=
\lim_{n \to \infty} \tfrac{1}{n} \|(g_1 \cdots g_k)^n\| 
=
\tau(g_1\cdots g_k)
=
\| [h']\|_{\tC(G)}
=
\| [h]\|_{\tC(G)}.
\end{multline*}
The conditions of Lemma \ref{L:limit g(t) exists if on FQG} are fulfilled by $G$, hence \ref{item:T loc com:limit} holds, and by Lemma \ref{L:lim exists => independence scaling} $G$ is independent of scaling.

It remains to prove point \ref{item:T loc com:abelian}.
First, suppose that $[u],[v] \in \tC(G)$ are in the image of $\Free_\RR(G)$ under $\hat{\pi} \circ \iota$, see Section \ref{V:hat pi construction}.
Then $u(t)=h_1^{\Myfloor{r_1t}}\cdots h_k^{\Myfloor{r_kt}}$ and $v(t)=g_1^{\Myfloor{s_1t}}\cdots g_n^{\Myfloor{s_nt}}$ for some $h_i,g_i \in G$ and $r_i,s_i \in \RR$.
Let $K$ be the subgroup of $G$ generated by $h_i, g_i$ and let $z$ denote the commutator $[u,v]$.
Then $z(t) \in [K,K]$ for any $t>0$.
Since by the assumption $[K,K]$ is bounded in $G$, Lemma \ref{L:characterize norm 0 in C} shows that $[z]=1$ in $\tC(G)$.
Thus, $[u]\cdot [v] = [v] \cdot [u]$ for any $[u],[v] \in \tC(G)$ in the image of $\Free_\RR(G)$.

Let $x,y \in \tC(G)$ be arbitrary.
We will show that $xyx^{-1}y^{-1}=1$.
Choose some $\epsilon>0$.
By Section \ref{V:hat pi construction} there are $x',y' \in \tC(G)$ in the image of $\Free_\RR(G)$ such that $\|x^{-1}x'\|_{\tC(G)} < \epsilon$ and $\|y^{-1}y'\|_{\tC(G)}<\epsilon$.
We leave it to the reader to check that $(xyx^{-1}y^{-1})\cdot(x'y'x'{}^{-1}y'{}^{-1})^{-1})$ is a product of four conjugates of $(x^{-1}x')^{\pm 1}$ and $(y^{-1}y')^{-1}$.
But we have seen above that $x'y'x'{}^{-1}y'{}^{-1})^{-1}=1$.
Hence, $\| xyx^{-1}y^{-1}\|_{\tC(G)} < 4\epsilon$.
Since $\epsilon>0$ was arbitrary $xyx^{-1}y^{-1}=1$.
\end{proof}

\section{Relation to the ultralimit asymptotic cone}
\label{section:ultrafilter cone}

\noindent
{\bf Running assumptions:}
All group $G,H,\dots$ are equipped with a conjugation-invariant norm.
Throughout $T,T',\dots$ denote scaling sets.

\begin{void}
Let $T$ be a scaling set.
A {\bf $T$-valued scaling sequence} is a function $\B d \colon \NN \to T$ such that $\lim_{n \to \infty} \B d(n)=\infty$.
We call $\B d$ a {\bf scaling sequence} if $T=(0,\infty)$.
We will write $d_n$ for $\B d(n)$ when this is convenient, which is the usual practice. 

Throughout this section $\omega$ denotes a non principal ultrafilter on $\NN$.
The elements of $\Cone_\omega(G,\B d)$ are denoted by $\{g_n\}$ where $(g_n)$ is a sequence in $G$ such that $\limsup_n \tfrac{1}{d_n} \|g_n\|_G < \infty$.
\end{void}

\begin{prop}\label{P:directional into asymptotic}
Let $T$ be a scaling set.
For any $T$-valued scaling sequence and any non principal ultrafilter $\omega$ there is a well defined canonical homomorphism 
\[
\rho_{\omega,\B d} \colon \tC(G,T) \xto{ \ [g] \mapsto \{g \circ \B d\} \ } \Cone_\omega(G,\B d).
\]
It is Lipschitz with constant $1$.
\end{prop}

\begin{proof}
Since $d_n \in T$ and $\lim_{n \to \infty} d_n=\infty$, Lemma \ref{L:limsup<norm} implies
\[
\limsup_n \tfrac{1}{d_n}\|g(d_n)\| \leq
\limsup_{t \in T} \tfrac{1}{t}\|g(t)\| \leq \|[g]\|_{\tC(G,T)} < \infty.
\]
Thus, if $g \in \C C(G,T)$ then $g \circ \B d \colon \NN \to G$ represents an element in $\Cone_{\omega}(G,\B d)$.
To see that $\rho$ is independent of representatives, suppose that $g,g' \in \C C(G,T)$ represent the same element in $\tC(G,T)$.
Then $[g^{-1}g']$ is the trivial element in $\tC(G,T)$ so Lemma \ref{L:characterize norm 0 in C} implies
\[
\lim_n \tfrac{1}{d_n} \| g(d_n)^{-1}g'(d_n)\| = \lim_{t \in T} \tfrac{1}{t}\|g(t)^{-1} g'(t)\|=0.
\]
Therefore $\{g \circ \B d\}=\{ g' \circ \B d\}$ in $\Cone_\omega(G,\B d)$.
It follows that $\rho$ is well defined.
It is a homomorphism because the group structures on $\tC(G,T)$ and $\Cone_\omega(G,\B d)$ are induced from that of $G$.
\end{proof}

Recall Definition \ref{def:scaling independence} of independence of scaling. 
Also recall that $\tC(G)$ denotes $\tC(G,(0,\infty))$ with the maximal scaling set $(0,\infty)$.

\begin{theorem}\label{T:equivalence scaling independence and injective rho}
The following are equivalent for a group $G$.
\begin{enumerate}[label=(\alph*)]
\item\label{T:equivalence scaling independence and injective rho:scaling}
$G$ is independent of scaling.

\item\label{T:equivalence scaling independence and injective rho:injective}
The homomorphism $\rho_{\omega,d} \colon \tC(G) \to \Cone_{\omega}(G,\B d)$ is injective for any non-principal ultrafilter $\omega$ and any scaling sequence $d$.
\end{enumerate}
If, in addition, $\|[g]\|_{\tC(G)} = \lim_{t\to \infty} \tfrac{1}{t} \|g(t)\|$ for every $[g] \in \tC(G)$, then $\rho_{\omega,d}$ is an isometric embedding.
\end{theorem}

\begin{proof}
\ref{T:equivalence scaling independence and injective rho:scaling} $\implies$ \ref{T:equivalence scaling independence and injective rho:injective}
Fix some $\omega$ and $d$ and consider some $[g] \in \tC(G)$ in the kernel of $\rho=\rho_{\omega,d}$.
Then $\omegalim_n \tfrac{1}{d_n}\|g(d_n)\|=0$ and therefore there is a subsequence $\{ n_k\}_{k \geq 1}$ such that $\lim_{k \to \infty} \tfrac{1}{d_{n_k}} \| g(d_{n_k})\|=0$.
By further passage to a subsequence we may assume that $\{d_{n_k}\}_k$ is increasing.
Set $T=\{ d(n_k) : k \geq 1\}$.
This is a scaling set and we set $[h]=\widehat{\res}_T([g])$.
Then  $h=g|_T$ and since $d_{n_k}$ is increasing
\[
\limsup_{t \in T} \tfrac{1}{t}\|h(t)\| = \limsup_k \tfrac{1}{d(n_k)}\|g(d(n_k))\|=0.
\]
By Lemma \ref{L:characterize norm 0 in C} $[h]=1$, but by hypothesis $\widehat{\res}_T$ is injective, so $[g]=1$.

\ref{T:equivalence scaling independence and injective rho:injective} $\implies$ \ref{T:equivalence scaling independence and injective rho:scaling}
Consider a scaling set $T$.
Choose a scaling sequence $d_1<d_2<\dots$ in $T$ (this can be done since $T$ is unbounded).
By inspection the injective homomorphism $\rho_{\omega,d} \colon \tC(G) \to \Cone_\omega(G,\B d)$ factors as
\[
\tC(G) \xto{\widehat{\res}_T} \tC(G,T) \xto{\rho_{\omega,\B d}} \Cone_\omega(G,\B d).
\]
Therefore $\widehat{\res}_T$ is injective and by Proposition \ref{P:restriction injective => isometry} it is an isometry.

Suppose that $\lim_{t>0} \tfrac{1}{t}\|g(t)\|=\|[g]\|_{\tC(G)}$ where $[g] \in \tC(G)$.
Clearly $\| \rho([g])=\|\{ g \circ d\}\| = \omegalim_n \tfrac{1}{d_n}\|g(d_n)\|$ and the $\omega$-limit is an honest limit by the assumption.
Then $\| \rho([g])\|=\|[g]\|_{\tC(G)}$ and $\rho$ is an isometry.
\end{proof}

\section{Examples: Abelian and nilpotent groups and verbal norms}

We begin with the simplest examples of abelian groups.

\begin{prop}\label{prop:tCZ} 
Let $\ZZ$ be equipped with the standard metric.
Then 
\begin{enumerate}[label=(\alph*)]
\item
$\tC(\ZZ)$ is isometric to $\RR$ with the usual metric.

\item
$\ZZ$ is independent of scaling (Definition \ref{def:scaling independence}) thus $\tC(\ZZ,T)$ is isometric to $\RR$ for any scaling set $T$.

\item
$\rho \colon \tC(\ZZ) \to \Cone_\omega(\ZZ,\B d)$ is an isometry for any $\omega$ and scaling sequence  $\B d$.
\end{enumerate}
\end{prop}

\begin{proof}
We choose the scaling set $T=(0,\infty)$.
By Proposition \ref{P:properties Psi iota} and Paragraph \ref{V:hat pi construction} the homomorphism $\hat{\pi} \circ \iota \colon \Free_\RR(\ZZ) \to \tC(\ZZ)$ is Lipschitz with constant $1$ and has dense image.
Therefore it suffices to show that the image of $\hat{\pi}\circ \iota$ is isometric to $\RR$.

With the notation in Paragraph \ref{V:notation generators FRG}, let $L \leq \Free_\RR(\ZZ)$ be the subgroup $L=\{ \overline{1}(r) : r \in \RR\}$.
We claim that
\[
\hat{\pi}\circ \iota(\Free_\RR(\ZZ)) = \hat{\pi} \circ \iota(L).
\]
Since $\hat{\pi} \circ \iota$ is a homomorphism it suffices to show that the image of any generator $\alpha = \overline{n}(r) \in \Free_\RR(\ZZ)$ is in the image of $L$.
Set $[h]=\hat{\pi}(\iota(\alpha))$.
It is represented by the function $h \colon T \to \ZZ$ gives by $h(t)=n \cdot \Myfloor{rt}$.
Consider $\alpha'=\overline{1}(nr) \in L$ and set $[h']=\hat{\pi}(\iota(\alpha'))$ represented by $h'(t)=\Myfloor{nrt}$.
Now,
\[
\limsup_{t>0} \tfrac{1}{t} (|\Myfloor{nrt}-n\Myfloor{rt}|) \leq
\limsup_{t>0} \tfrac{1}{t} |n| = 0.
\]
By Lemma \ref{L:characterize norm 0 in C} $[h^{-1}h']_{\tC(\ZZ)}=0$ so $[h]=[h']$ as required.

By the definition of the metric on $\Free_\RR(\ZZ)$ it is clear that $L$ is isometric to $\RR$ via $r \mapsto \overline{1}(r)$.
It is left to show that $\hat{\pi} \circ \iota \colon L \to \tC(\ZZ)$ is an isometric embedding.
Indeed, the stable length of any $n \in \ZZ$ is $|n|$ so by Lemma
\[
\| \hat{\pi}(\iota(\overline{1}(r))) \|_{\tC(\ZZ)} = |r| \cdot \tau(1) = |r| = \| \overline{1}(r)\|_L.
\]
Theorems \ref{T:abelian independent of scaling} and \ref{T:equivalence scaling independence and injective rho} show that $\ZZ$ is independent of scaling and that $\rho_{\omega,\B d} \colon \tC(\ZZ) \to \Cone_\omega(\ZZ,\B d)$ is an isometric embedding.
It is, in fact, an isometry since it is well known that $\Cone_\omega(\ZZ,\B d) \cong \RR$.
\end{proof}

\begin{prop}\label{prop:tCZn} 
Equip $\ZZ^n$  with the standard $L^1$-norm $\| v \|=\sum_i|v_i|$.
\begin{enumerate}[label=(\alph*)]
\item
$\tC(\ZZ^n,T)$ is isometric to $\RR^n$ with the standard $L^1$-norm.
\item
$\ZZ^n$ is independent of scaling and $\rho_{\omega,d} \colon \tC(\ZZ^n) \to \Cone_\omega(\ZZ^n,\B d)$ is an isometry.
\end{enumerate}
\end{prop}

\begin{proof}
This follows from Proposition \ref{prop:tCZ} and Theorem \ref{T:tC products}.
The second assertion follows from Theorems \ref{T:abelian independent of scaling} and \ref{T:equivalence scaling independence and injective rho} and the well known fact that $\Cone_\omega(\ZZ^n,d) \cong \RR^n$.
\end{proof}

We move on to study nilpotent groups.

\begin{prop}\label{prop:cone of general nilpotent groups}
Suppose that $G$ is a nilpotent group equipped with a conjugation-invariant norm.
Then $G$ is independent of scaling, $\tC(G)$ is abelian, and $\rho_{\omega,\B d} \colon \tC(G) \to \Cone_\omega(G,\B d)$ is an isometric embedding for any ultrafilter $\omega$ and scaling sequence $\B d$.
\end{prop}

\begin{proof}
This follows from Example \ref{example:nilpotent loc bdd comm} and Theorems \ref{T:abelian independent of scaling} and \ref{T:equivalence scaling independence and injective rho}.
\end{proof}

\begin{prop}\label{prop:quotient by bounded comm standard word norm}
Let $G$ be normally generated by $A$ and equipped with the standard conjugation-invariant word norm $\| \ \|_A$, Definition \ref{defn:standard word norm}.
Suppose that $[G,G]$ is bounded and that the image of $A$ in $G_{\OP{ab}}$ is finite.
Then 
\begin{enumerate}[label=(\alph*)]
\item\label{item:prop quotient bdd comm:scaling}
$G$ is independent of scaling, 
\item\label{item:prop quotient bdd comm:rho}
$\tC(G) \to \Cone_\omega(G,\B d)$ is an isometry, and 
\item\label{item:prop quotient bdd comm:Rn}
both groups are bi-Lipschitz equivalent to $(\RR^n,\| \ \|_1)$ where $n=\dim G_{\OP{ab}} \otimes \RR$. 
\end{enumerate}
\end{prop}

\begin{proof}
Let $\pi \colon G \to G_{\OP{ab}}$ be the quotient with kernel $[G,G]$.
Set $B=\pi(A)$.
By assumption $B$ is finite, and it generates $G_{\OP{ab}}$.
By Proposition \ref{prop:quotient of canonical word norms} $\pi \colon (G,\| \ \|_A) \to (G_{\OP{ab}},\| \ \|_B)$ is a metric quotient.
Since $[G,G]$ is bounded Corollary \ref{cor:metric quotient bounded kernel} implies that $\tC(\pi)$ is an isometry $\tC(G,T) \to \tC(G_{\OP{ab}},T)$ for any scaling set $T$.
But any abelian group is independent of scaling by Theorem \ref{T:abelian independent of scaling}, hence $G$ is independent of scaling.
This proves point \ref{item:prop quotient bdd comm:scaling}.

Let $T$ denote the torsion subgroup of $G_{\OP{ab}}$ and set $H=G_{\OP{ab}}/T$.
Equip $H$ with the quotient norm, Lemma \ref{L:quotient pseudonorm}, and let $\tau \colon G_{\OP{ab}} \to H$ be the quotient map.
By Proposition \ref{prop:quotient of canonical word norms} the norm on $H$ is the standard word norm $\| \ \|_{\tau(B)}$.
Clearly $H \cong \ZZ^n$ for some $n$ and let $\lambda \colon H \to \ZZ^n$ be an isomorphism.
It follows from Proposition \ref{prop:finite generating sets bi lipschitz equivalence} that $\lambda$ is a bi-Lipschitz equivalence.
The naturality of the maps $\rho_{\omega,\B d}$, Section \ref{section:ultrafilter cone}, implies the commutativity of the following diagram.
\[
\xymatrix{
\tC(G) \ar[r]^{\tC(\pi)} \ar[d]_{\rho} 
& 
\tC(G_{\OP{ab}}) \ar[d]_\rho \ar[r]^{\tC(\tau)} 
&
\tC(H) \ar[d]_\rho \ar[r]^{\tC(\lambda)} 
&
\tC(\ZZ^n) \ar[d]^\rho
\\
\Cone_\omega(G,\B d) 
\ar[r]_{\pi_*} &
\Cone_\omega(G_{\OP{ab}},\B d) 
\ar[r]_{\tau_*} &
\Cone_\omega(H,\B d) 
\ar[r]_{\lambda_*} &
\Cone_\omega(\ZZ^n)
}
\]
We have seen above that $\tC(\pi)$ is an isometry.
The boundedness of $[G,G]$ also easily implies that $\pi_*$ is an isometry.
Since $T$ is finite it is bounded in $G_{\OP{ab}}$ and the same argument shows that $\tC(\tau)$ and $\tau_*$ are isometries.
Since $\lambda$ is a bi-Lipschitz equivalence $\tC(\lambda)$ is a bi-Lipschitz equivalence by Proposition \ref{P:naturality C and tC}, and $\lambda_*$ is a bi-Lipschitz equivalence by easy inspection of the definition of ultrafilter cones.
By Proposition \ref{prop:tCZn} the leftmost vertical arrow $\rho$ is an isometry between groups that are isometric to $(\RR^n,\| \ \|_1)$.
Therefore $\rho$ at the first column is a bi-Lipschitz equivalence of groups which are bi-Lipschitz equivalent to $(\RR^n,\| \ \|_1)$.
Since it is also an isometric embedding, it is an isometry and this proves \ref{item:prop quotient bdd comm:rho} and \ref{item:prop quotient bdd comm:Rn}.
\end{proof}

\begin{example}\label{example:fg nilpotent groups}
Let $G$ be one of the groups equipped with a conjugation-invariant norm
\begin{enumerate}[leftmargin=* , label=(\roman*)]
\item \label{example:intro ab-ht-sol:nilp}
A finitely generated nilpotent group.

\item \label{example:intro ab-ht-sol:sol}
A finitely generated solvable group whose commutator subgroup is finitely generated nilpotent.

\item \label{example:intro ab-ht-sol:ht}
Thompson's group $F$ or, more generally, a Higman-Thompson group.
\end{enumerate}
Then $G$ is independent of scaling, $\rho_{\omega,\B d} \colon \tC(G) \to \Cone_\omega(G,\B d)$ is an isometry for any ultrafilter $\omega$ and any scaling sequence $\B d$, both groups are bi-Lipschitz equivalent to $(\RR^n,\| \ \|_1)$.
\begin{proof}
Since $G$ is finitely generated, the conjugation-invariant {\em word} norm on $G$ Lipschitz dominates any other conjugation-invariant norm.
Then $[G,G]$ is bounded in $G$ for groups $G$ of the form \ref{example:intro ab-ht-sol:nilp} by \cite[Theorem 5.H]{MR3426433}.
Similarly, $[G,G]$ is bounded in $G$ for  groups of the form \ref{example:intro ab-ht-sol:sol} by \cite[Theorem 5.K]{MR3426433}.
For groups of the form \ref{example:intro ab-ht-sol:ht}, Theorem 1.1 and the subsequent discussion in \cite[Theorem 1.1]{MR3693109} show that $[G,G]$ is six-uniformly simple, i.e., for any $1 \neq g \in [G,G]$, every other $h \in [G,G]$ can be written as a product of at most six conjugates of $g$. 
Hence $[G,G]$ is bounded in all three cases \ref{example:intro ab-ht-sol:nilp}--\ref{example:intro ab-ht-sol:ht}.
Apply Proposition \ref{prop:quotient by bounded comm standard word norm}.
\end{proof}
\end{example}

\begin{example}\label{E:symmetric8}
Let $G$ be a group, all of whose elements have finite order.
Then $\tC(G)=1$ by Proposition \ref{P:distorted => trivial cone}.
For example $S_\infty$, the group of finitely supported permutations of $\{1,2,\dots\}$ is a torsion group and therefore $\tC(S_{\infty})=1$.
In contrast, Karlhofer \cite{2203.10889} proved that the ultrafilter cone $\OP{Cone}_{\omega}(S_{\infty})$ with respect to any non-principal ultrafilter is a nontrivial contractible simple group in which every element is a commutator.
\end{example}

\begin{example}\label{example:commutator length abelian}
Let $G$ be a group and equip $H=[G,G]$ with the commutator length norm.
Then $\tC(H)$ is abelian.
More generally, let $H=\Gamma_n(G)$ be the $n^{\text{th}}$ group in the lower central series of $G$.
That is, $H$ is the subgroup of $G$ generated by all commutators of length $n$.
Equip $H$ with the conjugation-invariant word norm associated with this set of generators.
Then $\tC(H)$ is nilpotent of class $n$.
Indeed,
let $[x_1,x_2,\dots,x_n]$ denote the iterated commutator.
Apply Proposition \ref{P:words preserved by cone} with the word $w=[x_1,x_2,\dots,x_n] \in \Free(x_1,\dots,x_n)$.
\end{example}

\section{Example: The free group}\label{S:coneFn}

{\bf Assumptions:}
In this section we study the free group $G=\Free(A)$  equipped with the standard conjugation-invariant word norm, i.e., the word norm associated to the constant length function $\mu(a)=1$.

\begin{theorem}\label{theorem:F(A) independent of scaling}
$\Free(A)$ is independent of scaling.
The homomorphisms $\rho_{\omega,\B d} \colon \tC(\Free(A)) \to \Cone_\omega(\Free(A),\B d)$ are injective.
\end{theorem}

\begin{proof}
The hypothesis of Lemma \ref{lem:multivariate powers of elements} is satisfied for $\Free(A)$ by \cite[Theorem 1.1, Example 1.2]{MR4741237}.
Then $\Free(A)$ is independent of scaling by Lemmas \ref{L:limit g(t) exists if on FQG} and \ref{L:lim exists => independence scaling}.
The rest follows from Theorem \ref{T:equivalence scaling independence and injective rho}.
\end{proof}

We do not know whether $\rho_{\omega, d}$ are isometries.
In the remainder of this section we will give some details about the structure of $\tC(\Free(A))$.
Once again, our results are incomplete.

\begin{defn}
Let $\Gamma$ be a group.
An element $g \in \Gamma$ is called {\bf pure} if $g=h^n$ for some $h \in G$ and some $n \in \ZZ$ implies $h=g^{\pm 1}$ and $n=\pm 1$.
Let $\Pure(\Gamma)$ denote the set of pure elements in $\Gamma$.
\end{defn}

\noindent
{\bf Example:}
The identity element $e \in \Gamma$ is never a pure element.
The only pure elements in $\ZZ$ are $\pm 1$.
Divisible groups contain no pure elements.

\begin{void}\label{par:Fix Theta in F(A)}
$\Pure(\Gamma)$ is closed under conjugation and inverses.
The relation $g \sim g'$ if $g'$ is conjugate to $g^{\pm 1}$ is an equivalence relation on $\Pure(\Gamma)$.

{\em For the remainder of this section we fix}
\[
\Theta \subseteq \Pure(\Free(A)),
\]
a set of cyclically reduced representatives for the equivalence classes in $\Pure(\Free(A))$.
Notice that $\Theta$ does not contain the identity element.
\end{void}

Recall the homomorphisms $\hat{\pi}$ and $\iota$ from paragraph \ref{V:hat pi construction} and Definition  \ref{D:iota FR(A)-->tD(A)}.
Recall that $\Free(\Theta;\tau)$ denotes $\Free(\Theta)$ equipped with the conjugation-invariant norm defined in paragraph \ref{V:word norms} with respect to the stable length function $\tau \colon \Theta \to \RR$.

\begin{theorem}\label{theorem:F(Theta) to tC(Free)}
Set $G=\Free(A)$.
The restriction of $\hat{\pi} \circ \iota \colon \Free_\RR(G) \to \tC(G)$ to $\Free_\RR(\Theta)$ is an {\em injective} homomorphism
\[
\xi_\Theta \colon \Free_\RR(\Theta;\tau) \to \tC(G)
\]
with dense image.
It is Lipschitz with constant $1$ and $\|\xi(\alpha)\|=\|\alpha\|$ for all generators $\alpha=\overline{\theta}(r) \in \Free_\RR(\Theta)$.
Upon completion,
\[
\hat{\xi}_\Theta \colon \widehat{\Free(\Theta;\theta)} \to \tC(G)
\]
is a metric quotient homomorphism (Definition \ref{VD:metric quotient}).
\end{theorem}

\medskip
\noindent
{\bf Conjecture:}
The map $\hat{\xi}$ is an isometry.

\medskip
\noindent
{\bf Notation:}
The remainder of this section is devoted to the proof of Theorem \ref{theorem:F(Theta) to tC(Free)}.
Throughout we denote
\[
G=\Free(A).
\]
Since $\Free(A)$ is independent of scaling we will work with the scaling set $T=(0,\infty)$.
Throughout $\mu \colon G \to \RR$ denotes the norm function and $\tau$ the stable length.

\begin{void}
Let us recall several elementary facts about the free group $\Free(A)$.
Let $w$ be a word in the alphabet $A \sqcup A^{-1}$.
We will write $|w|$ for the {\bf length} of the word $w$, i.e., the number of letters in the alphabet $A \sqcup A^{-1}$ appearing in it.
We call $w$ {\bf reduced} if it contains no subword of the form $aa^{-1}$ or $a^{-1}a$ which we call an {\bf elementary relation}.
The elements of $\Free(A)$ are equivalence classes of words and each is represented by a unique reduced word.
Any word $w$ may be reduced by successively eliminating elementary relations; The resulting reduced word is independent of the order elementary relations are eliminated.

At this stage it is worthwhile to highlight some elementary facts related to this process of reduction.
Let us fix terminology and say that a {\bf sequence} in a word $w$ is just a collection of symbols (letters) in $w$, while a {\bf subword} or an {\bf interval} is a sequence of consecutive symbols in $w$.
Clearly, the reduced form of $w$ is obtained by removing a collection of symbols from $w$, which form a disjoint union of intervals in $w$.
In the sequel we will use repeatedly the following observation.
Consider a sequence $u$ in $w$.
Let $w-u$ denote the word obtained by removing the symbols in $u$ from $w$.
Then the reduced form of $w-u$ is obtained from $w$ by the removal of $i$ disjoint intervals whose union contains the symbols in $u$.
Clearly, $i \leq |u|$.

Suppose that  $w,u$ are reduced words.
Their concatenation $wu$ may or may not be reduced.
It is not reduced only if an elementary relation appears at the juncture.
In this case the reduction sequence of $wu$ is unique and it is obtained in a unique way by successively removing an elementary relation at the juncture. 
Thus, $w=w'v$ and $u=v^{-1}u'$ for some suffix $v$ of $w$ and $w'u'$ is the reduced form of $wu$.

Suppose that $w=w_n\cdots w_1$ and $u=u_1\cdots u_m$ are {\em reduced} words expressed as the concatenation of non-empty (reduced) words.
We say that $w_i$ and $u_j$ {\bf collide} if in the process of reduction of $wu$ ``they meet each other'' in the sense that there are letters in $w_i$ that are cancelled with those of $u_j$ in the unique(!) process of reduction of $wu$.

More precisely, $w_i$ and $u_j$ collide if $w=w'v$ and $u=v^{-1}u'$ and
\[
\max\left\{ \sum_{k=1}^{i-1} |w_k| , \sum_{k=1}^{j-1}|u_k| \right\} \, <  \,
\min\left\{ |v|, \sum_{k=1}^{i} |w_k| , \sum_{k=1}^{j}|u_k| \right\}.
\]
Figure \ref{Fig:wiuj collide} illustrates the situation.
The ``collision'' is at an interval of length 
\[
\ell = \min\left\{ |v|, \sum_{k=1}^{i} |w_k| , \sum_{k=1}^{j}|u_k| \right\} - \max\left\{ \sum_{k=1}^{i-1} |w_k| , \sum_{k=1}^{j-1}|u_k| \right\}.
\]

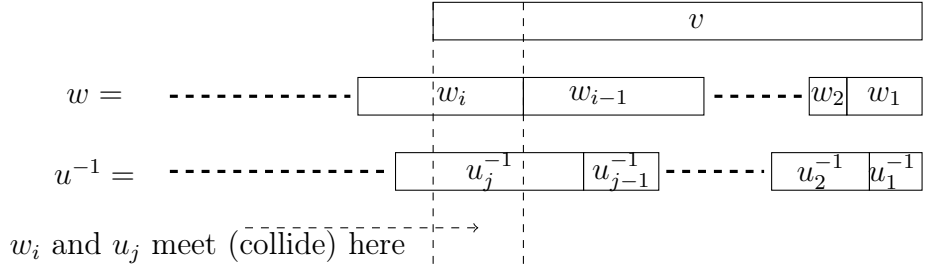
\begin{figure}
\begin{tikzpicture}
\draw (5.5,0) rectangle ++(6.5,-0.5) ;
\draw (9,-0.25) node {$v$};
%
\draw (11,-1) rectangle ++(1,-0.5) ;
\draw (11.5,-1.25) node {$w_1$} ;
\draw (10.5,-1) rectangle ++(0.5,-0.5) ;
\draw (10.75,-1.25) node {$w_2$} ;
\draw[dashed, very thick] (10.4,-1.25) -- (9.2,-1.25) ;
\draw (6.7,-1) rectangle ++(2.4,-0.5) ;
\draw (7.7,-1.25) node {$w_{i-1}$} ;
\draw (4.5,-1) rectangle ++(2.2,-0.5) ;
\draw (5.75,-1.25) node {$w_{i}$} ;
\draw[dashed, very thick] (2,-1.25) -- (4.4,-1.25) ;
\draw (1,-1.25) node{$w=$} ;
\draw (11.3,-2) rectangle ++(0.7,-0.5) ;
\draw (11.6,-2.25) node {$u_1^{-1}$} ;
\draw (10,-2) rectangle ++(1.3,-0.5) ;
\draw (10.65,-2.25) node {$u_2^{-1}$} ;
\draw[dashed, very thick] (8.6,-2.25) -- (9.9,-2.25) ;
\draw (7.5,-2) rectangle ++(1,-0.5) ;
\draw (8,-2.25) node {$u_{j-1}^{-1}$} ;
\draw (5,-2) rectangle ++(2.5,-0.5) ;
\draw (6.25,-2.25) node {$u_{j}^{-1}$} ;
\draw[dashed, very thick] (2,-2.25) -- (4.9,-2.25) ;
\draw (1,-2.25) node{$u^{-1}=$} ;
\draw[dashed] (5.5,0)--(5.5,-3.5) ;
\draw[dashed] (6.7,0)--(6.7,-3.5) ;
\draw[dashed, -to] (3,-3)--(6.1,-3);
\draw (2.5,-3.25) node {$w_i$ and $u_j$ meet (collide) here};
\end{tikzpicture}
\caption{Collision of $w_i$ and $u_j$}\label{Fig:wiuj collide}
\end{figure}

Let $g$ be a {\em cyclically reduced} 
word in the alphabet $A \sqcup A^{-1}$.
Notice that the word $g^n=g \cdots g$ formed by the $n$-fold concatenation of $g$ is (cyclically) reduced.
We say that a word $v$ is a {\bf $g$-packet} if it is a sub-word (i.e., an interval) in the word $g^n=g \cdots g$ for some $n \neq 0$ (including negative $n$).
It is clear that any interval in a $g$-packet is itself a $g$-packet.
It is also clear that a $g$-packet of length $k \cdot |g|$ is conjugate to $g^{\pm k}$ in $\Free(A)$.
The following is therefore clear
\end{void}

\begin{prop}
Let $w=w_n\cdots w_1$ and $u=u_1\cdots u_k$ be reduced words.
Suppose that  in the process of reduction of $wu$ the subwords $w_i$  and $u_j$ collide at in interval of length $\ell$.
If $w_i$ is a $g$-packet and $u_j$ is an $h$-packet, and if $\ell=\OP{lcm}(|g|,|h|)$ then $g^{\pm k}$ is conjugate to $h^{ \pm m}$ where $k=\ell/|g|$ and $m=\ell/|h|$.
 \end{prop}

Recall $\Theta \subseteq \Free(A)$ from paragraph \ref{par:Fix Theta in F(A)}.

\begin{prop}\label{P:pure roots}
For any $1 \neq g \in G=\Free(A)$ there exists some $\theta \in \Theta$ such that $g$ is conjugate to $\theta^n$ for some $n \neq 0$.
\end{prop}

\begin{proof}
Use induction on $|g|$.
By replacing $g$ with a conjugate, we may assume that $g$ is cyclically reduced.
If $g$ is pure, we are done.
Otherwise $g=h^k$ for some $h \in G$ and $k \in \ZZ$ and $k \geq 2$.
Then $h$ is cyclically reduced and $|g|=k \cdot |h|$.
Apply the induction hypothesis to $h$.
\end{proof}

\begin{lemma}\label{L:Theta is pure}
Let $\theta_1, \theta_2 \in G$ be cyclically reduced and pure elements.
Suppose that $\theta_1^{n_1}=\theta_2^{n_2}$ for some $n_1,n_2 \neq 0$.
Then $\theta_1=\theta_2^{\pm 1}$.
\end{lemma}

\begin{proof}
By replacing $\theta_1$ with $\theta_1^{-1}$ if $n_1<0$ and replacing $\theta_2$ with $\theta_2^{-1}$ if $n_2<0$ we may assume that $n_1,n_2>0$.

Clearly $\theta_1^{n_1}$ and $\theta_2^{n_2}$ are cyclically reduced.
Set $\ell_1=|\theta_1|$ and $\ell_2=|\theta_2|$.
The result is obvious if $\ell_1=\ell_2$ by comparing the reduced words $\theta_1^{n_1}=\theta_2^{n_2}$.
So we assume that $\ell_1<\ell_2$ and write $m=\OP{gcd}(\ell_1,\ell_2)$.
Set $p_1=\tfrac{\ell_1}{m}$ and $p_2=\tfrac{\ell_2}{m}$.
Then $\theta_1=\eta_0 \cdots \eta_{p_1-1}$ with $|\eta_i|=m$.
Similarly $\theta_2=\zeta_0 \cdots \zeta_{p_2-1}$ with $|\zeta_i|=m$.
Clearly, $\eta_0=\zeta_0$ since these are the first $m$ symbols in $\theta_1^{n_1}=\theta_2^{n_2}$.
By the Chinese remainder theorem for any $0 \leq r <p_2$ there exists $0 \leq q < p_2$ such that $q p_1 = r \mod p_2$.
That means that the $q$-th occurrence of $\eta_0$ in $\theta_1^{n_1}$ is equal to $\zeta_r$ in some occurrence of $\theta_2$ in $\theta_2^{n_2}$ (notice that $p_1| n_2$).
Therefore $\zeta_i=\eta_0$ for all $i$.
By a similar argument $\eta_i=\zeta_0$ for all $i$.
So $\theta_1=\eta_0^{p_1}$ and $\theta_2=\eta_0^{p_2}$.
The result follows since $\theta_1$ and $\theta_2$ are pure.
\end{proof}

\begin{lemma}\label{L:equal packets equal words}
Consider $\theta_1, \theta_2 \in \Theta$ and suppose that $\ell \geq \OP{lcm}(|\theta_1|,|\theta_2|)$.
Suppose that there exists a reduced word $v$ such that $|v| \geq \ell$ and such that $v$ is a $\theta_1$-packet and $v^{-1}$ is a $\theta_2$-packet.
Then $\theta_1=\theta_2$.
\end{lemma}

\begin{proof}
By taking a suffix of $v$ we may assume that $\ell=\OP{lcm}(|\theta_1|,|\theta_2|)$.
Indeed, such a suffix remains reduced and a $\theta_i$-packet ($i=1,2$).
Set $n_1=\tfrac{\ell}{|\theta_1|}$ and $n_2=\tfrac{\ell}{|\theta_2|}$. 
Then $v$ is a cyclic permutation of $\theta_1^{\pm n_1}$ and $v^{-1}$ is a cyclic permutation of $\theta_2^{\pm n_2}$.
Hence $\theta_1^{\pm n_1}$ is a cyclic permutation of $\theta_2^{\pm n_2}$.
Then $\theta_1^{\pm n_1} = \theta_3^{\pm n_2}$ for some $\theta_3$ which is a cyclic permutation of $\theta_2$.
Lemma \ref{L:Theta is pure} implies that $\theta_1=\theta_3^{\pm 1}$, so $\theta_1$ is conjugate to $\theta_2^{\pm 1}$.
That is, $\theta_1$ and $\theta_2$ belong to the same equivalence class in $P(G)$, so $\theta_1=\theta_2$ by the definition of $\Theta$.
\end{proof}

It is well known that $G=\Free(A)$ has no distorted elements, namely $\tau(g)>0$ for all $1 \neq g \in G$.
To see this, observe that for every nontrivial $g\in \Free(A)$ there exists a homogeneous quasimorphism
$\psi\colon \Free(A)\to \RR$ such that $\psi(g)>0$ \cite[Section 4.E]{MR3426433}. It follows that $\tau(g)>0$.

\begin{lemma}\label{L:small intervals reduction}
Consider cyclically reduced $g \in G$.
Then the removal of $m$ symbols from $g^n$ results in a  word whose reduced form is obtained from $g^n$ by the removal of at most $m$ intervals of total length at most
\[
\tfrac{|g| \cdot (1+|g|)}{\tau(g)} m.
\]
In particular it is a product of at most $m+1$ $g$-packets of total length at least $n|g|-\tfrac{|g|\cdot(1+|g|)}{\tau(g)} m$.
\end{lemma}

\begin{proof}
Notice that $w=g^n$ is reduced.
Let us first consider any sequence $u'$ of $m'$ symbols in $w$ such that the reduced form of $w-u'$ is obtained by the removal of a single interval $I$ containing $u'$.
If $|I|$ is not a multiple of $|g|$, choose some interval $J$ in $w$ containing $I$ such that $|J|=k \cdot |g|$ and $|J|-|I|<|g|$.
By removing $u'$ from $J$ and following the same reduction procedure in $w - u'$ that gives rise to the reduced word $w - I$, we see that the reduced form of $J-u'$ is $J -I$.
Hence, $u'$ together with the symbols in $J-I$ form a cancellation sequence in $J$.
Since $J$ is a cyclic permutation of $g^k$, the definition of the cancellation norm on $G$ implies that
\[
\|g^k\| =\|J\| \leq |u'|+|J-I| \leq m'+|g|.
\]
By the definition of the stable length
\[
\tau(g) \leq \tfrac{\| g^k\|}{k} \leq \tfrac{m'+|g|}{k}.
\]
It follows that 
\[
|I| \leq |J|= k |g| \leq \tfrac{|g|}{\tau(g)} (m'+|g|)
\]
Let us now consider a sequence $u$ of $m$ symbols in $g^n$.
The reduction of $g^n-u$ is obtained by the removal of disjoint intervals $I_1,\dots, I_k$ from $g^n$ whose union contains $u$, hence $k \leq |u| =m$.
Then $u$ is the disjoint union of the sequences $u_i=u \cap I_i$ where $1 \leq i \leq k$, and set $m_i=|u_i|$.
By following the reduction process that yields $w-\cup_i I_i$ from $w-u$ it is clear that for every $1 \leq i \leq k$ the reduced form of  $w-u_i$ is $g^n-I_i$. 
Therefore,  $|I_i| \leq \tfrac{|g|}{\tau(g)} (m_i+|g|)$.
The total length of the intervals $I_i$ is estimated by 
\[
\sum_{i=1}^k |I_i| 
\leq \tfrac{|g|}{\tau(g)} \sum_{i=1}^k (m_i+|g|) 
=
\tfrac{|g|}{\tau(g)} (m+ k|g|) 
\leq \tfrac{|g|}{\tau(g)}(m+m|g|) 
= 
\tfrac{|g|(1+|g|)}{\tau(g)}m.
\]
This completes the proof.
\end{proof}

\begin{lemma}\label{L:collision of partitions}
Fix some numbers $\ell >0$ and $N>0$.
Suppose that $\C P_1=\{I_1,\dots,I_{m_1}\}$ and $\C P_2=\{J_1,\dots,J_{m_2}\}$ are partitions of $\{1,\dots,N\}$ into non-empty intervals where $N \geq \ell(m_1+m_2)$.
Then there are $I_i \in \C P_1$ and $J_j \in \C P_2$ which intersect at an interval of length at least $\ell$.
\end{lemma}

\begin{proof}
We will assume that the intervals in $\C P_1$ (resp. $\C P_2$) are ordered in an increasing manner, i.e., $1 \in I_1$ and $I_1 I_2 \cdots I_k$ is an interval for every $1 \leq k \leq m_1$.
Use (strong) induction on $N$.
We may assume without loss of generality that $|I_{m_1}| \leq |J_{m_2}|$ namely that $I_{m_1} \subseteq J_{m_2}$.
Choose $1 \leq k \leq m_1$ maximal with the property that
\[
I_{m_1-k+1} \cdots I_{m_1} \subseteq J_{m_2}.
\]
Set $K=I_{m_1-k+1} \cdots I_{m_1}$ for short.

If the length of one of $I_{m_1-k+1}, \dots ,I_{m_1}$ is at least $\ell$ then we are done.
So we assume that the length of all these $k$ intervals is $<\ell$.

Suppose that $|J_{m_2} \setminus K| \geq \ell$.
Then $K$ is a proper interval in $\{1,\dots,N\}$ and in particular $k \leq m_1-1$.
The maximality of $k$ implies that $I_{m_1-k} \supsetneq J_{m_2} \setminus K$ and in particular $J_{m_2}$ and $I_{m_1-k}$ intersect at an interval of length $\geq \ell$, and we are done again.

So we assume that $|J_{m_2} \setminus K| < \ell$.
This implies that $|J_{m_2}| < (k+1) \ell$.
Since $m_1 \leq k$ and $m_2 \geq 1$,
\[
|J_{m_2}| < (k+1)\ell \leq (m_1+m_2)\ell \leq N.
\]
So $J_{m_2}$ is a proper interval in $\{1,\dots,N\}$, and hence so is $K$.
This implies that $m_2 \geq 2$, that $k \leq m_1-1$.
The maximality of $k$ implies that $I_{m_1-k} \setminus J_{m_2} \neq \emptyset$.
Set $N'=N-|J_{m_2}|$.
Clearly, $N'<<N$.
Set
\begin{align*}
& \C P_1' =\{ I_1, \dots,I_{m_1-k-1}, I_{m_1-k} \setminus J_{m_2} \} \\
& \C P_2' =\{J_1,\dots,J_{m_2-1}\}.
\end{align*}
Both are partitions of $\{1,\dots,N'\}$ consisting of $m_1'=m_1-k$ and $m_2'=m_2-1$ intervals.
Also
\[
N'=N-|J_{m_2}| > 
N - (k+1)\ell \geq 
(m_1+m_2)\ell -(k+1) \ell =
(m_1'+m_2') \ell.
\]
By the induction hypothesis there are $I' \in \C P_1'$ and $J' \in \C P_2'$ which intersect at an interval of length $\geq \ell$.
Hence one of the intervals $I_1,\dots,I_{m_1-k}$ and $J_1,\dots,J_{m_2-1}$ intersect at an interval of length $\geq \ell$.
This completes the induction step.
\end{proof}

\begin{lemma}\label{L:non triviality product Theta}
Consider $\theta_1,\dots,\theta_p \in \Theta$  where $p \geq 1$.
Assume that $\theta_i \neq \theta_{i+1}$ for all $1 \leq i <p$.
Set $\ell=\OP{lcm}(|\theta_1|,\dots,|\theta_p|)$, and for every $1 \leq i \leq p$ set
\[
\kappa_i = \frac{|\theta_i|}{9\ell+|\theta_i| \cdot \tfrac{1+|\theta_i|}{\tau(\theta_i)}}.
\]
Let $n_1,\dots,n_p$ be integers such that $|n_i| >\tfrac{6 \ell}{|\theta_i|}$ for all $1 \leq i \leq p$.
Then
\[
\| \theta_1^{n_1} \cdots \theta_p^{n_p} \|_G >\min_{1 \leq i \leq p} |n_i| \cdot \kappa_i.
\]
\end{lemma}

\begin{proof}
Let $g$ be the (not necessarily reduced) word $\theta_1^{n_1} \cdots \theta_p^{n_p}$.
Let $u$ be a sequence in $g$ of length 
\[
m \leq \min_{1 \leq i \leq p} |n_i| \cdot \kappa_i.
\] 
We will show that $g - u$ is not the trivial element and by this complete the proof since the norm on $G$ coincides with the cancellation norm.

Partition the sequence $u$ to subsequences $u_i \subseteq \theta_i^{n_i}$ of length $m_i$ each ($1 \leq i \leq p$).
Thus, $m=\sum_i m_i$.
By Lemma \ref{L:small intervals reduction} the removal of $u_i$ from $\theta_i^{n_i}$ yields, after reduction, a reduced word $w_i$ which is a product of at most $m_i+1$ $\theta_i$-packets of total length
\begin{align*}
|w_i| &\geq 
|n_i| \cdot |\theta_i| -m_i \cdot |\theta_i| \cdot \tfrac{1+|\theta_i|}{\tau(\theta_i)} 
\\
& \geq
|n_i| \cdot |\theta_i| - m \cdot |\theta_i| \cdot \tfrac{1+|\theta_i|}{\tau(\theta_i)} && \text{(since $m_i \leq m$)}
\\
& \geq
m \tfrac{|\theta_i|}{\kappa_i} - m \cdot |\theta_i| \cdot \tfrac{1+|\theta_i|}{\tau(\theta_i)} && \text{(since $m \leq \kappa_i |n_i|$)} \\
& =
9\ell m && \text{(by definition of $\kappa_i$)}.
\end{align*}
Then $g - u$ is equivalent to $w_1 \cdots w_p$. 
For any $1 \leq i <p$ we will now examine the reduction of $w_i w_{i+1}$ at the juncture.
This reduction boils down to a word $v_i$ such that $w_i=w_i'v_i$ and $w_{i+1}=v_i^{-1} w'_{i+1}$ and $w_i'w'_{i+1}$ is reduced.
We claim that for every $1 \leq i <p$
\begin{equation}\label{E:bound on length of v}
|v_i| \leq \min\left\{ \tfrac{|w_i|}{3},\tfrac{|w_{i+1}|}{3} \right\}+\ell.
\end{equation}
Assume false, i.e., $|v_i|>\min\left\{ \tfrac{|w_i|}{3},\tfrac{|w_{i+1}|}{3} \right\}+\ell$.
Then $|v_i|> (3m+1)\ell$.

If $m=0$ then $m_i=0$ for all $i$ and then $u$ is empty so $w_i=\theta_i^{n_i}$ for all $i$.
Then $v_i$ is a $\theta_i$-packet and $v_{i}^{-1}$ is a $\theta_{i+1}$ packet of length $>(3m+1)\ell \geq \ell \geq \OP{lcm}(|\theta_i|,|\theta_{i+1}|)$.
Lemma \ref{L:equal packets equal words}  implies that $\theta_i=\theta_{i+1}$ which is a contradiction.
This proves \eqref{E:bound on length of v} if $m=0$.
So we assume that $m\geq 1$.

Given $1 \leq i <p$, recall that $w_i$ is a product of at most $m_i+1$ $\theta_i$-packets and similarly $w_{i+1}$ is a product of at most $m_{i+1}+1$ $\theta_{i+1}$-packets.
Therefore $v_i$ is a product of at most $m_i+1$ $\theta_i$ packets and $v_i^{-1}$ is a product of at most $m_{i+1}$ $\theta_{i+1}$-packets.
If we set $N=|v_i|$, then the partition of $w_i$ into the $\theta_i$-packets gives a partition $\C P_1$ of $\{1,\dots,N\}$ (the letters in $v_i$) into at most $m_i+1$ intervals.
Similarly, the partition of $w_{i+1}$ into the $\theta_{i+1}$-packets gives a partition $\C P_2$ of $\{1,\dots,N\}$ 
into at most $m_{i+1}+1$ intervals.
Then
\[
N = 
|v_i| > 
(3m+1)\ell \geq 
(m+2)\ell =
((m_i+1)+(m_{i+1}+1))\ell.
\]
Lemma \ref{L:collision of partitions} implies that in the process of reduction of $w_iw_{i+1}$ there is a collision of a $\theta_i$-packet in $w_i$ with a $\theta_{i+1}$-packet in $w_{i+1}$ at an interval $v_i'$ in $v_i$ of length at least $\ell \geq \OP{lcm}(|\theta_i|,|\theta_{i+1}|)$.
Lemma \ref{L:equal packets equal words} implies that $\theta_i=\theta_{i+1}$ which is a contradiction. 
This proves \eqref{E:bound on length of v} in the case $m \geq 1$.

Let us set $v_0$ and $v_p$ to be the empty words, which we regard as a prefix of $w_1$ and a suffix of $w_p$.
Clearly \eqref{E:bound on length of v} applies to them as well, where we agree that $w_0$ and $w_{p+1}$ are empty.
Then given $1 \leq i \leq p$
\[
|w_i|-|v_{i-1}|-|v_i| \geq
|w_i| - 2 \tfrac{|w_i|}{3} - 2\ell =
\tfrac{|w_i|}{3} - 2\ell.
\]
If $m=0$ then in particular $m_i=0$, in which case $w_i=\theta_i^{n_i}$, so
\[
\tfrac{|w_i|}{3} - 2\ell  = \tfrac{|n_i| \cdot |\theta_i|}{3} - 2\ell > \tfrac{6\ell}{3}-2\ell =0.
\]
If $m>0$ then 
\[
\tfrac{|w_i|}{3} - 2\ell \geq \tfrac{9\ell m}{3}-2\ell \geq \ell >0.
\]
This proves that $v_{i-1}^{-1}$ and $v_i$ are disjoint intervals in $w_i$ for all $1 \leq i \leq p$, so $w_i=v_{i-1}^{-1}\widetilde{w_i}v_i$ for non empty subword $\widetilde{w_i}$.
It follows that the reductions at the junctures of $w_1 \cdots w_p$ do not affect each other so the reduced form of $w_1\cdots w_p$ is $\widetilde{w_1}\cdots \widetilde{w_p}$ which is non-trivial since $\widetilde{w_i}$ are not empty and reduced, and there may be no reductions at the junctures.
This completes the proof.
\end{proof}

\begin{proof}[Proof of Theorem \ref{theorem:F(Theta) to tC(Free)}]
Let $\alpha=\overline{\theta_1}(r_1) \cdots \overline{\theta_p}(r_p)$ be an non-trivial element of $\Free_\RR(\Theta)$ represented in reduced form, i.e., $\theta_1,\dots,\theta_p \in \Theta$ where $p>0$, and $r_i \neq 0$ and $\theta_i \neq \theta_{i+1}$.
We will show that $\|\hat{\pi}(\iota(\alpha))\|_{\tC(G)}>0$.

Set $[g]=\hat{\pi}(\iota(\alpha))$.
By definition it is represented by the function 
\[
g(t) = \theta_1^{\Myfloor{r_1t}}  \cdots \theta_p^{\Myfloor{r_pt}}, \qquad (g \colon (0,\infty) \to G).
\]
Set $r=\min\{|r_1|,\dots,|r_p|\}$.
Clearly $r>0$.
Set $\ell=\OP{lcm}(|\theta_1|,\dots,|\theta_p|)$.
Clearly, $|\Myfloor{r_it}| \cdot |\theta_i| > 6\ell$ for all sufficiently large $t$. 
Let $\kappa_i$ be defined as in Lemma \ref{L:non triviality product Theta}.
Clearly $\kappa_i >0$ for all $i$, so $\kappa := \min\{ \kappa_1,\dots,\kappa_p\} >0$.
It follows from that lemma that 
\[
\limsup_{t\to \infty} \tfrac{1}{t} \|g(t)\| 
\geq \limsup_{t \to \infty} \min_{1 \leq i \leq p} \left|\tfrac{\Myfloor{r_it}}{t}\right| \cdot \kappa_i 
\geq \limsup_{t \to \infty} \tfrac{| \, \Myfloor{rt}\,|}{t} \kappa = |r| \kappa >0.
\]
Now apply Lemma \ref{L:limsup<norm} to deduce that $[g]$ is not the trivial element in $\tC(G)$.
Therefore $\OP{Ker}(\hat{\pi} \circ \iota|_{\Free_\RR(\Theta)})$ is trivial.

Set $\xi_\Theta=\hat{\pi} \circ \iota|_{\Free_\RR(\Theta)}$.
By Proposition \ref{P:pure roots} $\Theta \subseteq G$ satisfies the conditions of Proposition \ref{prop:quotient of tau}.
In its notation $\xi_\Theta$ is the composition $\hat{\pi} \circ \iota \circ \OP{id}^\tau_\mu \circ \incl$ and it follows (since $\tC(G)$ is complete) that $\xi_\Theta$ is Lipschitz with constant $1$, has dense image and $\widehat{\xi_\Theta} =\hat{\xi}_\Theta$ is a metric quotient homomorphism.
If $\alpha=\overline{\theta}(r) \in \Free_\RR(\Theta)$ is a generator then $\|\xi_\Theta(\alpha)\| =\| \hat{\pi}(\iota(\alpha))\| = |r| \cdot \tau(\theta) = \|\overline{\theta}(r)\|$ by Lemma \ref{L:pihat powers} and Corollary \ref{cor:Cnorm image of generators in FRG}.
\end{proof}

\appendix

\section{Proofs for Section \ref{sec:conjugation-invariant norms}}\label{app:cin proofs}

\begin{proof}[Proof of Lemma \ref{L:conj invariant consequences}]
Induction on $n$.
Set $g=g_1\dots g_{n}$ and $g'=g_1\dots g_{n-1}$ and $h=h_1\dots h_{n}$ and $h'=h_1\dots h_{n-1}$.
Then $\| g^{-1}h\| = \|g_n^{-1} g'{}^{-1} h' h_n\| \leq \|g_n^{-1} h_n\| + \|h_n^{-1} g'{}^{-1} h' h_n\| = \|g_n^{-1} h_n\| + \| g'{}^{-1} h' \| \leq \sum_{i=1}^n \|g_i^{-1} h_i\|$.
\end{proof}

\begin{proof}[Proof of Lemma \ref{L:quotient pseudonorm}]
Surjectivity of $\pi$ implies that $\| \ \|_H$ is well defined.
The symmetry $\|h\|_H=\|h^{-1}\|_H$ follows from the symmetry of $\| \ \|_G$ and since $\pi^{-1}(h^{-1})=(\pi^{-1}(h))^{-1}$.
Consider $h_1,h_2 \in H$.
Given $\epsilon>0$, choose some $g_1 \in \pi^{-1}(h_1)$ such that $\|g_1\| < \|h_1\|_H +\epsilon$.
Since $\pi^{-1}(h_1h_2) = g_1 \cdot \pi^{-1}(h_2)$, we get that 
\begin{align*}
\| h_1 h_2\|_{H} &= 
\inf \left\{ \|g\|_G : g \in \pi^{-1}(h_1h_2)\right\}\\
&=
\inf \left\{ \| g_1 x\|_G :  x \in \pi^{-1}(h_2)\right\} \\
&\leq
\|g_1\|_G + \inf \left\{ \| x\|_G :  x \in \pi^{-1}(h_2)\right\} \\
&<
\|h_1\|_H + \|h_2\|_H + \epsilon.
\end{align*}
Since $\epsilon>0$ is arbitrary the triangle inequality follows.

The conjugation invariance of $\| \ \|_H$ follows from that of $\| \ \|_G$ similarly by observing that $\pi^{-1}(hxh^{-1})=g \pi^{-1}(x) g^{-1}$ for any choice of $g \in \pi^{-1}(h)$.
By definition of the norm, $\| \pi(g)\|_H \leq \| g\|_G$ for any $g \in G$ so $\pi$ is Lipschitz with constant $1$.
\end{proof}

\begin{proof}[Proof of Proposition \ref{P:quotient complete}]
Let $h_1,h_2,\dots$ be a Cauchy sequence in $H$.
Since $(h_n)$ is convergent if and only if  it contains a convergent subsequence, by passage to a subsequence if necessary we may assume that $\|h_k^{-1}h_n\|_H < \tfrac{1}{2^k}$ for every $k \leq n$.
For every $k \geq 1$ denote $\delta_k=h_k^{-1}h_{k+1}$.
Thus, $\|\delta_k\|_H < \tfrac{1}{2^k}$.

Lift $h_1$ arbitrarily to some $g_1 \in G$.
By construction of $\| \ \|_H$, every $\delta_k$ lifts to some $\Delta_k \in G$ such that $\|\Delta_k\|_G < \tfrac{1}{2^k}$.
For every $k \geq 1$ set $g_k=g_1 \cdot \Delta_1 \cdots \Delta_{k-1}$.
Since $\pi$ is a homomorphism 
\[
\pi(g_k)=\pi(g_1) \pi(\Delta_1) \cdots \pi(\Delta_{k-1}) = h_1 \delta_1 \cdots \delta_{k-1} = h_k
\]
for all $k \geq 1$.
Consider some $m \geq 1$.
For any $k \geq m$
\[
\| g_m^{-1}g_k\|_G = \|\Delta_m \cdots \Delta_{k-1}\|_G < \sum_{i=m}^{k-1} \tfrac{1}{2^i} < \tfrac{1}{2^{m-1}} 
\]
which is arbitrarily small for sufficiently large $m$.
It follows that $(g_n)$ is a Cauchy sequence in $G$, and by hypothesis it converges to some $g \in G$.
Since $\pi$ is Lipschitz, the sequence $h_n=\pi(g_n)$ converges to $\pi(g)$.
\end{proof}

\begin{proof}[Proof of Proposition \ref{P:quotient Lipschitz maps}]
We denote the norms by $\| \ \|_G$, $\| \ \|_H$, $\| \ \|_{G'}$ and $\| \ \|_{H'}$. 
Let $C>0$ be the Lipschitz constant of $\varphi$.
Given $g' \in G'$, for any $\epsilon>0$ choose $g \in G$ such that $\pi_G(g)=g'$ and $\|g\|_G<\|g'\|_{G'}+\tfrac{\epsilon}{C}$.
Since $\pi_H$ is Lipschitz with constant $1$,
\begin{multline*}
\| \varphi'(g')\|_{H'} = \| \varphi'(\pi_G(g)) \|_{H'} = \| \pi_H(\varphi(g))\|_{H'} 
 \\
\leq 
\| \varphi(g)\|_H \leq C \| g\|_G < C\|g'\|_{G'}+\epsilon.
\end{multline*}
This completes the proof since $\epsilon$ is arbitrary.
\end{proof}

\begin{proof}[Proof of Proposition \ref{P:norm-mu universally Lipschitz}]
Choose $g \in G$.
For any presentation $g=u_1\cdots u_n$, where each $u_i$ is conjugate to some $a_i^{\pm 1}$ we get that
\[
\|f(g)\|_H \leq \sum_i \|f(u_i)\|_H = \sum_i \|f(a_i^{\pm 1})\|_H  
= \sum_i \|f(a_i)\|_H \leq C \sum_i \mu(a_i).
\]
This shows that $\|f(g)\|_H \leq C\|g\|_{G;\mu}$.
\end{proof}

\begin{proof}[Proof of Proposition \ref{prop:finite generating sets bi lipschitz equivalence}]
Since the roles of $A$ and $A'$ are interchangeable, it suffices to show that $\OP{id} \colon (G,\|\ \|_A) \to (G,\| \ |_{A'})$ is  Lipschitz.
The norms take integer values only so we may set $C= \max \{ \| a\|_{A'} : a \in A\}$.
Then $\|\OP{id}(a)\|_{A'} \leq C = C \cdot \|a\|_A$.
The result follows from Proposition \ref{P:norm-mu universally Lipschitz}.
\end{proof}

\begin{proof}[Proof of Proposition \ref{prop:quotient of canonical word norms}]
Let $\| \ \|_H$ denote the quotient pseudonorm on $H$ defined in Lemma \ref{L:quotient pseudonorm}.
Choose some $h \in H$.
We will show that $\| h \|_H=\| h\|_B$ and by this complete the proof.
By the description of $\| \ \|_A$ and $\| \ \|_B$ in Definition \ref{defn:standard word norm} and since $B=\pi(A)$ it follows that $\|h\|_B \leq \|g\|_A$ for any $g \in \pi^{-1}(h)$, hence $\|h\|_B \leq \|h\|_H$.
Conversely, by the same description it is clear that if $\|h\|_B=n$ then there exists $g \in \pi^{-1}(h)$ such that $\| g\|_A \leq n$ and therefore $\|h\|_H \leq n = \|h\|_B$.
\end{proof}

\section{Proofs of the results in Section \ref{sec:FSA}}\label{app:proof of FSA}

\begin{lemma}\label{lem:can WAS triangle}
Let $\B w_0, \B w_1, \B w_2$ be words in the alphabet $A(S)$.
Then
\[
\| \B w_1 \B w_0 \B w_2\|_{\can} \leq \| \B w_1 \B w_2\|_{\can} + \| \B w_0\|_{\can}
\]
\end{lemma}

\begin{proof} 
Any cancellation sequences $\B u_0$ and $\B u_1 \B u_2$ in $\B w_0$ and $\B w_1 \B w_2$ give rise to a cancellation sequence $\B u = \B u_1 \B u_0 \B u_2$ in $\B w_1 \B w_0 \B w_2$.
The result follows from the definition since $\mu_S(\B u_1\B u_2) + \mu_S(\B u_0)= \mu_S(\B u) \geq \| \B w_1 \B w_0 \B w_2 \|_{\can}$.
\end{proof}

\begin{proof}[Proof of Lemma \ref{lem:minimal cancellation seq in syllable}]
Let $\B v =a(t_1) \cdots a(t_n)$ where $t_i \in \myinterval{0,s_i}_S$  be an arbitrary cancellation sequence in $\B w$.
By definition, $\B w - \B v = \prod_i a(s_i-t_i)$ represents the trivial element, so by Lemma \ref{lem:syllable basic 1} $\sum_i s_i-t_i=0$, hence $\sum_i s_i = \sum_i t_i$.
Then by definition
\[
\mu_S(\B v) = \sum_i |t_i| \mu(a) = \left(\sum_i |t_i|\right) \mu(a) 
\geq \left|\sum_i t_i\right| \cdot \mu(a) = \left|\sum_i s_i\right| \cdot \mu(a).
\]
Since $\B v$ was arbitrary, this shows that
\[
\| \B w\|_{\can} \geq |s| \cdot \mu(a).
\]
We now construct a cancellation sequence $\B u=a(t_1) \cdots a(t_n)$ in $\B w$ such that $\mu_S(\B u)=|s| \cdot \mu(a)$ and by this complete the proof.
If $s =0$ then $\B w$ represents the trivial word so choosing $t_i=0$ for all $i$ does the trick.
Suppose that $s > 0$.
Set $I_+=\{i : s_i >0\}$. 
Clearly
\[
0 < s = \sum_i s_i   \leq \sum_{i \in I_+} s_i.
\]
It is then clear that we can choose $0 \leq t_i \leq s_i$ for all $i \in I_+$ such that $t_i \in S$ and $\sum_i t_i =s$.
We choose $t_i=0$ if $i \notin I_+$.
Then
\[
\B w - \B u = \prod_{i \in I_+} a(s_i-t_i) \cdot \prod_{i \notin I_+}  a(s_i)
\]
represents the trivial element since $\sum_{i \in I_+} (s_i-t_i) + \sum_{i \notin I_+} s_i = 0$ and also
\[
\mu_S(\B u) = \sum_{i \in I_+} |t_i| \cdot \mu(a) =  \left(\sum_{i \in I_+} t_i\right) \cdot \mu(a) = s \cdot \mu(a) = |s| \cdot \mu(a).
\]
Hence $\|\B w\|_{\can} = |s| \cdot \mu(a) = \mu_S(\B u)$, as needed.
The case $s<0$ is similar.
\end{proof}

\begin{lemma}\label{lem:cancel entire syllable}
Let $\B w=\B w_1 \cdots \B w_n$, where $n>0$, be a word in $W(A(S))$ written in syllabic form.
If $\B u$ is a cancellation sequence in $\B w$ with corresponding syllabic form $\B u_1 \cdots \B u_n$, then there exists $1 \leq k \leq n$ such that $\B w_k - \B u_k$ is trivial in $\Free_S(A)$.
Moreover, by the removal of the $k$th syllable from both $\B w$ and $\B u$,
\[
\B u' = \B u_1 \cdots \widehat{\B u_k} \cdots \B u_n
\]
is a cancellation sequence in $\B w' = \B w_1 \cdots \widehat{\B w_k} \cdots \B w_n$.
\end{lemma}

\begin{proof}
Since $\B w_i$ are syllables of $A$-type $a_i$, by Lemma \ref{lem:syllable basic 1} there exist $r_i \in S$ such that in $\Free_S(A)$
\[
\B w - \B u = (\B w_1-\B u_1) \cdots (\B w_n- \B u_n) = a_i(r_i) \cdots a_n(r_n).
\]
Since $a_i \neq a_{i+1}$ for all $i$, if $r_i \neq 0$ for all $i$ then this is a non-trivial reduced word, contradiction to the assumption that $\B u$ is a cancellation sequence.
So $r_k=0$ for some $k$, hence $\B w_k-\B u_k$ is trivial in $\Free_S(A)$.

Finally, since $\B w_k - \B u_k$ is trivial in $\Free_S(A)$, then $\B w - \B u$ and 
\[
(\B w_1-\B u_1) \cdots \widehat{(\B w_k- \B u_k)} \cdots (\B w_n- \B u_n)
\]
are equal in $\Free_S(A)$.
But $\B w - \B u$ is trivial,  so $\B u'$ is a cancellation sequence in $\B w'$.
\end{proof}

\begin{proof}[Proof of Proposition \ref{prop:can-WAS minimum cancellation sequence}]
Write $\B w = \B w_1 \cdots \B w_n$ in syllabic form. 
We prove the result by induction on $n$.
The case $n=0$ is trivial since $\B w$ is the empty word.
The case $n=1$ is covered by Lemma \ref{lem:minimal cancellation seq in syllable}.
So we assume that $n \geq 2$ and the result holds for words with less than $n$ syllables.
By definition of $\| \B w\|_{\can}$ 
there exists a sequence $\B u^{(m)}$ of cancellation sequences in $\B w$ such that
\[
\lim_{m \to \infty} \mu_S(\B u^{(m)}) = \| \B w\|_{\can}.
\]
We can write them in syllabic form
\[
\B u^{(m)} = \B u^{(m)}_1 \cdots \B u^{(m)}_n
\]
compatible with the syllable form of $\B w$.
By Lemma \ref{lem:cancel entire syllable} for every $m \geq 1$ there exists $1 \leq k_m \leq n$ such that $\B w_{k_m} - \B u^{(m)}_{k_m}$ is trivial.
By passage to a subsequence we may assume that there exists $1 \leq k \leq n$ such that $\B w_k - \B u^{(m)}_k$ is trivial for all $m \geq 1$.

We write $\hat{\B w} = \B w_1 \cdots \widehat{\B w_k} \cdots \B w_n$ and $\hat{\B u}^{(m)} = \B u_1^{(m)} \cdots \widehat{\B u_k^{(m)}} \cdots \B u_n^{(m)}$.
By Lemma \ref{lem:minimal cancellation seq in syllable} $\B w_k$ has a minimal cancellation sequence $\B v_k$.
Since $\hat{\B w}$ has less than $n$ syllables (it may have $n-1$ or $n-2$ or $n-3$ syllables), by the induction hypothesis it admits a minimal cancellation sequence $\hat{\B v}$.
Clearly, $\B v_k$ and $\hat{\B v}$ can be combined to form a cancellation sequence $\B v$ in $\B w$.
The minimality of $\B v_k$ and $\hat{\B v}$ and Lemma \ref{lem:can WAS triangle} implies that 
\[
\| \B w\|_{\can}
\leq 
\| \hat{\B  w}\|_{\can} +  \| \B w_k\|_{\can} 
=
\mu_S(\hat{\B v})+\mu_S(\B v_k)
\leq
\mu_S(\hat{\B u}^{(m)})+\mu_S(\B u_k^{(m)})
\xto{m \to \infty}
\| \B w \|_{\can}.
\]
Hence, the first inequality is an equality and therefore $\mu_S(\B v) = \mu_S(\hat{\B v})+\mu_S(\B v_k) = \| \B w \|_{\can}$.
This completes the proof.
\end{proof}

\begin{lemma}\label{lem:can norm WAS splitting} 
Let $\B w = \B w_1 \cdots \B w_n$ in $W(A(S))$ be in syllabic form. 
Then there exists some $1 \leq k \leq n$ such that 
\[
\| \B w \|_{\can} =  \| \B w_k\|_{\can} +  \| \B w_1 \cdots \widehat{\B w_k} \cdots \B w_n\|_{\can}
\]
\end{lemma}

\begin{proof}
By Proposition \ref{prop:can-WAS minimum cancellation sequence} there is a minimal cancellation sequence $\B u = \B u_1 \cdots \B u_n$ in $\B w$.
By Lemma \ref{lem:cancel entire syllable} there exists $k$ such that $\B w_k - \B u_k$ is trivial and $\B u' = \B u_1 \cdots \widehat{\B u_k} \cdots \B u_n$ is a cancellation sequence in $\B w' = \B w_1 \cdots \widehat{\B w_k} \cdots \B w_n$.
By Lemma \ref{lem:can WAS triangle}
\[
\| \B w \|_{\can}
\leq 
\| \B w' \|_{\can} + \| \B w_k \|_{\can}
\leq
\mu_S(\B u')+\mu_S(\B u_k)
= \mu_S(\B u)
= \| \B w\|_{\can}.
\]
The result follows.
\end{proof}

For a syllable $\B w = a(s_1) \cdots a(s_n)$ we write $\OP{tot}(\B w)=\sum_i s_i$.
It is clear that $\B w = a(s)$ in $\Free_S(A)$.

\begin{proof}[Proof of Proposition \ref{prop:can-norm invariant representatives}]
We need to check invariance of $\| \ \|_{\can}$ under the relations in \eqref{eqn:presentation of FSA} in the presentation of $\Free_S(A)$ in \ref{V:presentation FSA}.
Hence, we need to show that $\| \B w\|_{\can}=\| \B w'\|_{\can}$  in  the two cases where
\begin{enumerate}[label=(\roman*)]
\item\label{can-norm invariant:i}
$\B w = \B w_1 \B w_2$ and $\B w' = \B w_1 \cdot a(0) \cdot \B w_2$ for some words $\B w_1, \B w_2$ and $a \in A$.

\item\label{can-norm invariant:ii}
$\B w = \B w_1 \cdots \B w_n$ and $\B w' = \B w'_1 \cdots \B w'_n$ are in syllabic forms, the $A$-types of $\B w_i$ and $\B w'_i$ are the same and $\OP{tot}(\B w_i)=\OP{tot}(\B w'_i)$ for all $1 \leq i \leq n$.
\end{enumerate}
Case \ref{can-norm invariant:i}:
Observe that if $\B u = \B u_1 \B u_2$ is a cancellation sequence in $\B w$ where $\B u_1, \B u_2$ are sequences in $\B w_1,\B w_2$ then $\B u'=\B u_1 a(0) \B u_2$ is a cancellation sequence in $\B w'$ and $\mu(\B u)=\mu(\B u')$ since $\mu(a(0))=0 \cdot \mu(a)=0$.
Similarly, since $\myinterval{0,0}=\{0\}$, any cancellation sequence $\B u'$ in $\B w'$ has the form $\B u'=\B u_1 a(0) \B u_2$ for sequences $\B u_1, \B u_2$ in $\B w_1, \B w_2$, and moreover since $a(0)$ is trivial in $\Free_S(A)$, it follows that $\B u = \B u_1 \B u_2$ is a cancellation sequence in $\B w$ and $\mu(\B u')=\mu(\B u)$.
It follows from the definition of $\| \ \|_{\can}$ that $\|\B w\|_{\can} = \| \B w'\|_{\can}$.

Case \ref{can-norm invariant:ii}:
We will use induction on the number of syllables $n$, the case $n=0$ is trivial and $n=1$ follows from Lemma \ref{lem:minimal cancellation seq in syllable}.
By Lemma \ref{lem:can norm WAS splitting} there exists some $1 \leq k \leq n$ such that
\[
\| \B w\|_{\can} = \| \B w_k\|_{\can} + \| \hat{\B w} \|_{\can}
\]
where $\hat{\B w} = \B w_1 \cdots \widehat{\B w_k} \cdot \B w_n$.
Since $\hat{\B w}$ and $\hat{\B w}'$ have  less that $n$ syllables (they have $n-1$ or $n-2$ or $n-3$ syllables), and since they satisfy \ref{can-norm invariant:ii}, the induction hypothesis implies that $\| \hat{\B w}\|_{\can} = \| \hat{\B w}'\|_{\can}$.
Similarly, $\B w_k\|_{\can} = \| \B w_k'\|_{\can}$.
By Lemma \ref{lem:can WAS triangle}
\[
\| \B w'\|_{\can} 
\leq
\| \hat{\B w}'\|_{\can} +\| \B w'_k \|_{\can}
=
\| \hat{\B w}\|_{\can} +\| \B w_k \|_{\can}
=
\| \B w\|_{\can}.
\]
By symmetry $\|\B w\|_{\can} \leq \| \B w'\|_{\can}$, and equality holds.
\end{proof}

Given a group $G$ and $g \in G$ we will write $h=g^\bullet$ to denote that $h$ is a conjugate of $g$.
The proof of the next lemma is left to the reader.

\begin{lemma}\label{lem:bullets 1}
Let $g_1, \dots, g_n$ and $h_1, \dots, h_n$ be elements in a group $G$.
If $g_1 \cdots g_n =1$ then 
\[
g_1h_1g_2 h_2 \cdots g_n h_n = h_1^\bullet \cdots h_n^\bullet.
\]
\end{lemma}

\begin{proof}[Proof of Theorem \ref{thm:FSA cancellation=CIWN}]
Consider some $g \in \Free_S(A)$.
Represent it by some word
\[
\B w = a_1(s_1) \cdots a_n(s_n)
\]
in $W(A(S))$.
Let $\B u= a_1(t_1) \cdots a_n(t_n)$ be a minimal cancellation sequence (Proposition \ref{prop:can-WAS minimum cancellation sequence}).
By definition $\B w - \B u = \prod_i a_i(s_i-t_i)$ is trivial in $\Free_S(A)$, so Lemma \ref{lem:bullets 1} implies the following equality in $\Free_S(A)$.
\[
\B w = 
a_1(s_1-t_1) a_1(t_1) a_2(s_2-t_2) a_2(t_2) \cdots a_n(s_n-t_n) a_n(t_n) 
= 
a_1(t_1)^\bullet \cdots a_n(t_n)^\bullet.
\]
Since $a_i(t_i) \in A(S)$ are generators of $\Free_S(A)$, by the definition of the norm $\| \ \|_{\Free_S(A);\mu_S}$ and of $\| \B w\|_{\can}$
\[
\| g \|_{\Free_S(A);\mu_S} \leq \sum_i \mu_S(a_i(t_i)) = \mu_S(\B u) = \| \B w\|_{\can} = \| g\|_{\can}.
\]
For the reverse inequality choose an arbitrary $\varepsilon>0$.
By the definition of $\| g\|_{\Free_S(A);\mu}$ there are $a_1(s_1), \dots , a_n(s_n) \in A(S)$ and $h_1 ,\dots, h_n \in \Free_S(A)$ such that
\[
g= \prod_i a_i(s_i)^{h_i}
\]
and such that
\[
\sum_i \mu(a_i(s_i)) < \| g\|_{\Free_S(A);\mu} + \varepsilon.
\]
For every $1 \leq i \leq n$ choose $\B u_i \in W(A(S))$ which represent $h_i$.
Let $\B u_i^{-1}$ denote the obvious word which represents $h_i^{-1}$.
Then the word
\[
\B w = \B u_1^{-1} \cdot a_1(s_1) \cdot \B u_1 \cdot \B u_2^{-1} \cdot a_2(s_2) \cdot \B u_2 \cdots \B u_n^{-1} \cdot a_n(s_n) \cdot \B u_n 
\]
represents $\prod_i a_i(s_i)^{h_i}$, i.e., it represents $g$.
It is clear that the letters $a_i(s_i)$ in $\B w$ form a cancellation sequence.
It follows that
\[
\| g\|_{\can} = \| \B w\|_{\can} \leq \sum_i \mu_S(a_i(s_i)) < \|g\|_{\Free_S(A);\mu} + \varepsilon.
\]
Since $\varepsilon>0$ is arbitrary, $\| g\|_{\can} \leq \|g\|_{\Free_S(A);\mu}$ and hence equality holds.
\end{proof}

\begin{proof}[Proof of Proposition \ref{prop:mu>0 => norm}]
If $\mu(a)=0$ for some $a \in A$ then clearly $a(1) \in \Free_S(A)$ has $\| a(1)\|_{\can} = \mu(a)=0$, showing that $\| \ \|_{\Free_S(A);\mu}$ is not a norm.
Conversely, if $\B w=a_1(s_1) \cdots a_n(s_n)$ is a word which represents a non-trivial element in $\Free_S(A)$ then any cancellation sequence $\B u=a_1(t_1) \cdots a_n(t_n)$, in particular a minimal one, must have $t_k \neq 0$ for some $k$, whence $\| \B w\| = \mu(\B u)\geq |t_k| \cdot \mu(a_k)>0$.
\end{proof}

\begin{proof}[Proof of Proposition \ref{prop:restriction norms FSA<FRA}]
Consider some $g \in \Free_S(A)$.
Since $A(S) \subseteq A(\RR)$ and $\mu_S \colon A(S) \to [0,\infty)$ is the restriction of $\mu_\RR \colon A(\RR) \to [0,\infty)$, it follows from the definitions, see \ref{V:word norms}, that 
\[
\|g \|_{\Free_\RR(A);\mu_\RR} \leq \| g\|_{\Free_S(A);\mu_S}.
\]
To prove the reverse inequality we use the description as cancellation norms, Theorem \ref{thm:FSA cancellation=CIWN}.
We will show that
\[
\| \B w\|_{\can;S;\mu} \leq \| \B w\|_{\can;\RR;\mu} 
\]
for every $\B w \in W(A(S))$.
Write $\B w = \B w_1 \cdots \B w_n$ in syllabic form.
We prove the inequality by induction on $n$.
The case $n=0$ is trivial and $n=1$ is covered by Lemma \ref{lem:minimal cancellation seq in syllable}.
By Lemma \ref{lem:can norm WAS splitting} there exists $1 \leq k \leq n$ such that
\[
\| \B w\|_{\can;\RR;\mu} =  \| \B w_k\|_{\can;\RR;\mu} + \| \B w'\|_{\can;\RR;\mu}
\]
where $\B w' = \B w_1 \cdots \widehat{\B w_k} \cdots \B w_n$.
The induction hypothesis applies to $\B w_k$ and $\B w'$ and together with Lemma  \ref{lem:can WAS triangle}
\[
\| \B w\|_{\can;S;\mu} \leq \| \B w_k\|_{\can;S;\mu} + \| \B w'\|_{\can;S;\mu}
=
\| \B w_k\|_{\can;\RR;\mu} + \| \B w'\|_{\can;\RR;\mu}
=
\| \B w\|_{\can;\RR;\mu}.
\]
\end{proof}

\begin{proof}[Proof of Proposition \ref{P:geodesics in F_R(A)}]
We write $\| \ \|$ for the norm of $\Free_\RR(A)$.
Since the metric on $\Free_\RR(A)$ is invariant with respect to translations, it suffices to find for every $w \in \Free_\RR(A)$ a geodesic to the identity element.
Let $\B w=a_1(r_1) \cdots a_n(r_n)$ be a word in $W(A(\RR))$ representing $w$.
In the remainder of the proof we will freely regard words as the elements they represent in $\Free_\RR(S)$.

Let $\B u=a_1(s_1) \cdots a_n(s_n)$ be a minimal cancellation sequence in $\B w$, see Proposition \ref{prop:can-WAS minimum cancellation sequence}.
Define $\alpha \colon [0,1] \to F_\RR(A)$ by
\[
\alpha(t) = a_1(r_1-s_1t) \cdots a_n(r_n-s_nt).
\]
It is clear that $\alpha(0)=w$ and $\alpha(1)=1$ because $\B w - \B u$ represents the trivial element.
Consider some $0 \leq t' < t \leq 1$.
Then $\alpha(t)=\prod_i a_i(r_i-ts_i)$ and $\alpha(t')=\prod_i a_i(r_i-ts'_i)$.
By Lemmas \ref{L:conj invariant consequences} and \ref{lem:minimal cancellation seq in syllable}
\[
\|\alpha(t)^{-1}\alpha(t')\| \leq \sum_i \|a_i(t's_i-ts_i)\| = |t-t'| \sum_i \mu(a_i)|s_i| = |t-t'| \cdot \mu(\B u) = |t-t'| \cdot \|w\|.
\]
In particular, given $0 \leq t_1 \leq t_2 \leq 1$ we get
\begin{align*}
\|w\| &= \|\alpha(0) \alpha(1)^{-1}\| \\
&\leq \|\alpha(0) \alpha(t_1)^{-1}\| + \|\alpha(t_1)\alpha(t_2)^{-1}\|+\|\alpha(t_2)\alpha(1)^{-1}\| \\
&\leq t_1 \|w\| + (t_2-t_1)\|w\|+(1-t_2)\|w\|\\
& =\|w\|.
\end{align*}
Hence, equality holds everywhere and, in particular, 
\[
\|\alpha(t_1)\alpha(t_2)^{-1}\|=|t_2-t_1| \cdot \|w\|.
\]
It follows that $\alpha$ is a geodesic.
\end{proof}

\begin{proof}[Proof of Proposition \ref{prop:FSA dense FRA}]
Given $g \in \Free_\RR(A)$ and $\varepsilon>0$ choose  $\B w = a_1(r_1) \cdots a_n(r_n)$ in $W(A(\RR))$ which represents $g$.
Since $S$ is dense, there are $s_i \in S$ such that 
\[
\sum_i |s_i - r_i| \cdot \mu(a_i) < \varepsilon.
\]
Consider the word  $\B  w' = a_1(s_1) \cdots a_n(s_n)$ in $W(A(S))$.
It represents an element $g' \in \Free_S(A)$ and by Lemma \ref{L:conj invariant consequences}
\begin{multline*}
\| \B w^{-1} \B w' \|_{\Free_\RR(A);\mu} \leq \sum_i \| a_i(r_i)^{-1} a_i(s_i)\|_{\Free_\RR(A);\mu} 
\\
=
\sum_i \| a_i(s_i-r_i)\|_{\Free_\RR(A);\mu} 
=
\sum_i |s_i-r_i| \mu(a_i) < \varepsilon.
\end{multline*}
We have shown that there exists $g' \in \Free_S(A)$ at distance $<\varepsilon$ from $g$ and we are done.
\end{proof}

\begin{proof}[Proof of Proposition \ref{P:naturality FRA} ]
For any $a(s) \in A(S)$, 
 \begin{multline*}
\|f_*(a(s))\|_{\Free_S(B);\mu_B} =
\|f(a)(s)\|_{\Free_\RR(B);\mu_B} =
\mu_{B,S}(f(a)(s)) 
\\
= 
|s| \cdot \mu_B(f(a)) \leq 
C |s| \mu_A(a) = C \cdot \mu_{A,S}(a(s)).
\end{multline*}
The result follows from Proposition \ref{P:norm-mu universally Lipschitz}.
\end{proof}

\begin{proof}[Proof of Lemma \ref{L:psit FRA properties}]
\ref{L:psit FRA properties:multiplicative}
We will write $\| \ \|$ for the norm $\| \ \|_{\Free_\RR(A);\mu_\RR}$ on $\Free_\RR(A)$.
For any $a(r) \in A(\RR)$,
\[
\| \psi_t(a(r))\| =
\| a(tr)\| \leq 
|tr| \mu(a) =
|t| \mu_\RR(a(r)).
\] 
Proposition \ref{P:norm-mu universally Lipschitz} shows that $\psi_t$ is Lipschitz with constant $|t|$.

If $t=0$ then $\psi_t$ is trivial so $\| \psi_t(w)\|=t\|w\|$ trivially.
So we assume $t \neq 0$.
By Proposition \ref{P:psit monoidal properties} $\psi_{1/t} \circ \psi_t$ is the identity.
Since $\psi_{1/t}$ is Lipschitz with constant $\tfrac{1}{|t|}$, for any $w \in\Free_\RR(A)$
\[
\|w\| = \| \psi_{1/t}(\psi_t(w))\| \leq \tfrac{1}{|t|} \| \psi_t(w)\| \leq \tfrac{1}{|t|}|t| \cdot \|w\| = \| w\|.
\]
Therefore the inequalities are equalities and in particular $\|\psi_t(w)\|=|t| \cdot \|w\|$.

\smallskip
\noindent
\ref{L:psit FRA properties:difference}
Let $\B w=a_1(r_1) \cdots a_k(r_k)$ be a word representing $w$.
Set $C_w=\sum_i |r_i| \mu(a_i)$.
Notice that $\psi_t(w)$ is represented by $\prod_i a_i(tr_i)$ and $\psi_s(w)$ by $\prod_i a_i(sr_i)$.
Lemma \ref{L:conj invariant consequences} implies
\begin{multline*}
\| \psi_t(w)^{-1} \psi_s(w)\| \leq \sum_i \| a_i(tr_i)^{-1}a_i(sr_i)\| =
 \\
\sum_i \| a_i(sr_i-tr_i)\| \leq 
\sum_i |sr_1-tr_i| \mu(a_i) = C_w|s-t|.
\end{multline*}
\end{proof}

\bibliography{bibliography}
\bibliographystyle{plain}

\end{document}